\documentclass[11pt]{amsart}
\usepackage{iftex}
\ifPDFTeX
  \pdftrailerid{}
\fi
\usepackage[T1]{fontenc}
\usepackage{lmodern,microtype,booktabs}
\usepackage{amsmath,amssymb,amsthm,mathtools,mathrsfs,aliascnt}
\usepackage[margin=1in]{geometry}
\usepackage[hidelinks]{hyperref}
\usepackage[nameinlink,noabbrev,capitalize]{cleveref}
\allowdisplaybreaks
\newtheorem{theorem}{Theorem}[section]
\newtheorem{maintheorem}{Theorem}

\crefname{maintheorem}{Theorem}{Theorems}
\Crefname{maintheorem}{Theorem}{Theorems}
\newaliascnt{proposition}{theorem}
\newtheorem{proposition}[proposition]{Proposition}
\aliascntresetthe{proposition}
\newaliascnt{lemma}{theorem}
\newtheorem{lemma}[lemma]{Lemma}
\aliascntresetthe{lemma}
\newaliascnt{corollary}{theorem}
\newtheorem{corollary}[corollary]{Corollary}
\aliascntresetthe{corollary}
\theoremstyle{definition}
\newaliascnt{definition}{theorem}
\newtheorem{definition}[definition]{Definition}
\aliascntresetthe{definition}
\newaliascnt{example}{theorem}
\newtheorem{example}[example]{Example}
\aliascntresetthe{example}
\theoremstyle{remark}
\newaliascnt{remark}{theorem}
\newtheorem{remark}[remark]{Remark}
\aliascntresetthe{remark}
\newcommand{\Z}{\mathbb Z}
\newcommand{\Q}{\mathbb Q}
\newcommand{\N}{\mathbb N}
\newcommand{\Fp}{\mathbb F_p}
\newcommand{\Zp}{\mathbb Z_p}
\newcommand{\Qp}{\mathbb Q_p}
\newcommand{\C}{\mathcal C}
\newcommand{\HH}{\mathcal H}

\newcommand{\cO}{\mathcal O}
\newcommand{\I}{\mathfrak I}
\newcommand{\br}[1]{\{#1\}}
\newcommand{\ord}{\operatorname{ord}}
\newcommand{\im}{\operatorname{im}}

\newcommand{\length}{\operatorname{length}}
\newcommand{\Frac}{\operatorname{Frac}}
\newcommand{\PP}{\mathsf P}
\newcommand{\NN}{\mathcal N}
\newcommand{\Comp}{\operatorname{Comp}}
\newcommand{\Tay}{\operatorname{T}}
\numberwithin{equation}{section}

\newcommand{\diag}{\operatorname{diag}}
\newcommand{\GL}{\operatorname{GL}}
\newcommand{\Mat}{\operatorname{Mat}}
\newcommand{\Nm}{\operatorname{Nm}}
\newcommand{\Pic}{\operatorname{Pic}}
\newcommand{\Spec}{\operatorname{Spec}}

\newcommand{\qbin}[2]{\genfrac{[}{]}{0pt}{}{#1}{#2}_q}
\newcommand{\gaussbin}[3]{\genfrac{[}{]}{0pt}{}{#1}{#2}_{#3}^{\mathrm{Gauss}}}

\title[Quantum color Newton rings]{Quantum Color Newton Rings}
\author{Honghuai Fang}
\thanks{Institute for Theoretical Sciences, Westlake University,
600 Dunyu Road, Xihu District, Hangzhou, Zhejiang 310030, China.
Email: \texttt{fanghonghuai@westlake.edu.cn}.}
\date{}
\subjclass[2020]{Primary 13J10, 57K14; Secondary 11B65, 11S80, 14C22, 39A13}
\keywords{Newton ring, symmetric-color quantum invariant, Habiro ring,
Alexander polynomial, cyclotomic interpolation, Frobenius lift, normal jet, spectral module}
\hypersetup{pdftitle={Quantum Color Newton Rings}}
\begin{document}
\begin{abstract}
We study integral Newton rings attached to symmetric quantum colors.
We show that their boundary values at a deleted node are exactly the
elements of Habiro's cyclotomic completion, and we describe their
quantum neighborhoods at $q=1$ uniformly in the rank. The
symmetric-color $\mathfrak{sl}_n$ invariants of a knot are the values
of an element of such a ring. In its expansion at $q=1$, the
coefficient of $(q-1)^m$ has denominator dividing the $(n-1+2m)$th
power of the Alexander polynomial, with integral numerator, and this
exponent is sharp. For $5_2$ and $6_1$ at $q=i$, the period matrices
of the rank-two coefficient recurrences have spectral modules of exact
order three in the Picard group of an affine elliptic curve.
\end{abstract}
\maketitle
\setcounter{tocdepth}{2}
\hypersetup{bookmarksdepth=2}
\begingroup\small\tableofcontents\endgroup

\clearpage
\section{Introduction}\label{sec:introduction}

\subsection{Overview}\label{sec:intro-overview}

Newton interpolation expands functions on $0,1,2,\ldots$ in the
products $X$, $X(X-1)$, $X(X-1)(X-2)$, and so on. Symmetric quantum
colors lead to the same construction on the nodes
\[
 x_r=q^{n+2r}+q^{-n-2r},\qquad r\ge0.
\]
For a knot $K$, the integral cyclotomic expansion \cite{FZ26} expresses
its reduced quantum $\mathfrak{sl}_n$ invariants as the values of
\[
 F_{K,n}(X;q)=c_0(q)+c_1(q)(X-x_0)
       +c_2(q)(X-x_0)(X-x_1)+\cdots,
\]
where every $c_k(q)$ is an integral Laurent polynomial. At the
$r$th color point all terms after the first $r+1$ vanish, so the
infinite series has a Laurent polynomial value at each color.

Unlike classical Newton interpolation, the nodes here depend on $q$
and collide at roots of unity. At $q=1$ they all equal $2$. For the
figure-eight knot in rank two, every Newton coefficient is $1$, so
\begin{equation}\label{eq:introduction-figure-eight}
 F_{4_1,2}(2+u;1)=\sum_{k\ge0}u^k=\frac{1}{1-u}.
\end{equation}
All color values at $q=1$ equal $1$, but as a function of $u$ the
series has denominator $1-u=\Delta_{4_1}(z)$, where $u=z+z^{-1}-2$.
The function is therefore not determined by its color values at
$q=1$, and we study it near the points where the colors collide.

Such series form Newton rings, and we first determine their
structure. Consider a function defined on all nodes except the first.
Its value at the missing node
lies in Habiro's cyclotomic completion, and the function extends
across that node exactly when this value is a Laurent polynomial.
Every finite Taylor expansion with Habiro coefficients occurs at this
node, and the complete expansion determines the function. The
cyclotomic and $p$-adic valuations of the node differences give
optimal finite sampling rules and analytic interpolation in a
$p$-adic color. When the rank varies, the coefficients near $q=1$ are
ordinary integer-valued polynomials in the rank.

For knot elements, these structures give finite rank--color
reconstruction and a precise description near $q=1$. If
$\Delta_K(z)=\delta_K(z+z^{-1}-2)$ is the symmetric Alexander
polynomial normalized by $\Delta_K(1)=1$, then
$F_{K,n}(2+u;1)=\delta_K(u)^{-(n-1)}$. At quantum order $m$, the
denominator divides $\delta_K(u)^{n-1+2m}$ and the numerator is an
integral polynomial. The bound holds for every knot and is attained at
every rank and order. The interpolating function and its color
sequence have the same annihilating difference operators, and
Frobenius power substitutions give integral comparisons.

At other roots of unity, coefficient recurrences give explicit
rational functions and their local expansions. Finitely many colored
invariants, known to finite quantum precision, determine any
prescribed Taylor coefficient at a missing node, even though the
specialized nodes collide. We give closed formulas for the
specializations of $3_1$, $4_1$, $5_2$, and $6_1$ at every root of
unity. For $5_2$ and $6_1$, we find period matrices that are not
conjugate over the Laurent ring to the companion matrices of their
characteristic polynomials. In rank two at $q=i$, the obstruction has
exact order three in the Picard group of an affine elliptic curve:
$m$ copies of the period matrix are conjugate to $m$ copies of the
companion matrix exactly when $3\mid m$.

\subsection{Newton rings and their Habiro boundaries}\label{sec:intro-boundaries}

Put $R=\mathbb Z[q^{\pm1}]$. For an integer $n\ge0$, define
\[
 x_r^{(n)}=q^{n+2r}+q^{-n-2r},\qquad
 B_k^{(n)}(X)=\prod_{i=0}^{k-1}(X-x_i^{(n)}),\qquad r,k\ge0.
\]
The associated Newton ring is
\begin{equation}\label{eq:introduction-ring}
 \mathcal C_n=\varprojlim_k R[X]/(B_k^{(n)}).
\end{equation}
Every element has a unique expansion
$F=\sum_{k\ge0}c_k(q)B_k^{(n)}(X)$ with $c_k\in R$.
Evaluation at $x_r^{(n)}$ uses only the terms $k\le r$.
The resulting color-value sequences are characterized by the
integrality of their divided differences
(\cref{prop:compatible,prop:divided-difference}).

The node differences factor as
\begin{equation}\label{eq:introduction-node-difference}
 x_r^{(n)}-x_s^{(n)}=\{r-s\}\{n+r+s\},
 \qquad \{j\}=q^j-q^{-j}.
\end{equation}
At a root of unity, two nodes coincide exactly when one of these
quantum integers vanishes. Values on any infinite set of colors
determine an element of $\mathcal C_n$. The ring is a domain with unit
group $\{\pm q^a:a\in\mathbb Z\}$, and its reduction modulo every
rational prime is again a domain (\cref{thm:identity,thm:units-primes}).

Deleting the first node $a_n=q^n+q^{-n}$ leaves the node set of
$\mathcal C_{n+2}$ and gives an injection
$\rho_n:\mathcal C_n\hookrightarrow\mathcal C_{n+2}$.
A function on this tail can be evaluated at the missing node because
\[
 B_k^{(n+2)}(a_n)=(-1)^k\prod_{j=1}^k\{j\}\{n+j\}.
\]
These products define the same completion of $R$ as the Habiro
factorials. The evaluation therefore takes values in Habiro's
cyclotomic completion \cite{Habiro02,Habiro04},
\[
 \mathcal H=\varprojlim_m R/((q;q)_m),\qquad
 (q;q)_m=\prod_{j=1}^m(1-q^j).
\]
Write $\beta_n:\mathcal C_{n+2}\to\mathcal H$ for this boundary map
and $u=X-a_n$ for the normal coordinate.

\begin{maintheorem}[Habiro boundary]\label{thm:introduction-boundary}
For every $n\ge0$, the map $\beta_n$ is surjective and
\begin{equation}\label{eq:introduction-pullback}
 \mathcal C_n\simeq R\times_{\mathcal H}\mathcal C_{n+2}.
\end{equation}
For every $h\ge1$, normal expansion induces an isomorphism
\begin{equation}\label{eq:introduction-boundary-neighborhood}
 \mathcal C_{n+2}/u^h\mathcal C_{n+2}
       \simeq\mathcal H[u]/(u^h).
\end{equation}
The full Taylor map
$\mathcal C_{n+2}\hookrightarrow\mathcal H[[u]]$ is injective,
and $\mathcal H[[u]]$ is the $u$-adic completion of
$\mathcal C_{n+2}$.
\end{maintheorem}

\Cref{thm:introduction-boundary} combines
\cref{prop:boundary-exact,thm:pullback,thm:normal-completion}. For
$n\ge2$, the same normal description holds for $\mathcal C_n$ at every
earlier node $a_{n-2s}=q^{n-2s}+q^{-n+2s}$, with
$1\le s\le\lfloor n/2\rfloor$ (\cref{thm:all-boundaries}).

Let $\mathfrak I_m^{(n)}$ be the ideal of functions vanishing on the
first $m$ colors. There is an $R$-module isomorphism
\[
 \mathfrak I_1^{(n)}/(X-a_n)\mathcal C_n\simeq\mathcal H/R,
\]
and for general $m$ the quotient
$\mathfrak I_m^{(n)}/B_m^{(n)}\mathcal C_n$ has a filtration with $m$
successive quotients $\mathcal H/R$. Since $\mathcal H$ is uncountable,
no $\mathfrak I_m^{(n)}$ with $m\ge1$ is finitely generated, and the
tail extensions are birational but not finite
(\cref{thm:defect,prop:birational-tails}). The Newton topology is not
the adic topology of any ideal (\cref{prop:nonadic}).

For $d\ge1$, the substitutions $q\mapsto q^d$ and $X\mapsto\Theta_d(X)$, where
$\Theta_d(z+z^{-1})=z^d+z^{-d}$, define commuting Adams operations
$\psi_d$ on the rings and their boundaries. For every prime $p$,
$\psi_p$ reduces to the $p$th-power map modulo $p$. On the degree-$j$
piece of the normal associated graded at $X=2$, the action is
$h(q)u^j\mapsto d^{2j}h(q^d)u^j$
(\cref{prop:adams,prop:normal-adams}).

\subsection{Cyclotomic sampling and \texorpdfstring{$p$}{p}-adic colors}
\label{sec:intro-sampling}

The identity theorem uses infinitely many exact values. At finite
cyclotomic precision finitely many samples suffice, and we compute the
minimal number from the collision multiplicities of the nodes.

Fix $d\ge1$, put $\pi=\Phi_d(q)$ and $\ell=d/\gcd(d,2)$,
and work in the discrete valuation ring $\mathcal O=R_{(\pi)}$.
Group nonnegative colors according to the orbits of their residues
under
\[
 r\longmapsto-n-r\qquad\text{on }\mathbb Z/\ell\mathbb Z.
\]
The difference between two distinct nodes in the same orbit has
$\pi$-valuation one for a two-element orbit and two for a fixed
orbit. Let $f_{n,d}$ be the number of solutions to
$2r\equiv-n\pmod\ell$; there are $(\ell-f_{n,d})/2$
two-element orbits.

\begin{maintheorem}[Optimal cyclotomic sampling]\label{thm:introduction-sampling}
Let $h\ge1$. To determine all color values of an element of
$\mathcal C_n$ modulo $\pi^h\mathcal O$, it suffices to know,
to that precision, the values at any $h$ distinct nonnegative colors
whose residues lie in each two-element orbit and at any
$\lceil h/2\rceil$ distinct nonnegative colors whose residues lie in
each fixed orbit. Every universally sufficient finite sample set
contains at least these numbers in the respective orbits.
The minimum total is
\begin{equation}\label{eq:introduction-sampling}
 M_{n,d}(h)=\frac{\ell-f_{n,d}}2\,h
                +f_{n,d}\left\lceil\frac h2\right\rceil.
\end{equation}
\end{maintheorem}

A fixed orbit needs fewer samples because the map $z\mapsto z+z^{-1}$,
which sends $q^{n+2r}$ to $x_r^{(n)}$, is ramified at $z=\pm1$. At $q=1$, all colors belong to one fixed orbit,
so $\lceil h/2\rceil$ samples recover every color value modulo
$(q-1)^hR_{(q-1)}$. Both the reconstruction and its optimality follow
from the interpolation remainder (\cref{thm:sampling}).

Modulo $\pi^h$, the polynomial color-value sequences form a module to
which an orbit of collision order $e\in\{1,2\}$ contributes
$\bigoplus_{j=0}^{\lceil h/e\rceil-1}\mathcal O/(\pi^{h-ej})$
(\cref{thm:value-algebra}). The mixed cyclotomic--Newton completion
splits into formal disks indexed by the orbits, and the expansion in
any one disk determines the original Newton element
(\cref{prop:disks}).

The color variable also admits $p$-adic analytic interpolation. Let
$p$ be odd, $q_0\in1+p\mathbb Z_p$ with $q_0\ne1$, and
$a=v_p(q_0-1)$. For every $F\in\mathcal C_n$, the values
$F(x_r^{(n)};q_0)$ extend uniquely to a continuous function on
$\mathbb Z_p$. This extension is a restricted power series
\[
 f(t)=F(q_0^{n+2t}+q_0^{-n-2t};q_0)
             \in\mathbb Z_p\langle t\rangle,
 \qquad f(t)=f(-n-t).
\]
If $f_{<m}$ is the sum of the first $m$ Newton terms, then
\cref{thm:padic-color} gives the uniform bound
\begin{equation}\label{eq:introduction-padic-error}
 v_p\bigl(f(t)-f_{<m}(t)\bigr)
       \ge2am+2v_p(m!)\qquad(t\in\mathbb Z_p).
\end{equation}

Keeping the quantum parameter variable gives an injective
realization of $\mathcal C_n$ in $\mathbb Z_p\langle H,U\rangle$
under $q=1+p^aH$ and $X=2+p^{2a}U$, for every $a\ge1$
(\cref{thm:analytic-bidisc}).

\subsection{Rank-uniform rings and quantum neighborhoods}\label{sec:intro-rank-rings}

To vary the rank, we replace $q^n$ by an independent variable $A$.
The nodes become $Aq^{2r}+A^{-1}q^{-2r}$.
The coefficient ring consists of Laurent polynomials in $A$ over
$\mathbb Q(q)$ that take integral Laurent values at every rank:
\[
 \mathscr I_q=\{f\in\mathbb Q(q)[A^{\pm1}]:
              f(q^n)\in R\text{ for every }n\ge2\}.
\]
The rank-uniform Newton completion of \cite{FZ26} is
\begin{equation}\label{eq:introduction-universal-ring}
 \mathcal U=\varprojlim_k
 \mathscr I_q[X]\Big/
 \left(\prod_{i=0}^{k-1}(X-Aq^{2i}-A^{-1}q^{-2i})\right).
\end{equation}
Each specialization $A=q^n$ maps $\mathcal U$ onto $\mathcal C_n$,
and the family of all rank specializations is jointly injective.

Near $q=1$, the rank coefficient ring is described by ordinary
integer-valued polynomials. Put $\eta=q-1$, $u=X-2$, and
\[
 B=\operatorname{Int}(\mathbb Z)
   =\{P(N)\in\mathbb Q[N]:P(\mathbb Z)\subseteq\mathbb Z\}.
\]
The binomial polynomials $\binom{N-2}{j}$ form a $\mathbb Z$-basis
of $B$. They arise from the Gaussian basis of $\mathscr I_q$ at
$q=1$, in the setting of quantum integer-valued polynomials studied
by Harman and Hopkins \cite{HH17}. At every finite quantum order,
\[
 \mathscr I_q/\eta^h\mathscr I_q\simeq B[\eta]/(\eta^h)
 \qquad(h\ge1)
\]
(\cref{prop:rank-coefficient-jets}).

\begin{maintheorem}[Quantum neighborhoods]\label{thm:introduction-neighborhoods}
For $n\ge0$ and $h\ge1$, Taylor substitution gives compatible
isomorphisms
\begin{align}
 \mathcal C_n/\eta^h\mathcal C_n
     &\simeq(\mathbb Z[\eta]/(\eta^h))[[u]],
       \label{eq:introduction-fixed-neighborhood}\\
 \mathcal U/\eta^h\mathcal U
     &\simeq(B[\eta]/(\eta^h))[[u]].
       \label{eq:introduction-universal-neighborhood}
\end{align}
The universal substitution is $A=(1+\eta)^N$.
The maps into the full completions
$\mathbb Z[[u,\eta]]$ and $B[[u,\eta]]$ are injective.
The coefficient of $u^j\eta^m$ uses only Newton indices
$k\le j+\lfloor m/2\rfloor$.
\end{maintheorem}

Thus the fiber at $q=1$ is $\mathbb Z[[u]]$ at fixed rank and
$B[[u]]$ uniformly in rank, with specialization $N\mapsto n$.
The finite-index bound comes from the quadratic vanishing of each
node displacement $x_r^{(n)}-2$ at $q=1$
(\cref{thm:quantum-neighborhoods}). This is also why the analytic
realization uses the scale $X-2=p^{2a}U$.

\subsection{Knot elements and rank--color reconstruction}\label{sec:intro-knot-elements}

Let $K$ be a zero-framed knot, and let $J_r^{(n)}(K;q)$ be its
reduced quantum $\mathfrak{sl}_n$ invariant in
$\operatorname{Sym}^r(\mathbb C^n)$, normalized to be one on the
unknot. The integral cyclotomic expansion conjectured in
\cite{CLZ15} and proved in \cite[Theorem~1.1]{FZ26} gives
\begin{equation}\label{eq:introduction-knot-element}
 F_{K,n}(X;q)=\sum_{k\ge0}c_{K,n,k}(q)B_k^{(n)}(X)
           \in\mathcal C_n,\qquad
 F_{K,n}(x_r^{(n)};q)=J_r^{(n)}(K;q).
\end{equation}
The identity theorem makes this element unique. Connected sum
becomes multiplication, and mirroring replaces $q$ by $q^{-1}$.
The rank-uniform construction in \cite[Theorem~8.7]{FZ26} gives
$\mathbf F_K\in\mathcal U$ specializing to every $F_{K,n}$.

Integral interpolation imposes congruences on the quantum invariants:
\[
 J_r^{(n)}(K;q)-J_s^{(n)}(K;q)
       \in\{r-s\}\{n+r+s\}R,
\]
and all higher divided differences are integral. Agreement on any
infinite set of colors implies agreement on all colors. If the
values are Laurent units at infinitely many colors, then
$F_{K,n}=1$, so every $J_r^{(n)}(K;q)$ is one
(\cref{cor:knot-congruences}).

At an earlier node $a_{n-2s}$, with
$1\le s\le\lfloor n/2\rfloor$, write the Habiro-valued normal
expansion as
\begin{equation}\label{eq:introduction-normal-series}
 F_{K,n}(a_{n-2s}+u_s;q)
      =\sum_{j\ge0}\tau_{K,j}^{(n,s)}(q)u_s^j,
 \qquad \tau_{K,j}^{(n,s)}\in\mathcal H.
\end{equation}
As in \cref{thm:introduction-boundary}, any one of these expansions
determines $F_{K,n}$. If $\ell=\operatorname{ord}(\zeta^2)$, the first
boundary satisfies
\[
 \tau_{K,0}^{(n,1)}(\zeta)=J_{\ell-1}^{(n)}(K;\zeta).
\]
Thus the values $J_{\ell-1}^{(n)}(K;\zeta)$, whose color depends on
$\zeta$, are all evaluations of a single Habiro element
(\cref{prop:boundary-special-colors}).

The Newton expansion also gives the perturbative dependence on the
color. For $d\ge0$, there is a polynomial $P_{K,n,d}\in\mathbb Q[Y]$
such that
\[
 [h^d]J_r^{(n)}(K;e^h)=P_{K,n,d}\bigl((n+2r)^2\bigr),
 \qquad d!\,P_{K,n,d}\in\mathbb Z[Y],
 \qquad \deg P_{K,n,d}\le\lfloor d/2\rfloor.
\]
The dependence on $r$ is through $(n+2r)^2$ because the nodes are
invariant under $r\mapsto-n-r$, and for $d>0$ the polynomial
$P_{K,n,d}$ is divisible by $Y-n^2$ because $J_0^{(n)}=1$
(\cref{prop:all-rank-perturbative}).

For a fixed color, strong integrality of the reduced colored
HOMFLY--PT invariant \cite{Morton07,Zhu19} gives the dependence on
the rank:
\[
 J_r^{(n)}(K;q)=\mathsf P_{K,r}(Q^n,Q),\qquad
 \mathsf P_{K,r}(Z,Q)\in\mathbb Z[Z^{\pm1},Q^{\pm1}],
 \qquad Q=q^2.
\]
At $q=\zeta$, the rank values therefore depend only on
$n\bmod\operatorname{ord}(\zeta^2)$.

Combining rank and color interpolation gives the mixed congruence
\begin{equation}\label{eq:introduction-mixed-rank-color}
 J_r^{(n+2a)}(K;q)-J_{r+a}^{(n)}(K;q)
                    \in(q^{2a}-1)R\qquad(a\ge1).
\end{equation}
At a root $\zeta$ of $q^{2a}-1$, the difference vanishes to second
order exactly when its quotient by $q^{2a}-1$ vanishes at $\zeta$,
and \cref{thm:mixed-rank-color} computes this value. For a fixed
color, $D+1$ rank values reconstruct $\mathsf P_{K,r}$ when its
$Z$-exponents lie in a prescribed interval of width $D$; this number
is optimal for the space of Laurent polynomials supported in that
interval (\cref{thm:rank-reconstruction}).

At finite quantum precision, the Newton coefficients satisfy a degree
bound in the rank that does not depend on the knot.

\begin{maintheorem}[Finite-rank quantum reconstruction]\label{thm:introduction-rank-jets}
For $k\ge1$ and $m\ge0$, the function
\[
 n\longmapsto[\eta^m]c_{K,n,k}(1+\eta)
\]
is the restriction of a polynomial in $\operatorname{Int}(\mathbb Z)$
of degree at most $k+m$, vanishing at $n=1$.
For any $L,s\ge1$, the finite data
\begin{equation}\label{eq:introduction-rank-grid}
 J_r^{(n)}(K;1+\eta)\pmod{\eta^{s+2L}},
 \qquad 0\le r\le L,\quad 2\le n\le L+s,
\end{equation}
determine $c_{K,n,k}(1+\eta)\pmod{\eta^s}$ for every $n\ge2$
and $0\le k\le L$.
\end{maintheorem}

The additional $2L$ quantum orders compensate for the quadratic
vanishing of the color-interpolation factors. After inversion of the
finite color transform, the rank-degree bound and the known value
at $n=1$ recover the coefficient jets at every rank by polynomial
interpolation (\cref{thm:quantum-rank-degree,cor:finite-rank-color-jets}).

\subsection{Alexander denominators and the classical specialization}
\label{sec:intro-Alexander}

As in the figure-eight example \eqref{eq:introduction-figure-eight},
setting $q=1$ while keeping $u=X-2$ free produces the Alexander
polynomial as a denominator. For a general knot, the denominators at
higher quantum orders divide powers of the same polynomial.

Normalize $\Delta_K$ to be symmetric with $\Delta_K(1)=1$, and write
$\Delta_K(z)=\delta_K(z+z^{-1}-2)$, where
$\delta_K\in\mathbb Z[u]$. Consider the full-function expansion
and the Newton-coefficient generating functions:
\begin{align}
 F_{K,n}(2+u;1+\eta)
      &=\sum_{m\ge0}R_{K,n,m}(u)\eta^m,
          \label{eq:introduction-full-expansion}\\
 A_{K,n,m}(u)
      &=\sum_{k\ge0}[\eta^m]c_{K,n,k}(1+\eta)u^k.
          \label{eq:introduction-coefficient-expansion}
\end{align}
Their difference comes from the motion of the Newton nodes, which
begins at order $\eta^2$.

\begin{maintheorem}[Alexander denominators]\label{thm:introduction-Alexander}
For every knot $K$, $n\ge2$, and $m\ge0$,
\begin{equation}\label{eq:introduction-Alexander-bound}
 R_{K,n,m},\ A_{K,n,m}
       \in\delta_K^{-(n-1+2m)}\mathbb Z[u],\qquad
 R_{K,n,0}=A_{K,n,0}=\delta_K^{-(n-1)}.
\end{equation}
For $m\ge1$, the numerators at this denominator satisfy
\begin{equation}\label{eq:introduction-highest-class}
 \delta_K^{n-1+2m}R_{K,n,m}
       \equiv\delta_K^{n-1+2m}A_{K,n,m}
                         \pmod{\delta_K\mathbb Z[u]}.
\end{equation}
For every $n\ge2$ and $m\ge0$, there is a knot for which the
exponent $n-1+2m$ cannot be decreased in either expansion.
\end{maintheorem}

At $n=2$, the leading term is the Alexander series of the
Melvin--Morton--Rozansky theorem \cite{BNG96}. The exponent
$n-1+2m$ also occurs in Gruen's symmetric large-color expansion for
positive and negative braid knots \cite[Theorem~1.4]{Gruen22}.
Here the bound holds for every knot, the numerators are integral, and
the two Newton expansions have the same highest numerator class. The
calculations for $3_1$, $5_2$, and $6_1$ in \cref{sec:intro-examples}
give sharpness at every rank and quantum order.

The denominator count starts from the rational graph expansion of
Rozansky and Kricker \cite{RozanskyRat,Kricker00}. In the
symmetric-power direction of $\mathfrak{sl}_n$ the root charges lie in
$\{-1,0,1\}$, so charge conservation allows at most two active
half-edges at a trivalent vertex. Hence a graph with $2d$ trivalent
vertices has at most $2d$ edges carrying color-dependent Alexander
denominators, while its contribution begins only at quantum order
$d$. The wheel term adds $n-1$ further factors, and the integral
quantum neighborhood gives the integral numerator
(\cref{thm:rank-loop-denominators}).

The two expansions are related by a triangular family of differential
operators in $u$ that express the motion of the Newton nodes. At
quantum order $m\ge1$, every correction of positive order has at least
one fewer Alexander denominator factor, so
$R_{K,n,m}-A_{K,n,m}\in\delta_K^{-(n-2+2m)}\mathbb Z[u]$, which gives
the numerator congruence \eqref{eq:introduction-highest-class}. At
first order the two coefficients agree exactly
(\cref{prop:rank-triangular,thm:rank-Alexander}).

Modulo $\eta$, the rank-uniform knot element is
$\delta_K(u)^{-(N-1)}\in\operatorname{Int}(\mathbb Z)[[u]]$.
For fixed $n,m$, the sequence
$([\eta^m]c_{K,n,k}(1+\eta))_{k\ge0}$ eventually satisfies the linear
recurrence whose coefficients are those of $\delta_K^{n-1+2m}$, and is
eventually periodic modulo any positive integer. When $\Delta_K=1$,
the coefficient generating functions are polynomials: for every $M$,
all sufficiently late Newton coefficients lie in $(q-1)^{M+1}R$
(\cref{cor:universal-classical-knot,cor:jet-recurrence}).

\subsection{Difference equations and Frobenius comparison}\label{sec:intro-difference}

Write $X=z+z^{-1}$. By \cref{thm:introduction-Alexander},
\[
 \mathcal L_{K,n}(z,q)
   =\sum_{m\ge0}R_{K,n,m}(z+z^{-1}-2)\eta^m
   \in\mathbb Z[z^{\pm1},\Delta_K(z)^{-1}][[\eta]],
\]
and in this variable the color shift is $Ef(z,q)=f(q^2z,q)$. For
$P=\sum_{j=0}^d a_j(z,q)E^j$ with
$a_j\in\mathbb Z[z^{\pm1},q^{\pm1}]$, one has
\begin{equation}\label{eq:introduction-annihilators}
 P\mathcal L_{K,n}=0
 \quad\Longleftrightarrow\quad
 \sum_{j=0}^d a_j(q^{n+2r},q)J_{r+j}^{(n)}(K;q)=0
       \ \text{for all sufficiently large }r.
\end{equation}
Thus discrete color recurrences and formal difference equations have
the same annihilators (\cref{thm:rank-annihilators}). Compatibility
with colored Jones recurrences also appears in Willetts's rank-two
unified invariant \cite{Willetts22}.

The first two quantum orders constrain the classical symbol of an
annihilator. Put $P_0(z,L)=\sum_j a_j(z,1)L^j$ and
\[
 A_P(z)=\partial_LP_0(z,1),\qquad
 B_P(z)=\left.\sum_j\partial_q a_j(z,q)\right|_{q=1}.
\]
Then
\begin{equation}\label{eq:introduction-subprincipal}
 P_0(z,1)=0,\qquad
 B_P\Delta_K=2(n-1)A_P\,z\Delta_K'.
\end{equation}
If $\Delta_K$ is squarefree over $\mathbb Q$ and $A_P\ne0$,
then $\Delta_K\mid A_P$ in $\mathbb Q[z^{\pm1}]$
(\cref{prop:rank-subprincipal}).

When $A_P\ne0$, each higher quantum order is obtained by integrating
an explicit rational differential determined by the preceding
orders. A rational solution exists precisely when all finite
residues of this differential vanish. Among solutions regular at
$z=1$, the trivial-color identity
$\mathcal L_{K,n}(q^n,q)=1$ fixes the integration constant at each
order (\cref{thm:rank-residue}).

We use two power substitutions:
\[
 \varphi_p:(z,q)\mapsto(z^p,q),\qquad
 \Psi_p:(z,q)\mapsto(z^p,q^p).
\]
The first acts only on the color coordinate; the second is a
simultaneous Frobenius lift. Both preserve the $p$-adically completed
Alexander localization because
$\Delta_K(z^p)\equiv\Delta_K(z)^p\pmod p$.
The substitution $\varphi_p$ carries recurrences of step $p$ to
recurrences of step one, whereas $\Psi_p$ commutes with the color
shift.
The quotient $\varphi_p(\mathcal L_{K,n})/\mathcal L_{K,n}$ is
integral in this completion and has constant term
$\Delta_K(z)^{n-1}/\Delta_K(z^p)^{n-1}$
(\cref{prop:rank-scalar-frobenius}).

After normalization by the classical Alexander factor, the
simultaneous Frobenius comparison admits an integral logarithm
(\cref{prop:rank-integral-log}). Put
\[
 \mathscr B_{K,p}=
 \widehat{\mathbb Z[z^{\pm1},\Delta_K(z)^{-1}]}_p[[\eta]],
 \qquad U=\Delta_K(z)^{n-1}\mathcal L_{K,n}\in1+\eta\mathscr B_{K,p}.
\]
For every odd prime $p$,
\begin{equation}\label{eq:introduction-Frobenius-log}
 \Lambda_{K,n,p}
    =\frac1p\log\!\left(\frac{\Psi_p(U)}{U^p}\right)
               \in\eta\mathscr B_{K,p}.
\end{equation}
For the trefoil, the first coefficient is
\[
 [\eta]\Lambda_{3_1,n,p}
  =n(n-1)\left(\Delta_{3_1}(z^p)^{-2}
                       -\Delta_{3_1}(z)^{-2}\right)
\]
(\cref{ex:trefoil-frobenius}).

\subsection{Root-of-unity expansions and recurrence methods}
\label{sec:intro-periods}

At $q=\zeta$, the Newton nodes are periodic with period
$\ell=\operatorname{ord}(\zeta^2)$. By the inverse Newton transform,
the color sequence is $q$-holonomic exactly when the
Newton-coefficient sequence is (\cref{prop:inverse}). Grouping a
coefficient recurrence into blocks of length $\ell$ produces a period
matrix, and its resolvent gives the generating functions of the
coefficient blocks.

Put $\sigma=\zeta^{n\ell}\in\{1,-1\}$ and define the complete
period polynomial
\begin{equation}\label{eq:introduction-period}
 W_{n,\zeta}(X)=\Theta_\ell(X)-2\sigma.
\end{equation}
Its roots are the specialized nodes of one period, with multiplicity,
and
\[
 B_{\ell a+b}^{(n)}(X;\zeta)
    =W_{n,\zeta}(X)^a B_b^{(n)}(X;\zeta),
       \qquad a\ge0,\quad0\le b<\ell.
\]
A generating function in the block index $a$ therefore becomes a
function of $X$ by substituting $v=W_{n,\zeta}(X)$.

For $F=\sum_k c_kB_k^{(n)}$, set
\[
 H_{b,m}(v)=\sum_{a\ge0}[\epsilon^m]
          c_{\ell a+b}(\zeta(1+\epsilon))v^a.
\]
When the coefficients satisfy a Laurent-polynomial transfer
recurrence, its specialized period matrix $M_0$ gives the block
resolvent $(I-vM_0)^{-1}$. At a nonzero eigenvalue $\lambda$ with
largest Jordan block of size $s_\lambda$, the order-$m$ series has
pole order at most $(s_\lambda+1)m+s_\lambda$ at
$v=\lambda^{-1}$ (\cref{thm:resolvent}). The full expansion
also includes the quantum variation of the Newton kernel.

\begin{maintheorem}[From coefficient blocks to the full expansion]
\label{thm:introduction-period-transfer}
Suppose $d(v)\in\mathbb Z[\zeta][v]$, $d(0)=1$, and $\beta\ge1$
satisfy
\[
 d(v)^{\beta+2m}H_{b,m}(v)\in\mathbb Z[\zeta][v]
                  \qquad(0\le b<\ell,\ m\ge0).
\]
Then
\begin{equation}\label{eq:introduction-period-transfer}
 [\epsilon^m]F(X;\zeta(1+\epsilon))
    \in d(W_{n,\zeta}(X))^{-(\beta+2m)}\mathbb Z[\zeta][X].
\end{equation}
Every positive-order contribution from the Newton kernel has
denominator exponent at most $\beta+2m-1$.
\end{maintheorem}

The proof uses that, at positive quantum order $h$, the corrections
from the moving kernel are polynomials of degree at most $2h-1$ in the
block index (\cref{thm:rank-full-series}).

Expansions at different roots of unity are related by $p$-adic
recentering.
Let $\Delta\in\mathbb Z[z^{\pm1}]$ be symmetric with $\Delta(1)=1$,
and put $D_\zeta(z)=\Delta(\zeta^{n\ell}z^\ell)$.
Suppose an element of $\mathcal C_n$ has, at every root $\zeta$, a
full expansion
$f_\zeta(z,t)\in\mathbb Z[\zeta,z^{\pm1},D_\zeta^{-1}][[t]]$,
where $t=q-\zeta$ and $X=z+z^{-1}$.
If $\zeta,\zeta'$ lie in a common finite $p$-adic cyclotomic field
and $v_p(\zeta'-\zeta)>0$, then
\begin{equation}\label{eq:introduction-recentering}
 f_\zeta(z,t+\zeta'-\zeta)=f_{\zeta'}(z,t)
\end{equation}
in the common $p$-adically completed localization. At finite
$p$-adic precision this identity uses only finitely many Taylor
coefficients. The family of root expansions determines the original
Newton element (\cref{thm:rank-recentering}).

At a root of unity, each normal coefficient of an element of
$\mathcal C_n$ is given by a finite formula. At the boundary
$a_{n-2s}$, where $1\le s\le\lfloor n/2\rfloor$, the kernel is
\begin{equation}\label{eq:introduction-normal-kernel}
 B_k^{(n)}(a_{n-2s}+u_s;q)
    =\prod_{i=0}^{k-1}\bigl(u_s-\{s+i\}\{n-s+i\}\bigr).
\end{equation}
At $q=\zeta$, the constant term of the $i$th factor vanishes when
$i\equiv-s$ or $i\equiv s-n\pmod\ell$. Each such index contributes a
factor $u_s$, so a term with $j+1$ such indices does not contribute to
normal order $j$.

List the distinct residues of these indices as
\[
 \{\rho_0<\cdots<\rho_{\nu-1}\}
     =\{(-s)\bmod\ell,(s-n)\bmod\ell\},
\]
where $\nu\in\{1,2\}$. The $(j+1)$st such index is
$N_j=\ell\lfloor j/\nu\rfloor+\rho_{j\bmod\nu}$, and at each of them
the constant term vanishes to order $e_\partial=2/\nu$ in the quantum
parameter.

\begin{maintheorem}[Finite normal reconstruction]\label{thm:introduction-normal}
Let $F=\sum_k c_kB_k^{(n)}\in\mathcal C_n$, and let
$t_j(F)\in\mathcal H$ be its normal coefficients at $a_{n-2s}$.
Then
\begin{equation}\label{eq:introduction-normal-finite}
 t_j(F)(\zeta)=\sum_{k=0}^{N_j}c_k(\zeta)
       [u_s^j]B_k^{(n)}(a_{n-2s}(\zeta)+u_s;\zeta).
\end{equation}
The index $N_j$ is sharp for arbitrary integral Newton elements.
For mixed quantum order $m$, the coefficient
$[\epsilon^m]t_j(F)(\zeta(1+\epsilon))$ uses only indices
$k\le N_{j+\lfloor m/e_\partial\rfloor}$.
\end{maintheorem}

At $n=2$, the boundary is $X=2$, and the cutoffs are
\begin{equation}\label{eq:introduction-normal-cutoff}
 N_j=(j+1)\ell-1,\qquad
 N_{j,m}=\bigl(j+\lfloor m/2\rfloor+1\bigr)\ell-1
\end{equation}
(\cref{thm:finite-normal-jets}). The exact Laurent polynomials of
colors $0,\ldots,(j+1)\ell-1$ therefore recover
$\tau_{K,j}^{(2,1)}(\zeta)$. The inverse Newton transform also gives
an explicit sufficient quantum precision for a finite collection
of colored jets to recover any prescribed mixed normal coefficient
(\cref{cor:finite-normal-color-data}).

If the specialized function is rational, its normal coefficients
satisfy an eventual homogeneous constant-coefficient recurrence.
Over $\mathbb Q(\zeta)$, the minimal recurrence order is the degree
of the reduced denominator in the normal coordinate
(\cref{prop:rational-normal-recurrence}).

Transfer recurrences that are equivalent over the function field
need not be equivalent over the Laurent ring. The sum of the divisor
orders of the transfer determinant along a shift orbit is invariant
under rational changes of basis, and a nonzero orbit sum rules out a
transfer matrix invertible over the Laurent ring
(\cref{prop:orbit}). For period matrices we also ask whether two
matrices with the same characteristic polynomial are conjugate over
that ring.

If the shift acts by $\xi\mapsto\omega\xi$ with $\omega$ of order
$\ell$, the coefficients of the characteristic polynomial of the
period matrix are Laurent polynomials in $\xi^\ell$
(\cref{prop:period}). To compare two matrices, put
$A_0=k[\xi^{\pm1}]$ for a field $k$, let $p(Y)\in A_0[Y]$ be monic of degree $r$,
and suppose $S=A_0[Y]/(p)$ is the coordinate ring of a smooth
connected affine curve. A matrix $M\in\operatorname{Mat}_r(A_0)$
with characteristic polynomial $p$ defines an invertible $S$-module
$\mathscr P_M$: its underlying $A_0$-module is $A_0^r$, and $Y$
acts through $M$. The Latimer--MacDuffee correspondence identifies
conjugacy over $A_0$ with isomorphism of these modules
\cite{LM33,KS24}.

For two such matrices set
$\mathscr D=\mathscr P_{M_1}\otimes_S\mathscr P_{M_2}^{-1}$.
Repeated direct sums satisfy
\begin{equation}\label{eq:introduction-stabilization}
 M_1^{\oplus m}\sim_{\mathrm{GL}_{rm}(A_0)}M_2^{\oplus m}
       \quad\Longleftrightarrow\quad
 \mathscr D^{\otimes m}\simeq S.
\end{equation}
Thus the order of $[\mathscr D]\in\operatorname{Pic}(S)$ is the
least number of identical summands needed for conjugacy. If
$[\mathscr D]$ has finite order $d$ invertible in $k$, a connected
finite \'etale cover of degree $d$ trivializes $\mathscr D$, and
taking norms shows that every finite locally free trivializing cover
has degree divisible by $d$ (\cref{thm:stabilization}).

\subsection{Explicit knot elements in every rank}
\label{sec:intro-examples}

\begin{samepage}
The knots $3_1,4_1,5_2$, and $6_1$ have explicit Newton coefficients
for every $n\ge2$, obtained from the double-twist formulas of
\cite{CLZ21,FZ26}. Write
$(c_{3_1},c_{4_1},c_{5_2},c_{6_1})=(-1,1,-2,2)$, so that
$\delta_K(u)=1-c_Ku$ for these four knots.
The figure-eight coefficient sequence has the product generating
function
\begin{equation}\label{eq:introduction-figure-product}
 \sum_{k\ge0}c_{4_1,n,k}(q)v^k
       =\prod_{j=0}^{n-2}(1-q^{n-2-2j}v)^{-1}.
\end{equation}
It gives the exact pole orders of the coefficient jets, which
improve the general Alexander bound
(\cref{thm:figure-eight-product}).
\end{samepage}

Each of $3_1,5_2$, and $6_1$ attains the general bound at every rank
and every quantum order. With
$(\kappa_{3_1},\kappa_{5_2},\kappa_{6_1})=(1,7/4,-3/4)$,
the highest numerator value is
\begin{equation}\label{eq:introduction-sharp-value}
 \left.(1-c_Ku)^{n-1+2m}A_{K,n,m}(u)\right|_{u=1/c_K}
   =\frac{(n+2m-2)!}{(n-2)!\,m!}\,\kappa_K^m\ne0.
\end{equation}
Together with the numerator congruence, this proves sharpness for
both expansions in \cref{thm:introduction-Alexander}
(\cref{thm:all-rank-sharp}).

The trefoil and figure-eight coefficient sequences satisfy companion
recurrences of size $n-1$ whose matrices are invertible over the
Laurent ring. For $5_2$ and $6_1$ at $n=2$, the monic $C$-recurrences
of \cite{GS06} have nonzero determinant orbit sums, so no rational
change of basis makes them invertible over the Laurent ring
(\cref{prop:twist-C-recurrences}).

Grouping the coefficient formulas into complete periods gives the
specialization of each of the four knots at a root of unity in terms
of its first block.

\begin{maintheorem}[Four knots at every root of unity]\label{thm:introduction-four-roots}
Let $K\in\{3_1,4_1,5_2,6_1\}$, $n\ge2$, and let $\zeta$ be any
root of unity. Set $\ell=\operatorname{ord}(\zeta^2)$,
$\sigma=\zeta^{n\ell}$, and $a_0=\lfloor(n-2)/\ell\rfloor$.
Then, in the specialized Newton completion,
\begin{equation}\label{eq:introduction-four-roots}
 F_{K,n}(X;\zeta)
 =\frac{\displaystyle\sum_{b=0}^{\ell-1}
            c_{K,n,b}(\zeta)B_b^{(n)}(X;\zeta)}
        {(1-c_K\sigma W_{n,\zeta}(X))^{a_0+1}}.
\end{equation}
With $X=z+z^{-1}$, the displayed denominator is
$\Delta_K(\sigma z^\ell)^{a_0+1}$.
\end{maintheorem}

The formula follows from the $q$-Lucas identity and cancellation of
phases in the double-twist coefficient sums
(\cref{thm:four-all-roots}). At $\zeta=1$, it reduces to the
classical specialization $\delta_K^{-(n-1)}$.
At $n=2$, the Alexander denominator is also related to the
ADO--Alexander factorization of Willetts's unified invariant
\cite{Willetts22}.

For the trefoil and figure-eight we also bound all higher quantum
orders at every root. The trefoil has full order-$m$ denominator
dividing $(1+\sigma W)^{a_0+2m+1}$, and its nonvanishing coefficient
blocks attain the corresponding pole order. For the figure-eight,
the full denominator divides $(1-\sigma W)^{a_0+2m}$ for $m\ge1$.
Both knot elements therefore satisfy the recentering theorem at all
roots (\cref{thm:trefoil-root-jets}).

At $q=i$, the four specializations are
\begin{align}
 F_{K,2d+1}(X;i)&=(1+c_KX^2)^{-d},\qquad d\ge1,
       \label{eq:introduction-odd-i}\\*
 F_{K,2d}(X;i)&=
 \frac{1+2c_K+(-1)^{d-1}c_KX}
      {(1-c_K(X^2-4))^d},\qquad d\ge1.
       \label{eq:introduction-even-i}
\end{align}
Their expansions give every normal coefficient at every earlier
boundary. Each normal sequence has minimal recurrence order
$2\lfloor n/2\rfloor$ (\cref{thm:four-at-i}).

At $n=3$, all four boundary values are one and all first normal
coefficients are zero, while the second normal coefficients are
$1,-1,2,-2$, respectively.

The rank-two $C$-recurrences of $5_2$ and $6_1$ also give nontrivial
spectral line modules at $q=i$. Put $Q=q^2$ and let $\xi$ be the
independent coefficient-index variable, with $\xi=Q^k$ along the
coefficient sequence. At $q=i$ the shift sends $\xi$ to $-\xi$. Let $B_K(\xi)$ be the one-step matrix of the monic
$C$-recurrence in the consecutive-coefficient basis, and put
$M_K(\xi)=B_K(-\xi)B_K(\xi)$. Then
\[
 p_K(L,\xi)=\det(LI-M_K(\xi))=C_K(L,\xi^2;1).
\]
The associated spectral module determines whether the period matrix
is conjugate over the Laurent ring to the companion matrix of $p_K$.

\begin{maintheorem}[Elliptic spectral modules]\label{thm:introduction-spectral}
For $K=5_2,6_1$, put $A_0=\mathbb C[\xi^{\pm1}]$ and
$S_K=A_0[L]/(p_K)$. Both spectral curves are isomorphic to
\[
 E^\circ=E\setminus E[2],\qquad
 E:\ y^2=x^3+x^2-x,
\]
where $E$ is the smooth projective curve with origin $O$ at infinity
and $E[2]$ is its two-torsion subgroup.
Under the coordinate identifications of \cref{thm:elliptic-spectral},
the modules $\mathscr P_K=A_0^2$, with $L$ acting by $M_K$, satisfy
\[
 \mathscr P_{5_2}\simeq\mathcal O_{E^\circ}(P),\quad P=(-1,-1),
 \qquad
 \mathscr P_{6_1}\simeq\mathcal O_{E^\circ}(R),\quad R=(1,1).
\]
Their classes have exact order three in $\operatorname{Pic}(E^\circ)$
and are inverse to one another. For every positive integer $m$,
\begin{equation}\label{eq:introduction-threefold}
 M_K^{\oplus m}\sim_{\mathrm{GL}_{2m}(A_0)}
             \operatorname{Comp}(p_K)^{\oplus m}
       \quad\Longleftrightarrow\quad3\mid m.
\end{equation}
\end{maintheorem}

The companion matrix $\operatorname{Comp}(p_K)$ has free spectral
module and is conjugate to $M_K$ over $\mathbb C(\xi)$, so the
obstruction in \eqref{eq:introduction-threefold} appears only over the
Laurent ring.

The orders come from the group law, $2P=R$, $3P=(0,0)$, and
$3R=O$, together with the identification
$\operatorname{Pic}(E^\circ)\simeq E(\mathbb C)/E[2]$. A connected
finite \'etale cover of
degree three trivializes each spectral module, and every finite
locally free trivializing cover has degree divisible by three
(\cref{thm:elliptic-spectral,cor:threefold-conjugacy}).

\subsection{Organization}\label{sec:intro-organization}

Part~I concerns the Newton rings themselves.
Section~\ref{sec:general} proves the identity theorem and the Habiro
boundary theorems, Section~\ref{sec:sampling} treats cyclotomic
sampling and $p$-adic color interpolation, and
Section~\ref{sec:universal} constructs the rank-uniform quantum
neighborhoods.

Part~II concerns knot elements. Section~\ref{sec:jones} contains the
rank--color relations and finite-rank reconstruction.
Section~\ref{sec:alexander} proves the Alexander denominator bounds,
and Section~\ref{sec:recurrences} treats difference equations and
Frobenius comparisons. Section~\ref{sec:transport} contains the
root-of-unity expansions, finite normal reconstruction, and the
comparison of Laurent recurrence modules. Section~\ref{sec:examples}
gives the explicit computations for $3_1$, $4_1$, $5_2$, and $6_1$.

\part{General Theory}

\section{Newton rings and Habiro boundaries}\label{sec:general}

\subsection{Interpolation on a node set}

Let $D$ be an integral domain with fraction field $L$, and let $S=(s_0,s_1,\ldots)$ be a sequence of distinct elements of $D$. Set
\[
 P_0=1,\qquad P_m(X)=\prod_{i<m}(X-s_i),\qquad
 \NN_D(S)=\varprojlim_mD[X]/(P_m).
\]
All Newton inverse limits carry the topology defined by their finite quotient maps; the coefficient ring is discrete.

\begin{proposition}\label{prop:coordinates}
The residue classes of $P_0,\ldots,P_{m-1}$ form a $D$-basis of $D[X]/(P_m)$. Every element of $\NN_D(S)$ has a unique expansion
\begin{equation}\label{eq:newton-series}
 F=\sum_{k\ge0}c_kP_k,\qquad c_k\in D.
\end{equation}
As a topological $D$-module, $\NN_D(S)$ is $\prod_{k\ge0}D$.
\end{proposition}
\begin{proof}
The polynomial $P_k$ is monic of degree $k$, so the change from the monomial basis to the Newton basis is triangular with diagonal entries one. Reduction from level $m+1$ to level $m$ discards the coefficient of $P_m$ and leaves all earlier coefficients unchanged. Compatible finite residues are therefore precisely arbitrary sequences $(c_k)_{k\ge0}$.
\end{proof}

\begin{definition}
A function $f:S\to D$ is polynomially compatible if, for every finite subset $T\subset S$, there is a polynomial $Q_T\in D[X]$ whose restriction to $T$ is $f|_T$. Write $\Comp_D(S)$ for the ring of these functions with pointwise operations and the topology of agreement on finite sets.
\end{definition}

\begin{proposition}\label{prop:compatible}
Evaluation is a topological ring isomorphism
\[
 \NN_D(S)\xrightarrow{\sim}\Comp_D(S).
\]
In particular, the ring and its topology depend on the node set, not its enumeration.
\end{proposition}
\begin{proof}
A polynomial vanishing on $s_0,\ldots,s_{m-1}$ is divisible by $P_m$. Indeed, after dividing by $X-s_0$, evaluation at each other node still vanishes because $D$ is a domain and $s_i-s_0\ne0$. Repeating gives the assertion. Thus $D[X]/(P_m)$ identifies with the polynomially representable functions on the first $m$ nodes. Taking inverse limits gives the isomorphism. Finite initial sets form a cofinal subsystem of all finite subsets, proving enumeration independence.
\end{proof}

For $k\ge1$, distinct nodes $s_0,\ldots,s_k$, and values $f_0,\ldots,f_k$, define
\[
 V_k=\prod_{0\le i<j\le k}(s_j-s_i),\qquad
 \Delta_k(f)=\det\begin{pmatrix}
 1&s_0&\cdots&s_0^{k-1}&f_0\\
 1&s_1&\cdots&s_1^{k-1}&f_1\\
 \vdots&\vdots&&\vdots&\vdots\\
 1&s_k&\cdots&s_k^{k-1}&f_k
 \end{pmatrix},
\]
where $V_0=1$ and $\Delta_0(f)=f_0$.

\begin{proposition}\label{prop:divided-difference}
Every sequence in $L$ has a unique triangular Newton interpolation, with coefficients
\begin{equation}\label{eq:inversion}
 c_k=\frac{f_k-\sum_{j<k}P_j(s_k)c_j}{P_k(s_k)}
     =\frac{\Delta_k(f)}{V_k}.
\end{equation}
A sequence in $D$ is polynomially compatible if and only if $\Delta_k(f)\in V_kD$ for all $k$. For such a sequence,
\begin{equation}\label{eq:pairwise}
 f_r-f_s\in(s_r-s_s)D.
\end{equation}
\end{proposition}
\begin{proof}
The evaluation matrix of the Newton basis is lower triangular with nonzero diagonal, which gives the recursion and uniqueness. The $k$th Newton coefficient is the leading coefficient of the interpolant of degree at most $k$, so Cramer's rule in the monomial basis gives the determinant expression. Integrality at every finite level is equivalent to compatibility. Finally, evaluating one integral interpolating polynomial at two nodes gives \eqref{eq:pairwise}.
\end{proof}

\begin{example}
The data $(s_0,s_1,s_2)=(0,2,4)$ and $(f_0,f_1,f_2)=(0,0,4)$ satisfy all pairwise congruences over $\Z$, but $c_2=1/2$. Thus the pairwise congruences \eqref{eq:pairwise} do not imply integrality of the higher divided differences.
\end{example}

\begin{proposition}\label{prop:finite-order}
Assume $D$ is normal, let $m\ge1$, and put $A_m=D[X]/(P_m)$. Its normalization in $L^m$ is $D^m$, its discriminant is $V_{m-1}^2$, and the conductor of $A_m\subset D^m$ is
\[
 \bigoplus_{i=0}^{m-1} d_iD e_i,
 \qquad d_i=\prod_{\substack{0\le j<m\\j\ne i}}(s_i-s_j).
\]
Here $e_i$ is the $i$th coordinate idempotent of $D^m$.
For $F\in A_m$, its norm is $\prod_{i<m}F(s_i)$; it is a unit if and only if all these values are units of $D$.
\end{proposition}
\begin{proof}
Evaluation gives $A_m\otimes_D L\simeq L^m$.
Since $A_m$ is finite free over $D$, an element integral over $A_m$ is
integral over $D$. If it lies in $L^m$, normality puts every coordinate
in $D$. Conversely, $D^m$ is finite over $D$ and hence integral over
$A_m$. Thus $D^m$ is the normalization.

The trace matrix in the monomial basis is the product of a Vandermonde
matrix and its transpose. Its determinant is $V_{m-1}^2$.
The unique polynomial representative of $e_i$ over $L$ of degree less than $m$ is
\[
 L_i(X)=\frac{\prod_{j\ne i}(X-s_j)}{d_i}.
\]
For $c\in D$, membership $ce_i\in A_m$ is equivalent to
$cL_i\in D[X]$, which holds precisely when $c/d_i\in D$.
The conductor is an ideal of $D^m$, so these conditions give
$\bigoplus_i d_iD e_i$.

Multiplication by $F$ is diagonal over $L$ with entries $F(s_i)$.
Its determinant is therefore $\prod_iF(s_i)$.
It is a unit exactly when multiplication by $F$ is invertible on the
finite free $D$-module $A_m$, equivalently when $F$ is a unit of $A_m$.
\end{proof}

Finitely many values determine an element up to the following ambiguity.

\begin{proposition}\label{prop:ambiguity}
Let $T\subset S$ be finite and nonempty, put $P_T(X)=\prod_{t\in T}(X-t)$, and let $s\in S\setminus T$. Among elements of $\NN_D(S)$ vanishing on $T$, the possible values at $s$ form exactly the ideal $P_T(s)D$. More generally, if $Q_T$ is the unique polynomial of degree less than $|T|$ interpolating an element $F$ on $T$, then $Q_T\in D[X]$ and
\begin{equation}\label{eq:arbitrary-remainder}
 F(s)-Q_T(s)\in P_T(s)D.
\end{equation}
\end{proposition}
\begin{proof}
Compatibility on $T\cup\{s\}$ supplies an integral polynomial representative. If the values on $T$ vanish, repeated division by its monic linear factors shows that its value at $s$ is divisible by $P_T(s)$. Conversely, $aP_T$ realizes every value $aP_T(s)$ with $a\in D$. For a general $F$, divide an integral representative on $T$ by $P_T$; its remainder is the unique interpolant $Q_T$. Apply the first assertion to $F-Q_T$.
\end{proof}

\subsection{The quantum nodes and their arithmetic}\label{sec:arithmetic}

Put $R=\Z[q^{\pm1}]$ and, for every $n\ge0$, set
\begin{equation}\label{eq:quantum-ring}
 x_r^{(n)}=q^{n+2r}+q^{-n-2r},\qquad
 B_k^{(n)}(X)=\prod_{i<k}(X-x_i^{(n)}),\qquad
 \C_n=\varprojlim_kR[X]/(B_k^{(n)}).
\end{equation}
We use $a_n=x_0^{(n)}$ and $\{j\}=q^j-q^{-j}$. The nodes satisfy
\begin{equation}\label{eq:node-difference}
 x_r^{(n)}-x_s^{(n)}=\{r-s\}\{n+r+s\}.
\end{equation}
If $r>s\ge0$, the integers $r-s$ and $n+r+s$
are positive, so the nodes are distinct.

\begin{theorem}\label{thm:identity}
Two elements of $\C_n$ that agree on any infinite set of colors are equal. In particular, $\C_n$ is a domain.
\end{theorem}
\begin{proof}
For a nonzero Laurent polynomial, its breadth is the difference between
its largest and smallest exponents. Breadth is additive under
multiplication over a domain.
Write $f_r=F(x_r^{(n)})$ and suppose $f_r=0$ on an infinite set $E$. Fix $s$. By \eqref{eq:pairwise}, $x_r^{(n)}-x_s^{(n)}$ divides $f_s$ for every $r\in E$. For $r>s$, the Laurent breadth of this divisor is $2n+4r$, since its extreme terms have exponents $\pm(n+2r)$. Thus a nonzero fixed $f_s$ cannot be divisible by these elements for unbounded $r$. Hence $f_s=0$ for every $s$, and \cref{prop:compatible} implies $F=0$. Apply this to a difference for the first assertion. If $FG=0$, the zero sets of their color values cover $\N$. One zero set is infinite, so one factor is zero.
\end{proof}

\begin{theorem}\label{thm:units-primes}
For every $n\ge0$,
\begin{equation}\label{eq:units}
 \C_n^\times=R^\times=\{\varepsilon q^a:\varepsilon\in\{1,-1\},\ a\in\Z\}.
\end{equation}
For every rational prime $p$, coefficient reduction induces an isomorphism
\begin{equation}\label{eq:prime-reduction}
 \C_n/p\C_n\simeq
 \varprojlim_m\Fp[q^{\pm1}][X]/\bigl(\overline{B_m^{(n)}}\bigr),
\end{equation}
and this quotient is a domain.
\end{theorem}
\begin{proof}
An element is a unit in $\C_n$ if and only if all its color values are units in $R$: sufficiency follows by taking the compatible inverses in the finite free quotients, using \cref{prop:finite-order}. Write those values as $f_r=\varepsilon_rq^{a_r}$. For $r\ne s$, \eqref{eq:node-difference} shows that $(q-1)^2$ divides $f_r-f_s$. Evaluation at $q=1$ first gives $\varepsilon_r=\varepsilon_s$. Differentiation at $q=1$ then gives $a_r=a_s$. All color values are therefore the same Laurent unit, and \cref{thm:identity} makes the element constant. The converse is immediate.

Reduction of every Newton coefficient defines the map in \eqref{eq:prime-reduction}. It is surjective by coefficientwise lifting. Its kernel consists of series whose coefficients all belong to $pR$, and dividing every coefficient by $p$ identifies that kernel with $p\C_n$. The map is multiplicative because multiplication is defined at finite levels. Over $\Fp[q^{\pm1}]$ the quantum nodes remain distinct. For $r>s$, the extreme coefficients in their difference remain nonzero, including when $p=2$ and $n=0$. The breadth argument of \cref{thm:identity} therefore applies in the same way and proves that the quotient is a domain.
\end{proof}

\begin{corollary}\label{cor:sparse-units}
If $F\in\C_n$ has Laurent-unit values on infinitely many colors, then $F$ is a constant Laurent unit.
\end{corollary}
\begin{proof}
The value-and-derivative argument in the preceding proof shows that all the unit values on that infinite set are equal. Subtract their common value and apply \cref{thm:identity}.
\end{proof}

For $m\ge1$, put
\begin{equation}\label{eq:newton-open-ideals}
 \I_m^{(n)}=\ker\bigl(\C_n\longrightarrow R[X]/(B_m^{(n)})\bigr)
 =\{F\in\C_n:F(x_r^{(n)})=0\text{ for }0\le r<m\}.
\end{equation}
These kernels form a neighborhood basis of zero.

\begin{proposition}\label{prop:nonadic}
The Newton topology on $\C_n$ is not the $J$-adic topology for any ideal $J\subset\C_n$.
\end{proposition}
\begin{proof}
If these topologies agreed, for each $r$ there would be $k_r$ such that $J^{k_r}\subset\I_{r+1}^{(n)}$. For $a\in J$, evaluation gives $a(x_r)^{k_r}=0$ in the domain $R$, hence $a(x_r)=0$ for all $r$. Therefore $a=0$ and $J=0$. The zero-adic topology is discrete, whereas the Newton topology is not: $B_m^{(n)}$ is a nonzero element of every corresponding open ideal $\I_m^{(n)}$.
\end{proof}

\subsection{Tail restriction and the Habiro boundary}\label{sec:boundary}

The identity $x_{m+r}^{(n)}=x_r^{(n+2m)}$ defines
\[
 \rho_{n,m}:\C_n\longrightarrow\C_{n+2m},\qquad
 (\rho_{n,m}F)(x_r^{(n+2m)})=F(x_{m+r}^{(n)}).
\]
Restriction preserves polynomial compatibility. These maps are continuous and injective by \cref{prop:compatible,thm:identity},
and satisfy $\rho_{n+2m,k}\rho_{n,m}=\rho_{n,m+k}$. We write $\rho_n=\rho_{n,1}$.

For $m\ge1$, there is also a coefficient-shift isomorphism of topological $R$-modules
\begin{equation}\label{eq:mu}
 \mu_{n,m}:\C_{n+2m}\xrightarrow{\sim}\I_m^{(n)},\qquad
 \sum_{k\ge0}c_kB_k^{(n+2m)}\longmapsto\sum_{k\ge0}c_kB_{m+k}^{(n)}.
\end{equation}
Its image consists of the series whose first $m$ coefficients vanish. The factorization $B_{m+k}^{(n)}=B_m^{(n)}B_k^{(n+2m)}$ gives, first for polynomials and then by continuity,
\begin{equation}\label{eq:mu-linearity}
 F\mu_{n,m}(G)=\mu_{n,m}(\rho_{n,m}(F)G).
\end{equation}
In particular, the Newton divided difference
\begin{equation}\label{eq:partial}
 \partial_nF=\mu_{n,1}^{-1}(F-F(a_n))
          =\frac{\rho_nF-F(a_n)}{X-a_n}\quad\hbox{in }\C_{n+2}
\end{equation}
shifts the Newton coefficients by one and satisfies
\[
 \partial_n(FG)=\partial_n(F)\rho_n(G)+F(a_n)\partial_n(G).
\]

Write
\[
 \HH=\varprojlim_m R/((q;q)_m),\qquad
 (q;q)_m=\prod_{j=1}^m(1-q^j)
\]
for the Habiro ring. Set $u=X-a_n$ in the tail ring, and put
\begin{equation}\label{eq:delta}
 \delta_0^{(n)}=1,\qquad
 \delta_k^{(n)}=B_k^{(n+2)}(a_n)
    =(-1)^k\prod_{j=1}^k\br{j}\br{n+j}.
\end{equation}

\begin{lemma}\label{lem:cofinality}
The ideals $(\delta_k^{(n)})$ and $((q;q)_m)$ are cofinal. Thus
$\varprojlim_kR/(\delta_k^{(n)})=\HH$ canonically.
\end{lemma}
\begin{proof}
Every nonunit irreducible factor in \eqref{eq:delta} is cyclotomic. For $\pi=\Phi_d(q)$ and $\ell=d/\gcd(d,2)$, the factor $\br{j}$ has $\pi$-valuation one when $\ell\mid j$, and zero otherwise. This gives
\begin{equation}\label{eq:delta-valuation}
 v_\pi(\delta_k^{(n)})=
 \left\lfloor\frac k\ell\right\rfloor+
 \left\lfloor\frac{n+k}\ell\right\rfloor-
 \left\lfloor\frac n\ell\right\rfloor.
\end{equation}
For every fixed $d$ this tends to infinity. Consequently every fixed Habiro factorial divides all sufficiently late $\delta_k^{(n)}$. Conversely, a fixed $\delta_k^{(n)}$ has finitely many cyclotomic factors, each of finite multiplicity, so it divides a sufficiently late Habiro factorial. Unique factorization in $R$ proves cofinality.
\end{proof}

Evaluation at $a_n$ on finite quotients now gives a homomorphism
\begin{equation}\label{eq:beta}
 \beta_n:\C_{n+2}\longrightarrow\HH,\qquad
 \beta_n\left(\sum_{k\ge0}c_kB_k^{(n+2)}\right)=\sum_{k\ge0}c_k\delta_k^{(n)}.
\end{equation}

\begin{proposition}\label{prop:boundary-exact}
The map $\beta_n$ is surjective, and
\begin{equation}\label{eq:boundary-exact}
 0\longrightarrow\C_{n+2}\xrightarrow{\ u\ }\C_{n+2}
   \xrightarrow{\ \beta_n\ }\HH\longrightarrow0
\end{equation}
is an exact sequence of $R$-modules.
\end{proposition}
\begin{proof}
For $h\in\varprojlim_kR/(\delta_k^{(n)})$, choose $c_0,c_1,\ldots$ recursively so that
$h\equiv\sum_{j=0}^kc_j\delta_j^{(n)}\pmod{\delta_{k+1}^{(n)}}$.
The recursion is possible because the kernel of
$R/(\delta_{k+1})\to R/(\delta_k)$ consists of multiples of $\delta_k$.
This proves surjectivity. Multiplication by $u$ is injective because $\C_{n+2}$ is a domain, and its image lies in the boundary kernel.

For the converse, write $G=\sum g_kB_k^{(n+2)}$ with $\beta_nG=0$ and
$s_{k+1}=\sum_{j=0}^kg_j\delta_j^{(n)}$. Then $s_{k+1}\in\delta_{k+1}^{(n)}R$. Put
\[
 h_k=-s_{k+1}/\delta_{k+1}^{(n)},\qquad H=\sum_{k\ge0}h_kB_k^{(n+2)}.
\]
If $d_k=a_n-x_k^{(n+2)}$, then
$\delta_{k+1}=d_k\delta_k$ and
$uB_k^{(n+2)}=B_{k+1}^{(n+2)}-d_kB_k^{(n+2)}$.
The $B_0$-coefficient of $uH$ is $-d_0h_0=g_0$, and its $k$th coefficient for $k\ge1$ is
$h_{k-1}-d_kh_k=(s_{k+1}-s_k)/\delta_k=g_k$.
Thus $G=uH$.
\end{proof}

\begin{theorem}\label{thm:pullback}
The map
\begin{equation}\label{eq:pullback}
 F\longmapsto(F(a_n),\rho_nF)
 \quad\hbox{identifies}\quad
 \C_n\simeq R\times_{\HH}\C_{n+2}
\end{equation}
as topological $R$-algebras. In particular,
$\rho_n(\C_n)=\beta_n^{-1}(R)$.
\end{theorem}
\begin{proof}
In $R[X]$, set $I=(X-a_n)$ and $J=(B_k^{(n+2)})$.
The element $X-a_n$ is prime and does not divide $B_k^{(n+2)}$,
since $B_k^{(n+2)}(a_n)=\delta_k^{(n)}\ne0$. Hence
\[
 I\cap J=IJ=(B_{k+1}^{(n)}),\qquad
 R[X]/(I+J)\simeq R/(\delta_k^{(n)}).
\]
The two-ideal exact sequence therefore identifies
\[
 R[X]/(B_{k+1}^{(n)})\simeq
 R\times_{R/(\delta_k^{(n)})}R[X]/(B_k^{(n+2)}).
\]
The two coordinates are evaluation at $a_n$ and restriction to the
remaining nodes. These identifications commute with the transition
maps. Inverse limits commute with fiber products, and
\cref{lem:cofinality} identifies the inverse limit of the common
quotients with $\HH$. This proves \eqref{eq:pullback}, including
its topology.

The map $R\to\HH$ is injective: a nonzero Laurent polynomial cannot
be divisible by Habiro factorials of unbounded breadth. Thus an element
of $\C_{n+2}$ has a preimage precisely when its boundary value belongs
to $R$, proving the assertion about $\rho_n(\C_n)$.
\end{proof}

\begin{corollary}\label{cor:several-colors}
Let $N=n+2m$ and $1\le s\le m$. Evaluation at $x_{-s}^{(N)}$ defines a Habiro boundary map $\beta_{N,s}:\C_N\to\HH$. An element $G\in\C_N$ belongs to $\rho_{n,m}(\C_n)$ if and only if every $\beta_{N,s}(G)$ lies in $R$. For $G=\rho_{n,m}F$,
\[
 \beta_{N,s}(G)=F(x_{m-s}^{(n)}).
\]
\end{corollary}
\begin{proof}
At the negative node in question,
\[
 B_k^{(N)}(x_{-s}^{(N)})
 =(-1)^k\prod_{i<k}\{s+i\}\{N-s+i\}.
\]
Both index strings are positive. Every irreducible factor is
cyclotomic, and every fixed cyclotomic multiplicity tends to infinity
with $k$. The argument of \cref{lem:cofinality} therefore defines
the continuous boundary map $\beta_{N,s}$.

By \cref{thm:pullback}, the condition $\beta_{N,1}(G)\in R$ is
equivalent to a unique lift $G_1\in\C_{N-2}$ with
$\rho_{N-2}(G_1)=G$. For $1\le t<m$, the boundary maps satisfy
\[
 \beta_{N-2,t}(G_1)=\beta_{N,t+1}(G).
\]
This identity holds on polynomials because both sides evaluate at
$q^{N-2-2t}+q^{-N+2+2t}$, and it extends by continuity.
Induction on $m$ now gives the stated criterion. For
$G=\rho_{n,m}F$, the same evaluation identity gives
$\beta_{N,s}(G)=F(x_{m-s}^{(n)})$.
\end{proof}

\subsection{Open ideals and non-Noetherianity}

\begin{theorem}\label{thm:defect}
There is a canonical isomorphism of $\C_n$-modules
\[
 \I_1^{(n)}/(X-a_n)\C_n\simeq\HH/R,
\]
where $\C_n$ acts through evaluation at $a_n$. For $m\ge1$, the module
$\I_m^{(n)}/B_m^{(n)}\C_n$ has a filtration of length $m$ with successive quotients $\HH/R$, acted on through the $m$ deleted color evaluations. No $\I_m^{(n)}$ with $m\ge1$ is finitely generated.
\end{theorem}
\begin{proof}
Under \eqref{eq:mu}, the principal submodule $B_m^{(n)}\C_n$ corresponds to $\rho_{n,m}(\C_n)$. For $m=1$, \cref{thm:pullback,prop:boundary-exact} identify the quotient with $\HH/R$. For $F\in\C_n$ and $G\in\C_{n+2}$,
$\beta_n(\rho_n(F)G)=F(a_n)\beta_n(G)$, which proves the module assertion.

For $m>1$, use the chain of submodules of $\C_{n+2m}$
\[
 \rho_{n,m}(\C_n)\subset
 \rho_{n+2,m-1}(\C_{n+2})\subset\cdots\subset\C_{n+2m}.
\]
Injectivity and composition of restrictions reduce each quotient to the one-step case. The final quotient in this chain is $\HH/R$ with the action through the last deleted color.

The ring $R$ is countable, whereas $\HH$ is uncountable. Indeed, the binary sums
$\sum_{k\ge0}\epsilon_k(q;q)_k$, $\epsilon_k\in\{0,1\}$, are distinct: if two binary sequences first differ at $k$, their Taylor expansions at $q=1$ first differ in degree $k$, with coefficient $\pm(-1)^kk!$. Thus $\HH/R$ is uncountable and cannot be finitely generated over the countable ring $R$. If $\I_m^{(n)}$ were finitely generated over $\C_n$, so would be its quotient and the final quotient $\HH/R$ above. Since the action on the latter factors through $R$, this is impossible.
\end{proof}

In particular, $\C_n$ is non-Noetherian.

\begin{proposition}\label{prop:birational-tails}
Via $\rho_n$, regard $A=\C_n$ as a subring of $B=\C_{n+2}$ and put $u=X-a_n$.
Write $(A:B)=\{b\in B:bB\subseteq A\}$ for the conductor. Then
\[
 A[1/u]=B[1/u],\qquad \operatorname{Frac}(A)=\operatorname{Frac}(B),
 \qquad (A:B)=uB=\rho_n(\I_1^{(n)}).
\]
The extension $A\subset B$ is not finite. More generally,
\[
 \C_n[1/B_m^{(n)}]\simeq
 \C_{n+2m}[1/B_m^{(n)}].
\]
Thus all offsets of the same parity have canonically identified fraction fields.
\end{proposition}
\begin{proof}
For $G\in B$, the coefficient shift gives $uG=\rho_n(\mu_{n,1}G)\in A$. This proves the localization and fraction-field assertions. The same identity with $m$ factors proves the general localization statement.

The inclusion $uB\subset(A:B)$ follows from $A=\beta_n^{-1}(R)$. Conversely, if $bB\subset A$, surjectivity of $\beta_n$ gives $\beta_n(b)\HH\subset R$. In particular $\beta_n(b)\in R$. If this element were nonzero, multiplication by it would inject the uncountable ring $\HH$ into the countable ring $R$, a contradiction; here we use the domain property of $\HH$, which follows from its injective Taylor map \cite[Theorem~5.4]{Habiro04}. Therefore $\beta_n(b)=0$ and $b\in uB$. Finally, $B/A\simeq\HH/R$, with $A$ acting through $R$. If $B$ were finite over $A$, this quotient would be finitely generated over $R$, contradicting the cardinality argument in \cref{thm:defect}.
\end{proof}

\subsection{Taylor expansion at a deleted node}

For a polynomial $P$, its $j$th Hasse derivative at $a$ is
$P^{[j]}(a)=[u^j]P(a+u)$. Sums of Taylor coefficients are taken in
the Habiro topology on $\HH$. Completions by powers of a normal
coordinate carry the adic topology, with $\HH$ discrete.

\begin{lemma}\label{lem:hasse-convergence}
For each $j\ge0$ and $G=\sum_{k\ge0}g_kB_k^{(n+2)}$, the series
\begin{equation}\label{eq:normal-coefficient}
 t_j(G)=\sum_{k\ge0}g_k\bigl(B_k^{(n+2)}\bigr)^{[j]}(a_n)
\end{equation}
converges in $\HH$. The map
\begin{equation}\label{eq:normal-taylor}
 \Tay_n:G\longmapsto\sum_{j\ge0}t_j(G)u^j
\end{equation}
is a ring homomorphism $\C_{n+2}\to\HH[[u]]$, with $t_0=\beta_n$ and $\Tay_n((X-a_n)G)=u\Tay_n(G)$.
\end{lemma}
\begin{proof}
Write $d_i=a_n-x_i^{(n+2)}$, so
$B_k^{(n+2)}(a_n+u)=\prod_{i<k}(d_i+u)$.
Every coefficient of $u^j$ is a sum of products obtained from $\delta_k^{(n)}$ by omitting $j$ factors $d_i$. Each $d_i$ has valuation at most two at every cyclotomic prime. Therefore
\begin{equation}\label{eq:hasse-bound}
 v_{\Phi_d}\bigl((B_k^{(n+2)})^{[j]}(a_n)\bigr)
 \ge v_{\Phi_d}(\delta_k^{(n)})-2j.
\end{equation}
Fix a Habiro factorial $(q;q)_N$. For each of its finitely many
cyclotomic factors, the right side of \eqref{eq:hasse-bound}
eventually exceeds the multiplicity of that factor in $(q;q)_N$, so
by unique factorization all sufficiently late terms lie in
$(q;q)_NR$. The same argument shows that, for fixed $h$ and any fixed
Habiro ideal, all sufficiently late Newton basis terms have their
first $h$ Taylor coefficients in that ideal. Polynomial Taylor
expansion therefore extends to compatible homomorphisms into every
$\HH[u]/(u^h)$, and the remaining identities pass from polynomials to
the limit.
\end{proof}

\begin{theorem}\label{thm:normal-completion}
For every $h\ge1$, the Taylor map induces a ring isomorphism
\begin{equation}\label{eq:normal-quotient}
 \C_{n+2}/(X-a_n)^h\C_{n+2}\xrightarrow{\sim}\HH[u]/(u^h).
\end{equation}
The full map $\Tay_n$ is injective, and
\begin{equation}\label{eq:normal-completion}
 \varprojlim_h\C_{n+2}/(X-a_n)^h\C_{n+2}\simeq\HH[[u]].
\end{equation}
\end{theorem}
\begin{proof}
Surjectivity modulo $u$ is \cref{prop:boundary-exact}. Suppose a polynomial $\alpha_0+\cdots+\alpha_{h-1}u^{h-1}$ with coefficients in $\HH$ has been lifted through normal order $i-1$. Its error in order $i$ is some $e_i\in\HH$. Choose $G_i\in\C_{n+2}$ with $\beta_n(G_i)=e_i$, and add $(X-a_n)^iG_i$ to the lift. Earlier coefficients do not change, and the $i$th one becomes correct. Finite induction proves surjectivity modulo $u^h$.

For the kernel, the case $h=1$ is again \cref{prop:boundary-exact}. If the first $h$ normal coefficients of $G$ vanish, write $G=(X-a_n)H$ by that case. The identity in \cref{lem:hasse-convergence} says that the first $h-1$ normal coefficients of $H$ vanish. Induction gives $G\in(X-a_n)^h\C_{n+2}$.

Finally, if $G$ belongs to every such power, its value at a fixed tail node $x_r^{(n+2)}$ is divisible in $R$ by every power of $x_r^{(n+2)}-a_n$. This is a nonzero nonunit Laurent polynomial of positive breadth. Hence every color value of $G$ is zero, and $G=0$. Taking inverse limits of \eqref{eq:normal-quotient} proves the stated completion.
\end{proof}

Define the trace Chebyshev polynomials by
\begin{equation}\label{eq:chebyshev}
 \Theta_0=2,\qquad \Theta_1=X,\qquad
 \Theta_{m+1}=X\Theta_m-\Theta_{m-1}.
\end{equation}
They satisfy $\Theta_m(z+z^{-1})=z^m+z^{-m}$.

\begin{proposition}\label{prop:adams}
For $d\ge1$, the substitutions
\[
 \psi_d(q)=q^d,\qquad \psi_d(X)=\Theta_d(X)
\]
extend to commuting endomorphisms of $\C_n$, with
$\psi_a\psi_b=\psi_{ab}$ and
$\psi_p(F)\equiv F^p\pmod{p\C_n}$ for every prime $p$.
They commute with tail restriction and with the Habiro boundary maps.
\end{proposition}
\begin{proof}
At each of the first $k$ nodes, $\psi_d(B_k^{(n)})$ vanishes,
since its value is $B_k^{(n)}(x_r^{(n)}(q^d);q^d)$.
Repeated monic division gives
$B_k^{(n)}\mid\psi_d(B_k^{(n)})$ in $R[X]$.
The substitutions are therefore continuous. The Chebyshev identity
proves the composition law. Modulo $p$,
$\Theta_p(X)\equiv X^p$ and $q^p$ is the Frobenius image of $q$;
\cref{thm:units-primes} then gives the completed congruence.
The map $q\mapsto q^d$ is continuous on $\HH$, since the
$\Phi_e(q)$-valuation of $(q^d;q^d)_k$ is
$\lfloor k\gcd(e,d)/e\rfloor\ge\lfloor k/e\rfloor$.
Finally $\Theta_d(a_n(q))=a_n(q^d)$ proves boundary compatibility.
\end{proof}

\subsection{Normal expansions at the earlier nodes}
For $n\ge2$ and $1\le s\le\lfloor n/2\rfloor$, put
\begin{equation}\label{eq:earlier-node}
 b=n-2s,\qquad a_b=q^b+q^{-b},\qquad u_s=X-a_b.
\end{equation}
The Newton products at this point are
\begin{equation}\label{eq:earlier-kernel}
 B_k^{(n)}(a_b+u_s)
 =\prod_{i=0}^{k-1}\bigl(u_s-\{s+i\}\{n-s+i\}\bigr).
\end{equation}

\begin{theorem}\label{thm:all-boundaries}
Evaluation at $a_{n-2s}$ defines a surjective homomorphism
$\beta_{n,s}:\C_n\to\HH$, with kernel $u_s\C_n$.
For every $h\ge1$, normal expansion induces
\begin{equation}\label{eq:all-boundaries}
 \C_n/u_s^h\C_n\simeq\HH[u_s]/(u_s^h),\qquad
 \C_n\hookrightarrow\HH[[u_s]].
\end{equation}
If $F=\sum_{k\ge0}c_kB_k^{(n)}$, its normal coefficients are
\begin{equation}\label{eq:all-normal-coefficients}
 t_{n,s,j}(F)=
 \sum_{k\ge j}(-1)^{k-j}c_k\,
 e_{k-j}\bigl(\{s+i\}\{n-s+i\}:0\le i<k\bigr)\in\HH.
\end{equation}
The $u_s$-adic completion of $\C_n$ is $\HH[[u_s]]$.
\end{theorem}
\begin{proof}
Set $d_k=B_k^{(n)}(a_b)$. By \eqref{eq:earlier-kernel},
$d_k$ is a signed product of two positive strings of quantum integers.
Every irreducible factor is cyclotomic, and the multiplicity of each
fixed cyclotomic factor tends to infinity. The ideals $(d_k)$ are
therefore cofinal with the Habiro ideals, by the proof of
\cref{lem:cofinality}.

For $A_k=R[X]/(B_k^{(n)})$, multiplication by $X-a_b$ gives an exact
sequence
\[
 0\longrightarrow A_k\xrightarrow{X-a_b} A_k
 \longrightarrow R/(d_k)\longrightarrow0.
\]
To see injectivity, tensor with $\Q(q)$: all eigenvalues
$x_i^{(n)}-a_b$ are nonzero. The module $A_k$ is free over $R$,
so the original map is injective as well. The inverse system of
left-hand modules has surjective transition maps. Taking inverse
limits proves surjectivity of $\beta_{n,s}$ and the kernel statement.

A coefficient of $u_s^j$ in \eqref{eq:earlier-kernel} omits $j$ of
the factors defining $d_k$. Each omitted factor has valuation at most
two at a cyclotomic prime. Hence its valuation is at least
$v_{\Phi_e}(d_k)-2j$. For fixed $j$, these coefficients tend to zero
in the Habiro topology. Expanding the elementary symmetric functions
gives \eqref{eq:all-normal-coefficients} and a multiplicative Taylor
map. Successive lifts through $\beta_{n,s}$ prove surjectivity on
each finite normal quotient. Repeated division by $u_s$ identifies
its kernel with $u_s^h\C_n$.

If an element lies in every $u_s^h\C_n$, its value at each color is
divisible by every power of the nonzero nonunit
$x_r^{(n)}-a_b=\{r+s\}\{n+r-s\}$. Such a Laurent polynomial is zero.
The color identity theorem proves injectivity of the full map.
\end{proof}

\begin{proposition}\label{prop:normal-adams}
At the boundary \eqref{eq:earlier-node}, the action of $\psi_d$ on
normal degree $j$ is
\begin{equation}\label{eq:normal-adams}
 h(q)u_s^j\longmapsto w_{d,b}(q)^j h(q^d)u_s^j,
 \qquad
 w_{d,b}(q)=
 \begin{cases}
 d\displaystyle\frac{q^{db}-q^{-db}}{q^b-q^{-b}},&b>0,\\[2mm]
 d^2,&b=0.
 \end{cases}
\end{equation}
For $b>0$ the displayed quotient is a Laurent polynomial.
\end{proposition}
\begin{proof}
The full Taylor substitution is
\[
 \Tay_{n,s}(\psi_dF)(u_s,q)
 =\bigl(\Tay_{n,s}F\bigr)
 \bigl(\Theta_d(a_b(q)+u_s)-a_b(q^d),q^d\bigr).
\]
Differentiating $\Theta_d(z+z^{-1})=z^d+z^{-d}$ gives
$\Theta_d'(a_b)=w_{d,b}$. At $b=0$ its value is $d^2$.
Taking the associated graded proves the formula.
\end{proof}

The earlier nodes \eqref{eq:earlier-node} have the parity of $n$: for
even $n$ the last one is $2$, and for odd $n$ it is $q+q^{-1}$. At
odd $n$, evaluation at $X=2$ is not continuous for the Newton and
Habiro topologies: at $q=-1$ one has $B_k^{(n)}(2)=4^k$. The joint
formal expansion at $(q,X)=(1,2)$ exists for every $n$ by
\cref{thm:quantum-neighborhoods}.

\begin{example}[Geometric nodes]\label{ex:geometric-boundary}
Let $S=\mathbb Z[Q^{\pm1}]$, let
$\mathcal H_Q=\varprojlim_m S/((Q;Q)_m)$, and put
\[
 \mathcal G_b=\varprojlim_m S[Z^{\pm1}]/(P_m^{(b)}),\qquad
 P_m^{(b)}(Z)=\prod_{j=0}^{m-1}(Z-Q^{b+j}),\qquad b\ge0.
\]
The constant coefficient of $P_m^{(b)}$ is a unit, so the same
finite quotients can be formed in $S[Z]$. At the deleted node,
\[
 P_k^{(b+1)}(Q^b)=Q^{bk}(Q;Q)_k.
\]
Thus the boundary factorial is exactly the Habiro factorial, up to
an $S$-unit. The two-ideal argument of \cref{thm:pullback} gives
an injective tail restriction and a fiber product
\[
 \mathcal G_b\simeq S\times_{\mathcal H_Q}\mathcal G_{b+1}.
\]
Evaluation $\gamma_b$ at $Q^b$ is surjective, with
$\ker\gamma_b=v\mathcal G_{b+1}$, where $v=Z-Q^b$.

The coefficient of $v^j$ in $P_k^{(b+1)}(Q^b+v)$ has
$\Phi_d(Q)$-valuation at least $\lfloor k/d\rfloor-j$.
It therefore tends to zero in the Habiro topology for fixed $j$.
Successive lifting through $\gamma_b$, followed by division by $v$,
gives
\[
 \mathcal G_{b+1}/v^h\mathcal G_{b+1}
       \simeq\mathcal H_Q[v]/(v^h),\qquad
 \widehat{\mathcal G_{b+1}}^{\,(v)}\simeq\mathcal H_Q[[v]].
\]
The completion map is injective: an element in every power of $v$
has each node value divisible by every power of its nonzero
Laurent difference from $Q^b$. All its values vanish, and the
Laurent-breadth proof of \cref{thm:identity} applies.
\end{example}

\section{Cyclotomic and \texorpdfstring{$p$}{p}-adic interpolation}\label{sec:sampling}

Fix $n\ge0$ and $d\ge1$, set $\pi=\Phi_d(q)$ and
$\ell=d/\gcd(d,2)$, and let $\cO=R_{(\pi)}$ with completion
$\widehat\cO$. Define
\[
 \Omega_{n,d}=(\Z/\ell\Z)/\langle r\mapsto-n-r\rangle.
\]
The ring $\cO$ is a discrete valuation ring with residue field
$\Q(\zeta_d)$, where $\zeta_d$ is a primitive $d$th root of unity,
and every nonzero integer is a unit in $\cO$.

For distinct nonnegative colors,
\begin{equation}\label{eq:pair-valuation}
 v_\pi(x_r^{(n)}-x_s^{(n)})
   =\boldsymbol 1_{\ell\mid r-s}+\boldsymbol 1_{\ell\mid n+r+s}.
\end{equation}
Indeed, each nonzero quantum integer has a simple zero at a cyclotomic prime which divides it. Thus two nodes have the same residue precisely when their colors lie in the same orbit $\omega\in\Omega_{n,d}$. For distinct colors in that orbit, the valuation in \eqref{eq:pair-valuation} is a constant $e_\omega$, equal to two for a fixed orbit and one otherwise. The number of fixed orbits is
\begin{equation}\label{eq:fixed-count}
 f_{n,d}=\begin{cases}
 1&\ell\text{ odd},\\
 2&\ell\text{ even and }n\text{ even},\\
 0&\ell\text{ even and }n\text{ odd}.
 \end{cases}
\end{equation}
There are $(\ell-f_{n,d})/2$ nonfixed orbits.

\subsection{Local completions}

\begin{proposition}\label{prop:quadratic-cover}
Set $z_r=q^{n+2r}$ and
$\widetilde B_m(z)=\prod_{r<m}(z-z_r)(z-z_r^{-1})$.
Then
\[
 \varprojlim_m R[z^{\pm1}]/(\widetilde B_m)
 \simeq \C_n[z]/(z^2-Xz+1).
\]
The invariant subring for $z\mapsto z^{-1}$ is $\C_n$.
\end{proposition}
\begin{proof}
The identity
$\widetilde B_m(z)=z^mB_m^{(n)}(z+z^{-1})$
and the presentation
$R[z^{\pm1}]=R[X,z]/(z^2-Xz+1)$
identify the finite quotients. They are free of rank two over the corresponding unoriented quotient, with basis $1,z$, so the presentation passes to the inverse limit. The involution sends $z$ to $X-z$. If $a+bz$ is fixed, comparison of the coefficients of $z$ gives $2b=0$. The module $\C_n$ is torsion-free over $\Z$, so $b=0$.
\end{proof}

\begin{proposition}\label{prop:disks}
Choose a color $s_\omega$ in each orbit, and put $u_\omega=X-x_{s_\omega}^{(n)}$. The mixed completion has a decomposition
\begin{equation}\label{eq:disks}
 \varprojlim_{h,m}\cO[X]/(\pi^h,B_m^{(n)})
 \simeq\prod_{\omega\in\Omega_{n,d}}\widehat\cO[[u_\omega]].
\end{equation}
The map from $\C_n$ into any one factor is injective.
\end{proposition}
\begin{proof}
Group the roots of each $B_m^{(n)}$ by their residues. The resulting factors belonging to different orbits have unit resultants and are comaximal. Finite Chinese remainder therefore separates the orbits.

Within one orbit, translate by $x_{s_\omega}^{(n)}$. The translated
roots $a_i$ lie in $\pi\widehat\cO$. Put
$g_c(u)=\prod_{i<c}(u-a_i)$ and $\mathfrak m=(\pi,u)$.
If $h,c\ge N$, then $(\pi^h,g_c)\subseteq\mathfrak m^N$.
Conversely, the monic relation modulo $g_c$ gives $u^c\in\pi(\widehat\cO[u]/(g_c))$,
so $u^{ch}\in(\pi^h,g_c)$ and
\[
 \mathfrak m^{(c+1)h}\subseteq(\pi^h,g_c).
\]
The two systems of ideals are therefore cofinal.
Every orbit contains infinitely many colors, so its limit is
$\widehat\cO[[u_\omega]]$ with the $(\pi,u_\omega)$-adic topology.

If the image of $F\in\C_n$ in one disk is zero, evaluate that disk series at $u_\omega=x_r^{(n)}-x_{s_\omega}^{(n)}\in\pi\widehat\cO$ for every color in that orbit. Each value $F(x_r)$ vanishes in $\widehat\cO$, and hence in $R$. The orbit is infinite, so \cref{thm:identity} gives $F=0$.
\end{proof}

At a fixed orbit, $q^{n+2r}$ reduces to $z=\pm1$, where the
derivative of $z+z^{-1}$ vanishes; this is the source of the
valuation two in \eqref{eq:pair-valuation}.

\subsection{Values modulo a cyclotomic power}

After reduction modulo a node difference, the values at the nodes no
longer determine an element; the two-node algebra already shows this.

\begin{proposition}
\label{prop:collision}
Let $A$ be a domain, $a,b\in A$, and $\delta=b-a$ a nonzero nonunit.
For $V=A[X]/((X-a)(X-b))$, evaluation fits into the exact sequence
\[
 0\longrightarrow V\longrightarrow A^2
 \xrightarrow{(x,y)\mapsto y-x\bmod\delta}A/(\delta)
 \longrightarrow0.
\]
The evaluation matrix in the basis $1,X-a$ is
$\left(\begin{smallmatrix}1&0\\1&\delta\end{smallmatrix}\right)$.
After reduction modulo $\delta$, the class of $X-a$ is nonzero,
square-zero, and has value zero at both nodes. It is precisely the
image of $\operatorname{Tor}_1^A(A/(\delta),A/(\delta))\simeq A/(\delta)$
in the specialized exact sequence.
\end{proposition}
\begin{proof}
The two values of $s+t(X-a)$ are $(s,s+t\delta)$. This proves the
exact sequence and the matrix formula. The specialized algebra is
$(A/(\delta))[u]/(u^2)$, and evaluation sends $u$ to zero. The
resolution $0\to A\xrightarrow{\delta}A\to A/(\delta)\to0$
computes the stated Tor module; tensoring the displayed evaluation
sequence identifies its image with this kernel.
\end{proof}

For the quantum nodes, \eqref{eq:pair-valuation} gives the order of
these collisions in each residue class, and with it the exact number
of samples needed at a prescribed cyclotomic precision.

\begin{theorem}\label{thm:sampling}
Let $h\ge1$ and $t_\omega=\lceil h/e_\omega\rceil$. Choose any $t_\omega$ distinct nonnegative colors in each orbit. Their values modulo $\pi^h\cO$ determine all color values modulo $\pi^h\cO$ of every element of $\C_n$.

Conversely, any finite set of exact color samples that universally determines all color values modulo $\pi^h\cO$ must contain at least $t_\omega$ colors in every orbit. Thus the minimum sample count is
\begin{equation}\label{eq:sample-number}
 M_{n,d}(h)=\frac{\ell-f_{n,d}}2\,h
                  +f_{n,d}\left\lceil\frac h2\right\rceil.
\end{equation}
\end{theorem}
\begin{proof}
Fix an orbit, with chosen colors $r_0,\ldots,r_{t-1}$, where $t=\lceil h/e_\omega\rceil$. For a target color $r$ in that orbit, interpolation gives
\begin{equation}\label{eq:stable-reconstruction}
 F(x_r)\equiv\sum_{j=0}^{t-1}F(x_{r_j})
 \prod_{\substack{0\le i<t\\i\ne j}}
 \frac{x_r-x_{r_i}}{x_{r_j}-x_{r_i}}
 \pmod{\pi^h\cO}.
\end{equation}
When the target is sampled this is exact. Otherwise, every numerator and denominator factor in a Lagrange coefficient has valuation $e_\omega$. Hence each coefficient lies in $\cO$, in fact is a unit. By \cref{prop:ambiguity}, the error is divisible by $\prod_{i<t}(x_r-x_{r_i})$, whose valuation is $e_\omega t\ge h$. This proves \eqref{eq:stable-reconstruction} and shows that no extra precision in the sample values is required.

For necessity, suppose a sample set $T$ contains $t<t_\omega$ colors in an orbit. The zero polynomial and
$P_T(X)=\prod_{s\in T}(X-x_s)$
have identical exact samples. At an unsampled color $r$ in the deficient orbit, \eqref{eq:pair-valuation} gives
$v_\pi(P_T(x_r))=e_\omega t<h$:
factors from other orbits are units. Their target values therefore differ modulo $\pi^h$. This proves the lower bound for arbitrary, not necessarily linear, recovery rules.
\end{proof}

Evaluation modulo $\pi^h$ has a larger kernel than $\pi^h\C_n$. To
describe its image, define the algebra of polynomial value functions
by
\[
 \mathscr V_{n,d,h}=\im\left(\cO[X]\longrightarrow
                  \prod_{r\ge0}\cO/(\pi^h)\right).
\]
For one orbit $\omega$, choose $t=\lceil h/e_\omega\rceil$ sample nodes
$s_0,\ldots,s_{t-1}$ and write
$P_j=\prod_{i<j}(X-s_i)$.

\begin{theorem}\label{thm:value-algebra}
There is an isomorphism of $\cO$-modules
\begin{equation}\label{eq:value-module}
 \mathscr V_{n,d,h}\simeq
 \bigoplus_{\omega\in\Omega_{n,d}}
 \ \bigoplus_{j=0}^{\lceil h/e_\omega\rceil-1}
 \cO/(\pi^{h-e_\omega j}).
\end{equation}
Its length is
\begin{equation}\label{eq:value-length}
 \length_\cO\mathscr V_{n,d,h}
 =\frac{\ell-f_{n,d}}2\frac{h(h+1)}2
  +f_{n,d}\left\lfloor\frac{(h+1)^2}{4}\right\rfloor.
\end{equation}
For one orbit, with $e=e_\omega$, the kernel of evaluation modulo $\pi^h$ on all its colors is
\begin{equation}\label{eq:value-ideal}
 \bigl(P_t,\ \pi^{h-ej}P_j\;:\;0\le j<t\bigr)\subset\cO[X].
\end{equation}
\end{theorem}
\begin{proof}
The interpolation argument of \cref{thm:sampling} works over $\cO$ as
well as over $R$. Thus a polynomial vanishes modulo $\pi^h$ on every
color if and only if it does so on the complete sample set.
Let $P_T$ be the product of the corresponding linear factors.
It vanishes modulo $\pi^h$ at every color, so division by the monic
polynomial $P_T$ reduces the computation to $\cO[X]/(P_T)$.

Write $P_T=\prod_\omega P_\omega$, where $P_\omega$ is the product
for the samples in orbit $\omega$. Factors from different orbits have
unit resultants. Hence
\[
 \cO[X]/(P_T)\simeq\prod_\omega\cO[X]/(P_\omega).
\]
Under this isomorphism, the selected evaluations in an orbit depend
only on its factor. The value algebra therefore splits into the
images of these orbit factors.

Fix an orbit and put $e=e_\omega$. In its Newton basis, the evaluation
matrix is $M_{ij}=P_j(s_i)$, for $0\le i,j<t$.
It is lower triangular. For $i\ge j$, each of the $j$ factors of
$P_j(s_i)$ has valuation $e$, so
\[
 M=U\operatorname{diag}(1,\pi^e,\ldots,\pi^{e(t-1)}),
 \qquad U\in\operatorname{GL}_t(\cO).
\]
It follows that the image modulo $\pi^h$ is, as a module,
\[
 \bigoplus_{j=0}^{t-1}\pi^{ej}\cO/\pi^h\cO
 \simeq\bigoplus_{j=0}^{t-1}\cO/(\pi^{h-ej}).
\]
This proves \eqref{eq:value-module}.

A remainder $\sum_{j<t}a_jP_j$ has zero values modulo $\pi^h$ if and
only if $a_j\in\pi^{h-ej}\cO$ for every $j$.
Moreover, $P_t$ has zero values on the whole orbit modulo $\pi^h$.
Dividing an arbitrary polynomial by $P_t$ gives
\eqref{eq:value-ideal}.
Finally,
\[
 \sum_{j=0}^{h-1}(h-j)=\frac{h(h+1)}2,
 \qquad
 \sum_{j=0}^{\lceil h/2\rceil-1}(h-2j)
       =\left\lfloor\frac{(h+1)^2}{4}\right\rfloor.
\]
Summing over the orbits proves \eqref{eq:value-length}.
\end{proof}

The same algebra is the image of $(R\setminus(\pi))^{-1}\C_n$ under
evaluation modulo $\pi^h$. Indeed, an element of $\C_n$ agrees with
its integral interpolant on the complete sample set, hence at every
color modulo $\pi^h$. Division by an element of $R\setminus(\pi)$
preserves this agreement. The reverse inclusion follows from
$\cO[X]\subseteq(R\setminus(\pi))^{-1}\C_n$.

\begin{example}\label{ex:cubic-sampling}
For $n=3$ and $d=3$ or $6$, one has $\ell=3$. The orbits of colors are $\{1,2\}$ and $\{0\}$ modulo three, with valuations one and two. Hence
\[
 M_{3,d}(h)=h+\lceil h/2\rceil.
\]
At precision $\pi^2$, the colors $r=0,1,2$ suffice for every color value, and the evaluated algebra has module
$\cO/(\pi^2)\oplus\cO/(\pi)\oplus\cO/(\pi^2)$.
At precision $\pi^3$, one may use $r=1,2,4$ from the nonfixed orbit and $r=0,3$ from the fixed orbit. The sample count is five and the module length is ten.
\end{example}

\subsection{\texorpdfstring{$p$}{p}-adic interpolation}\label{sec:padic}

Let $p$ be a rational prime. We write $\Zp\langle H,U\rangle$ for
the ring of restricted power series: the coefficients lie in $\Zp$
and tend to zero as total degree tends to infinity.
Its Gauss norm is the supremum of the $p$-adic absolute values of the
coefficients. Binomial polynomials are integer-valued on $\Zp$,
as in Mahler interpolation \cite{Mahler58}.

\begin{theorem}\label{thm:analytic-bidisc}
For every $n\ge0$, the substitutions $q=1+\eta$ and $X=2+u$ define an injective ring homomorphism
\begin{equation}\label{eq:integral-taylor}
 \C_n\longrightarrow\Z[[\eta,u]].
\end{equation}
For every rational prime $p$ and integer $a\ge1$, the substitutions
\begin{equation}\label{eq:analytic-bidisc}
 q=1+p^aH,\qquad X=2+p^{2a}U
\end{equation}
define an injective ring homomorphism
$\C_n\longrightarrow\Zp\langle H,U\rangle$.
The image of the $k$th Newton summand $c_k(q)B_k^{(n)}(X)$ has Gauss norm at most $p^{-2ak}$.
\end{theorem}
\begin{proof}
For every nonnegative integer $b$,
\[
 (1+\eta)^b+(1+\eta)^{-b}-2
 =(1+\eta)^{-b}\bigl((1+\eta)^b-1\bigr)^2\in \eta^2\Z[[\eta]].
\]
Give $\eta$ weight one and $u$ weight two. Each factor of
$B_k^{(n)}(2+u;1+\eta)$ has weight at least two, so the $k$th summand has weight at least $2k$. Its coefficient $c_k(1+\eta)$ lies in $\Z[[\eta]]$. The series therefore converges in $\Z[[\eta,u]]$, and polynomial density gives multiplicativity.

Substitute $u=x_r^{(n)}(1+\eta)-2$ in its image. This is a well-defined formal substitution since the substituted series has zero constant term, and the resulting series is the Taylor expansion of the Laurent polynomial $F(x_r^{(n)})$. If the image of $F$ is zero, all these Laurent polynomials have zero Taylor expansion, hence are zero. The color identity theorem gives $F=0$.

Under \eqref{eq:analytic-bidisc}, every integral Laurent polynomial in $q$ has Gauss norm at most one. Furthermore,
\[
 x_i^{(n)}(1+p^aH)-2\in p^{2a}\Zp\langle H\rangle.
\]
Thus every Newton factor has norm at most $p^{-2a}$. Completeness of the restricted power-series algebra proves convergence and the asserted norm bound. The resulting formal series is obtained from \eqref{eq:integral-taylor} by $\eta\mapsto p^aH$, $u\mapsto p^{2a}U$ and the injection $\Z\hookrightarrow\Zp$. This rescaling is injective on formal coefficients, so the analytic map is injective.
\end{proof}

Fix now an odd prime $p$ and an element $q_0\in1+p\Zp$, $q_0\ne1$. Put $a=v_p(q_0-1)$ and define
\[
 q_0^t=\exp(t\log q_0),\qquad
 x_n(t)=q_0^{n+2t}+q_0^{-n-2t}\quad(t\in\Zp).
\]
Since $v_p(\log q_0)=a$, the coefficient of $t^j$ in $q_0^t$ has
valuation $aj-v_p(j!)$, which tends to infinity. Thus $q_0^t$ is a
restricted power series. For $z\ne0$ in $\Zp$,
\begin{equation}\label{eq:padic-brace}
 v_p(q_0^z-q_0^{-z})=a+v_p(z).
\end{equation}
Indeed, $v_p(\log q_0)=a$ and $v_p(\exp(w)-1)=v_p(w)$ for $w\in p\Zp$; apply this to $w=2z\log q_0$. 

\begin{theorem}\label{thm:padic-color}
Let $F=\sum_{k\ge0}c_k(q)B_k^{(n)}(X)\in\C_n$. Then
\begin{equation}\label{eq:padic-function}
 f(t)=\sum_{k\ge0}c_k(q_0)B_k^{(n)}(x_n(t);q_0)
\end{equation}
belongs to $\Zp\langle t\rangle$ and is the unique continuous extension of the color sequence $F(x_r^{(n)};q_0)$ from $\N$ to $\Zp$. It satisfies
\begin{align}
 f(t)&=f(-n-t),\label{eq:padic-reflection}\\
 v_p(f(t)-f(s))&\ge2a+v_p(t-s)+v_p(n+t+s).\label{eq:padic-lipschitz}
\end{align}
For every $m\ge0$, the first $m$ Newton summands give the uniform value estimate
\begin{equation}\label{eq:factorial-error}
 v_p\left(f(t)-\sum_{k=0}^{m-1}c_k(q_0)B_k^{(n)}(x_n(t);q_0)\right)
 \ge 2am+2v_p(m!)\qquad(t\in\Zp).
\end{equation}
In particular, every odd derivative of $f$ vanishes at $-n/2$.
\end{theorem}
\begin{proof}
The exponentials $q_0^{t-i}-q_0^{-t+i}$ and
$q_0^{n+t+i}-q_0^{-n-t-i}$ belong to $p^a\Zp\langle t\rangle$. Their product is $x_n(t)-x_i^{(n)}(q_0)$. Since $c_k(q_0)\in\Zp$, the $k$th summand of \eqref{eq:padic-function} has Gauss norm at most $p^{-2ak}$; hence the series converges in the restricted algebra. At $t=r\in\N$, all terms with $k>r$ vanish. Density of $\N$ in $\Zp$ proves uniqueness, and $x_n(-n-t)=x_n(t)$ proves reflection.

Let $F_M$ be a finite partial sum. Polynomial divisibility gives
$F_M(x_n(t))-F_M(x_n(s))\in(x_n(t)-x_n(s))\Zp$.
When $x_n(t)\ne x_n(s)$, divide by this fixed nonzero scalar and
pass to the limit. Since $\Zp$ is closed in $\Qp$, the limiting
quotient remains in $\Zp$. By \eqref{eq:padic-brace},
\[
 v_p(x_n(t)-x_n(s))=2a+v_p(t-s)+v_p(n+t+s).
\]
If $x_n(t)=x_n(s)$, then $t=s$ or $t=-n-s$, and the full difference is zero by reflection. This proves \eqref{eq:padic-lipschitz} in all cases.

Factor the first $m$ Newton factors out of the remainder. The remaining series is bounded by one, because its terms are products of the factors with indices $i\ge m$, with integral coefficients, and its Gauss norms tend to zero. The valuation of the first product is
\[
 2am+v_p\left(\prod_{i=0}^{m-1}(t-i)\right)
       +v_p\left(\prod_{i=0}^{m-1}(n+t+i)\right).
\]
Each of the two products divided by $m!$ is integer-valued on $\Zp$: this follows for the first from the binomial polynomial $\binom{t}{m}$, and for the second from $\binom{n+t+m-1}{m}$, by continuity and density of integer arguments. This gives \eqref{eq:factorial-error}. Differentiating \eqref{eq:padic-reflection} an odd number of times at its fixed point proves the last assertion.
\end{proof}

\begin{remark}
The error bound \eqref{eq:factorial-error} concerns values on $\Zp$;
for the Gauss norm of the tail one only has the bound $p^{-2am}$.
Unlike the map of \cref{thm:analytic-bidisc}, the specialization
$q=q_0$ need not be injective on $\C_n$: for $q_0=1+p$ it kills
$q-(1+p)$.
\end{remark}

\section{Rank-uniform Newton rings and quantum neighborhoods}\label{sec:universal}

We now let the rank vary by replacing $q^n$ with an independent
variable $A$.

\subsection{Integer-valued rank coefficients}

The nodes $Aq^{2r}+A^{-1}q^{-2r}$ specialize to $x_r^{(n)}$ at
$A=q^n$.  The coefficient ring of the rank-uniform construction
\cite{FZ26} is the integer-valued Laurent ring
\begin{equation}\label{eq:rank-integer-values}
 \mathscr I_q=\{f(A)\in\Q(q)[A^{\pm1}]: f(q^n)\in R\text{ for every }n\ge2\}.
\end{equation}
For integers $a\ge k\ge0$, write
\[
 \gaussbin{a}{k}{q}
 =\frac{(q;q)_a}{(q;q)_k(q;q)_{a-k}}
\]
for the ordinary Gaussian polynomial, and set it equal to zero when
$k>a$. Put $\mathscr I_q^+=\mathscr I_q\cap\Q(q)[A]$ and define
\[
 b_j(A)=\frac{\prod_{i=0}^{j-1}(1-Aq^{-2-i})}{(q;q)_j},\qquad b_0=1.
\]

\begin{proposition}[{\cite[Theorem~8.6]{FZ26}}]\label{prop:integer-valued-rank}
One has
\begin{equation}\label{eq:rank-basis}
 \mathscr I_q^+=\bigoplus_{j\ge0}R b_j(A),\qquad \mathscr I_q=\mathscr I_q^+[A^{-1}].
\end{equation}
For every $m\ge1$, evaluation at the consecutive ranks $2,3,\ldots,m+1$ gives a surjection $\mathscr I_q^+\to R^m$.
\end{proposition}
\begin{proof}
For $n\ge0$,
\[
 b_j(q^{n+2})=\gaussbin{n}{j}{q},
\]
with value zero for $j>n$.
Since $\deg_A b_j=j$, the polynomials $b_0,\ldots,b_d$ form a
$\Q(q)$-basis in degree at most $d$.
Their evaluation matrix at $A=q^2,\ldots,q^{d+2}$ is lower
unitriangular over $R$.
If a polynomial $f$ of degree $d$ is integral at these nodes,
solving this triangular system gives its $b_j$-coefficients in $R$.
Conversely, every $b_j$ is integral at every allowed node.
This proves the first equality.
The same triangular matrix for $b_0,\ldots,b_{m-1}$ proves
surjectivity onto $R^m$.
Multiplication by a sufficiently large power of $A$ reduces a Laurent
polynomial to this case; all $q^n$ are units. This proves the second equality.
\end{proof}

In the variable $(Aq^{-2}-1)/(q-1)$, this is the quantum
integer-valued basis of Harman--Hopkins \cite{HH17}. Integrality at
every rank does not make differences of rank values divisible by
$q-1$: for example,
\[
 b_1(q^2)=0,\qquad b_1(q^3)=1.
\]

For expansions in the rank near $q=1$, let $N$ be an indeterminate
and put
\[
 B=\operatorname{Int}(\mathbb Z)
   =\{P(N)\in\mathbb Q[N]:P(\mathbb Z)\subseteq\mathbb Z\}.
\]
The polynomials $\binom{N-2}{j}$, $j\ge0$, form a $\mathbb Z$-basis
of $B$. Put $\eta=q-1$.

\begin{proposition}\label{prop:rank-coefficient-jets}
The substitution $A=(1+\eta)^N$, $q=1+\eta$ defines an injective
ring homomorphism
\[
 \mathscr I_q\longrightarrow B[[\eta]].
\]
For every $s\ge1$, it induces a compatible isomorphism
\begin{equation}\label{eq:rank-coefficient-jets}
 \mathscr I_q/\eta^s\mathscr I_q
       \simeq B[\eta]/(\eta^s).
\end{equation}
In particular, $\widehat{\mathscr I_q}^{\,\eta}\simeq B[[\eta]]$,
and the image of $b_j(A)$ modulo $\eta$ is $\binom{N-2}{j}$.
\end{proposition}
\begin{proof}
Write $f\in\mathscr I_q$ as a finite Laurent polynomial in $A$
with coefficients in $\mathbb Q(q)$. Substitution gives a Laurent
series in $\eta$ with coefficients in $\mathbb Q[N]$ and with a
finite lower bound on its powers of $\eta$. For each integer
$n\ge2$, its specialization at $N=n$ is the Taylor expansion of
$f(q^n,q)\in R$. Every negative-order coefficient therefore
vanishes at all $n\ge2$, so it is the zero polynomial. Each
nonnegative-order coefficient $P(N)$ takes integer values at
these same integers. Its finite-difference expansion at $2$,
\[
 P(N)=\sum_{j=0}^{\deg P}(\Delta_N^jP)(2)\binom{N-2}{j},
 \qquad \Delta_NP(N)=P(N+1)-P(N),
\]
has integer coefficients, proving $P\in B$.

At order zero, the product defining $b_j$ gives
\[
 b_j((1+\eta)^N,1+\eta)
        =\binom{N-2}{j}+O(\eta).
\]
Hence every element of $B$ lifts to $\mathscr I_q$. If the
constant Taylor coefficient of $f$ vanishes, each Laurent
polynomial $f(q^n,q)$ is divisible by $q-1$ in $R$.
It follows directly from the definition of $\mathscr I_q$ that
$f/(q-1)\in\mathscr I_q$. Repeating this argument identifies the
kernel modulo $\eta^s$ with $\eta^s\mathscr I_q$.

To lift a prescribed polynomial in $\eta$ with coefficients in
$B$, first lift its constant coefficient using the $b_j$.
Subtract the Taylor image of this lift and divide the target
remainder by $\eta$. Repeating through order $s-1$ proves
surjectivity in \eqref{eq:rank-coefficient-jets}. Finally, a zero
full Taylor image forces every $f(q^n,q)$ to vanish. After multiplication by a power
of $A$, the element $f$ is a polynomial over $\mathbb Q(q)$ with
infinitely many distinct roots $q^n$, and is therefore zero.
Taking inverse limits gives the completion statement.
\end{proof}

Under the substitution $A=(1+\eta)^N$, the divided-difference
coordinate satisfies
\begin{equation}\label{eq:classical-rank-coordinate}
 \frac{Aq^{-2}-1}{q-1}\equiv N-2\pmod{\eta},
\end{equation}
so at $q=1$ the rank enters the coefficients through the variable $N$
of $\operatorname{Int}(\mathbb Z)$.

\subsection{The rank-uniform Newton ring}

The rank-uniform color completion of \cite[Section~8.4]{FZ26} is
\begin{equation}\label{eq:universal-ring}
 \mathcal U=\varprojlim_m\mathscr I_q[X]\Big/
 \prod_{r=0}^{m-1}\bigl(X-Aq^{2r}-A^{-1}q^{-2r}\bigr).
\end{equation}

\begin{proposition}\label{prop:rank-specialization}
For every $n\ge2$, specialization $A=q^n$ gives a continuous surjection $\operatorname{sp}_n:\mathcal U\to\C_n$, and
\[
 \mathcal U\longrightarrow\prod_{n\ge2}\C_n
\]
is injective. If $n,n'\ge2$ are distinct, there is no ring homomorphism $T:\C_n\to\C_{n'}$ such that $\operatorname{sp}_{n'}=T\circ\operatorname{sp}_n$ on all of $\mathcal U$.
\end{proposition}
\begin{proof}
Specialization sends the universal nodes to $x_r^{(n)}$. Any coefficient sequence in $R$ lifts as the same constant sequence in $\mathscr I_q$, proving surjectivity. If all specializations of an element vanish, uniqueness of Newton coordinates implies that each coefficient $c_k(A)\in \mathscr I_q$ vanishes at all $A=q^n$. A Laurent polynomial over $\Q(q)$ with infinitely many distinct zeros is zero, so all $c_k$ vanish. For the last assertion, the element $A-q^n$ is killed by $\operatorname{sp}_n$, but its image under $\operatorname{sp}_{n'}$ is $q^{n'}-q^n\ne0$.
\end{proof}

Differences between rank specializations satisfy the following
divisibility. Let $f\in\mathscr I_q$, let $b(q)\in R\setminus\{0\}$
satisfy $b(q)f(A,q)\in R[A^{\pm1}]$, and let $n\ne M$ be at least two.
Then, in the UFD $R$,
\[
 \frac{q^{|n-M|}-1}{\gcd(q^{|n-M|}-1,b(q))}
 \ \text{divides}\ f(q^n,q)-f(q^M,q).
\]
This follows by multiplying the difference by $b$ and comparing
irreducible valuations.

\subsection{Formal quantum neighborhoods}

Put $u=X-2$. In both the fixed-rank and the rank-uniform rings, the
Newton nodes have constant term $2$ and differ from $2$ by terms of
order at least two in $\eta$.

\begin{theorem}\label{thm:quantum-neighborhoods}
For every $n\ge0$ and $s\ge1$, Taylor substitution induces
compatible ring isomorphisms
\begin{align}
 \mathcal C_n/\eta^s\mathcal C_n
   &\simeq(\mathbb Z[\eta]/(\eta^s))[[u]],
             \label{eq:fixed-quantum-neighborhood}\\
 \mathcal U/\eta^s\mathcal U
   &\simeq(B[\eta]/(\eta^s))[[u]],
             \label{eq:universal-quantum-neighborhood}
\end{align}
where the universal substitution is
$X=2+u$, $A=(1+\eta)^N$, $q=1+\eta$.
The original rings embed in their $\eta$-adic completions, and
\begin{equation}\label{eq:quantum-completions}
 \widehat{\mathcal C_n}^{\,\eta}\simeq\mathbb Z[[u,\eta]],
 \qquad
 \widehat{\mathcal U}^{\,\eta}\simeq B[[u,\eta]].
\end{equation}
For $n\ge2$, rank specialization in these coordinates is
$N\mapsto n$. The coefficient $[u^j\eta^m]$ depends only on
Newton indices $k\le j+\lfloor m/2\rfloor$.
\end{theorem}
\begin{proof}
The shifted fixed-rank nodes satisfy
\[
 b_i^{(n)}(\eta)=(1+\eta)^{n+2i}+(1+\eta)^{-n-2i}-2
                 \in\eta^2\mathbb Z[[\eta]].
\]
For the universal nodes, the same expression with $n$ replaced
by $N$ lies in $\eta^2B[[\eta]]$, by
\cref{prop:rank-coefficient-jets}. In the product
$\prod_{i<k}(u-b_i)$, a term of degree $j$ in $u$ contains
$k-j$ shifted nodes and therefore has $\eta$-order at least
$2(k-j)$. This proves the finite-dependence assertion and
constructs both Taylor maps. Finite polynomial approximation
with this bound shows that they preserve multiplication.

Modulo $\eta$, the fixed-rank map is
\[
 \sum_{k\ge0}c_kB_k^{(n)}\longmapsto
                       \sum_{k\ge0}c_k(1)u^k.
\]
It is surjective, and its kernel is $\eta\mathcal C_n$, by
coefficientwise division in $R$. For $\mathcal U$, the
constant Taylor map on each coefficient has image $B$ and
kernel $\eta\mathscr I_q$ by
\cref{prop:rank-coefficient-jets}. Arbitrary sequences of
coefficients are allowed in a Newton completion. Thus the
universal map is also surjective modulo $\eta$, with kernel
$\eta\mathcal U$.

These two order-zero assertions lift to every $s$. Given a
power series in $u$ with coefficients modulo $\eta^s$, lift
its constant $\eta$-coefficient as a Newton series. Subtract
that Taylor image, divide the remaining target by $\eta$, and
repeat. After $s$ steps the finite sum of lifts has the
prescribed image. For the kernel, an element with zero image
modulo $\eta^s$ is first divisible by $\eta$ using the
order-zero case. Its quotient has zero image modulo
$\eta^{s-1}$, so induction gives divisibility by $\eta^s$.
This proves \eqref{eq:fixed-quantum-neighborhood} and
\eqref{eq:universal-quantum-neighborhood}.

Taking the compatible inverse limit over $s$ proves
\eqref{eq:quantum-completions}. An element in every power of
$\eta$ has each Newton coefficient in every such power of its
coefficient ring. The intersection is zero in $R$, and also
in $\mathscr I_q$ by its injective Taylor map. Hence both
original rings embed in these completions. All the
constructions commute with $N\mapsto n$.
\end{proof}

\part{Knot Elements}

\section{Knot elements and rank--color relations}\label{sec:jones}

\subsection{Symmetric-color knot elements}
Let $J_r^{(n)}(K;q)$ denote the reduced, zero-framed quantum
$\mathfrak{sl}_n$ invariant of a knot $K$, colored by
$\operatorname{Sym}^r(\mathbb C^n)$, with $J_r^{(n)}(U;q)=1$.
The rank parameter is $n\ge2$, and $r\ge0$ is the symmetric-power
index. Unless another local coordinate is specified, Taylor expansions
in $q$ are taken in $\eta=q-1$.

The integral cyclotomic expansion of \cite[Theorem~1.1]{FZ26} provides
unique coefficients $c_{K,n,k}\in R$, with $c_{K,n,0}=1$, such that
\begin{equation}\label{eq:all-rank-cyclotomic}
 J_r^{(n)}(K;q)=\sum_{k=0}^r c_{K,n,k}(q)
       \prod_{i=0}^{k-1}\{r-i\}\{n+r+i\}.
\end{equation}
By \eqref{eq:node-difference}, the product in
\eqref{eq:all-rank-cyclotomic} is $B_k^{(n)}(x_r^{(n)})$.

\begin{proposition}\label{prop:knot-element}
For every $n\ge2$, the series
\begin{equation}\label{eq:all-rank-element}
 F_{K,n}(X;q)=\sum_{k\ge0}c_{K,n,k}(q)B_k^{(n)}(X)
\end{equation}
is the unique element of $\C_n$ with
$F_{K,n}(x_r^{(n)})=J_r^{(n)}(K;q)$ for every $r\ge0$.
It satisfies
\begin{equation}\label{eq:knot-operations}
 F_{U,n}=1,\qquad F_{K\#L,n}=F_{K,n}F_{L,n},\qquad
 F_{\overline K,n}(X;q)=F_{K,n}(X;q^{-1}).
\end{equation}
\end{proposition}
\begin{proof}
The expansion and \cref{prop:coordinates} give existence. Uniqueness
is \cref{thm:identity}. A zero-framed $(1,1)$-tangle acts by a scalar
on each simple color. The reduced invariants therefore multiply
under connected sum. Mirror reflection inverts $q$, and the node
and kernel polynomials are unchanged by this inversion. Applying
the color identity theorem proves \eqref{eq:knot-operations}.
\end{proof}

\begin{corollary}\label{cor:knot-congruences}
For all $r,s\ge0$,
\begin{equation}\label{eq:all-rank-color-congruence}
 J_r^{(n)}(K;q)-J_s^{(n)}(K;q)
 \in\{r-s\}\{n+r+s\}R.
\end{equation}
Every finite color divided difference is integral. Agreement at
infinitely many colors implies agreement at every color. If the
invariant is a Laurent unit at infinitely many colors, then
$F_{K,n}=1$.
\end{corollary}
\begin{proof}
Apply \cref{prop:divided-difference,thm:identity,cor:sparse-units} to
$F_{K,n}$ and use $J_0^{(n)}=1$.
\end{proof}

\subsection{Normal coefficients and special colors}
At a point \eqref{eq:earlier-node}, write
\[
 \Tay_{n,s}(F_{K,n})=\sum_{j\ge0}\tau_{K,j}^{(n,s)}(q)u_s^j.
\]
By \cref{thm:all-boundaries}, these coefficients lie in $\HH$ and are
given by
\begin{equation}\label{eq:knot-normal}
 \tau_{K,j}^{(n,s)}=
 \sum_{k\ge j}(-1)^{k-j}c_{K,n,k}\,
 e_{k-j}\bigl(\{s+i\}\{n-s+i\}:0\le i<k\bigr).
\end{equation}

\begin{proposition}\label{prop:boundary-special-colors}
Let $\zeta$ be a root of unity and $\ell=\ord(\zeta^2)$.
If $r\ge0$ satisfies $r\equiv-s$ or $r\equiv s-n\pmod\ell$,
then
\begin{equation}\label{eq:boundary-special-color}
 \tau_{K,0}^{(n,s)}(\zeta)=J_r^{(n)}(K;\zeta).
\end{equation}
In particular,
$\tau_{K,0}^{(n,1)}(\zeta)=J_{\ell-1}^{(n)}(K;\zeta)$.
\end{proposition}
\begin{proof}
Either congruence gives
$x_r^{(n)}(\zeta)=a_{n-2s}(\zeta)$.
For $k>r$, the specialized Newton kernel contains
$X-x_r^{(n)}(\zeta)$ and therefore vanishes at this boundary.
The boundary series and the finite color sum consequently agree.
\end{proof}

For the higher normal coefficients, finite reconstruction from
colored invariants, including mixed quantum--normal precision, is
given in \cref{thm:finite-normal-jets,cor:finite-normal-color-data},
and \cref{sec:explicit-normal} computes them for four knots.

\subsection{Perturbative color dependence}
\begin{proposition}\label{prop:all-rank-perturbative}
For every $d\ge0$ there is a polynomial $P_{K,n,d}(Y)\in\mathbb Q[Y]$
such that
\begin{equation}\label{eq:all-rank-perturbative}
 J_r^{(n)}(K;e^h)=\sum_{d\ge0}P_{K,n,d}((n+2r)^2)h^d,
 \qquad
 \deg P_{K,n,d}\le\lfloor d/2\rfloor,
 \qquad d!P_{K,n,d}\in\mathbb Z[Y].
\end{equation}
One has $P_{K,n,0}=1$, $P_{K,n,1}=0$, and
$Y-n^2\mid P_{K,n,d}(Y)$ for every $d>0$.
\end{proposition}
\begin{proof}
At $q=e^h$, a kernel factor is
\[
 2\cosh((n+2r)h)-2\cosh((n+2i)h)
 =2\sum_{a\ge1}\frac{Y^a-(n+2i)^{2a}}{(2a)!}h^{2a}.
\]
At order $d$, only $k\le\lfloor d/2\rfloor$ contributes, and the
Laurent coefficients expand as integral combinations of $e^{jh}$,
$j\in\mathbb Z$. The multinomial formula gives factorial integrality
and the degree bound, and each term with $k\ge1$ contains the factor
$Y-n^2$ coming from $i=0$.
\end{proof}

\subsection{HOMFLY--PT rank specialization}\label{sec:rank}

Let $H_r(K;A,q)$ be the HOMFLY--PT invariant colored by the $r$th
symmetric power, divided by its value on the unknot, with zero framing
\cite[Section~2.1 and equation~(2.2)]{CLZ15}.
With these normalizations $J_r^{(n)}(K;q)=H_r(K;q^n,q)$. Put $Q=q^2$
and $S=\Z[Q^{\pm1}]$. Strong integrality gives
\begin{equation}\label{eq:true-rank-polynomial}
 \PP_{K,r}(Z,Q)\in S[Z^{\pm1}],\qquad
 H_r(K;A,q)=\PP_{K,r}(A^2,Q).
\end{equation}
Here Laurent integrality follows from \cite{Morton07}, and the even
powers of both variables follow from \cite{Zhu19}.
Hence
\begin{equation}\label{eq:true-rank-values}
 J_r^{(n)}(K;q)=\PP_{K,r}(Q^n,Q),\qquad n\ge2,
\end{equation}
and rank variation is interpolation on the geometric nodes $Q^n$.

For the rank-uniform knot element, put
$\{a;A\}=Aq^a-A^{-1}q^{-a}$ and
$\{k\}!=\prod_{j=1}^k\{j\}$. The Laurent differential expansion
of \cite[Theorem~8.4]{FZ26}, in the reduced zero-framing
normalization of \eqref{eq:true-rank-values}, is
\begin{equation}\label{eq:Laurent-differential-source}
 H_r(K;A,q)
 =1+\sum_{k=1}^{r}
   \frac{\prod_{i=0}^{k-1}\{r-i\}}{\{k\}!}
   \{-1;A\}\prod_{i=0}^{k-1}\{r+i;A\}\,G_{K,k}(A,q),
\end{equation}
where $G_{K,k}\in\mathbb Z[A^{\pm1},q^{\pm1}]$.
The universal node difference is
\[
 (Aq^{2r}+A^{-1}q^{-2r})
 -(Aq^{2i}+A^{-1}q^{-2i})=\{r-i\}\{r+i;A\}.
\]
Consequently the universal Newton coefficients are exactly
\begin{equation}\label{eq:rank-Laurent-coefficient}
 \mathsf c_{K,0}=1,\qquad
 \mathsf c_{K,k}(A,q)=
       \frac{(Aq^{-1}-A^{-1}q)G_{K,k}(A,q)}{\{k\}!}
       \quad(k\ge1).
\end{equation}
Specializing $A=q^n$ and using uniqueness in
\eqref{eq:all-rank-cyclotomic} gives
$\mathsf c_{K,k}(q^n,q)=c_{K,n,k}(q)$. In particular,
$\mathsf c_{K,k}\in\mathscr I_q$, as in
\cite[Corollary~8.5]{FZ26}.

By \cite[Theorem~8.7]{FZ26}, the series
\[
 \mathbf F_K(X,A,q)=\sum_{k\ge0}\mathsf c_{K,k}(A,q)
       \prod_{i=0}^{k-1}(X-Aq^{2i}-A^{-1}q^{-2i})
\]
defines an element of $\mathcal U$ with
\[
 \mathbf F_K(Aq^{2r}+A^{-1}q^{-2r},A,q)=H_r(K;A,q).
\]
Its rank-$n$ specialization is the fixed-rank element $F_{K,n}$
by \cref{prop:rank-specialization} and uniqueness of the Newton
expansion.

The tail restriction of $F_{K,n}$ is in general different from
$F_{K,n+2m}$:
\begin{equation}\label{eq:tail-rank-comparison}
 \bigl(\rho_{n,m}(F_{K,n})-F_{K,n+2m}\bigr)(x_0^{(n+2m)})
 =J_m^{(n)}(K;q)-1.
\end{equation}
Here $x_0^{(n+2m)}=x_m^{(n)}$, while the trivial-color value of
$F_{K,n+2m}$ is one. For $m=1$ and $K=4_1$, the right side is
$\{n-1\}\{n+1\}\ne0$, by \cref{prop:seed-coefficients}.

\subsection{Rank congruences and mixed color comparisons}
For fixed $r$, Laurent polynomial interpolation in $Z=A^2$ gives
\begin{equation}\label{eq:rank-congruence-general}
 J_r^{(n)}-J_r^{(N)}\in(Q^{|n-N|}-1)\mathbb Z[Q^{\pm1}].
\end{equation}
For distinct ranks $n_0,\ldots,n_k$, the determinant with rows
\[
 (1,Q^{n_i},\ldots,Q^{(k-1)n_i},J_r^{(n_i)})
\]
is divisible by $\prod_{i<j}(Q^{n_j}-Q^{n_i})$. Indeed, division by
the monic polynomial $\prod_i(Z-Q^{n_i})$ reduces
$\PP_{K,r}(Z,Q)$ to an integral polynomial of degree at most $k$;
the quotient of the determinant by the Vandermonde is its leading
coefficient. At $q=\zeta$ the invariant is periodic in rank modulo
$\ell=\ord(\zeta^2)$.

\begin{theorem}\label{thm:mixed-rank-color}
For $n\ge2$, $r\ge0$ and $a\ge1$,
\begin{equation}\label{eq:mixed-rank-color}
 J_r^{(n+2a)}-J_{r+a}^{(n)}\in(Q^a-1)R.
\end{equation}
Put
\[
 L_{K,n,r,a}(q)=
 \frac{J_{r+a}^{(n)}(K;q)-J_r^{(n)}(K;q)}
      {x_{r+a}^{(n)}(q)-x_r^{(n)}(q)}\in R.
\]
If $\zeta^{2a}=1$ and $\omega=\zeta^2$, the canceled quotient satisfies
\begin{align}\label{eq:mixed-rank-color-jet}
 \left.\frac{J_r^{(n+2a)}-J_{r+a}^{(n)}}{Q^a-1}\right|_{q=\zeta}
 &=2\omega^n\partial_Z\PP_{K,r}(\omega^n,\omega)\notag\\
 &\quad-(\zeta^{n+2r}-\zeta^{-n-2r})L_{K,n,r,a}(\zeta).
\end{align}
For $d=\ord(\zeta)$, the numerator of
\eqref{eq:mixed-rank-color} is divisible by $\Phi_d(q)^2$ exactly
when \eqref{eq:mixed-rank-color-jet} is zero.
\end{theorem}
\begin{proof}
Split the difference into $J_r^{(n+2a)}-J_r^{(n)}$ and
$J_r^{(n)}-J_{r+a}^{(n)}$. The first is divisible by $Q^{2a}-1$;
the second is divisible by $\{a\}\{n+2r+a\}$.
For the first canceled quotient, the rank divided difference tends
to $2\omega^n\partial_Z\PP_{K,r}(\omega^n,\omega)$.
For the second, use
\[
 x_{r+a}^{(n)}-x_r^{(n)}
 =(Q^a-1)(q^{n+2r}-q^{-n-2r-2a}).
\]
This proves the formula. Every root of $q^{2a}-1$ is simple in
characteristic zero, which proves the last assertion.
\end{proof}

\begin{theorem}\label{thm:rank-reconstruction}
Suppose the $Z$-exponents of $\PP_{K,r}$ lie in $[b,b+D]$.
At any $D+1$ distinct ranks $n_0,\ldots,n_D\ge2$, one has in
$\mathbb Q(Q)[Z^{\pm1}]$
\begin{equation}\label{eq:rank-reconstruction}
 \PP_{K,r}(Z,Q)=Z^b\sum_{j=0}^{D}Q^{-bn_j}J_r^{(n_j)}(K;q)
   \prod_{i\ne j}\frac{Z-Q^{n_i}}{Q^{n_j}-Q^{n_i}}.
\end{equation}
On the space
$Z^b\mathbb Z[Q^{\pm1}][Z]_{\le D}$, reconstruction from values
at distinct ranks requires at least $D+1$ samples. Agreement at
infinitely many ranks determines $\PP_{K,r}$ without a degree bound.
\end{theorem}
\begin{proof}
Apply Lagrange interpolation to $Z^{-b}\PP_{K,r}$. With at most $D$
samples, the nonzero Laurent polynomial $Z^b\prod_j(Z-Q^{n_j})$
has the prescribed exponent range and vanishes at all samples.
A nonzero Laurent polynomial over $\mathbb Q(Q)$ has only finitely
many distinct zeros, proving the last assertion.
\end{proof}

\begin{corollary}\label{cor:rank-padic}
With $p,q_0,a$ as in \cref{thm:padic-color}, fix $r$. The function
\[
 g(t)=\PP_{K,r}(q_0^{2t},q_0^2)
\]
belongs to $\mathbb Z_p\langle t\rangle$, interpolates the rank
values, and satisfies
$v_p(g(t)-g(s))\ge a+v_p(t-s)$.
\end{corollary}
\begin{proof}
A finite Laurent polynomial in restricted exponentials is restricted
analytic. Its difference is divisible by $q_0^{2t}-q_0^{2s}$,
whose valuation is $a+v_p(t-s)$.
\end{proof}

\subsection{Quantum expansions and finite-rank reconstruction}

The full Taylor expansion and the Newton-coefficient generating
functions are
\begin{align}
 \mathcal F_{K,n}(u,\eta)
 &=F_{K,n}(2+u,1+\eta)
   =\sum_{m\ge0}R_{K,n,m}(u)\eta^m,
       \label{eq:full-rank-taylor}\\
 A_{K,n,m}(u)
 &=\sum_{k\ge0}[\eta^m]c_{K,n,k}(1+\eta)u^k.
       \label{eq:rank-newton-jets}
\end{align}
The first series has integral coefficients by
\cref{thm:quantum-neighborhoods}; the second has integral coefficients
because $c_{K,n,k}\in R$. The kernel nodes vary with $\eta$, so the
two coefficient families are related by a triangular transformation
(\cref{prop:rank-triangular}).

The Taylor coefficients of $\mathsf c_{K,k}$ have a degree bound in
the formal rank variable, which comes from the scalar denominator
$\{k\}!$ in \eqref{eq:rank-Laurent-coefficient}.

\begin{theorem}\label{thm:quantum-rank-degree}
For $k\ge1$ and $m\ge0$, define
\[
 a_{K;k,m}(N)=[\eta^m]\mathsf c_{K,k}
                   ((1+\eta)^N,1+\eta).
\]
Then
\begin{equation}\label{eq:quantum-rank-degree}
 a_{K;k,m}\in B,\qquad
 \deg_N a_{K;k,m}\le k+m,\qquad a_{K;k,m}(1)=0.
\end{equation}
For every $n\ge2$,
$a_{K;k,m}(n)=[\eta^m]c_{K,n,k}(1+\eta)$, and
\begin{equation}\label{eq:quantum-rank-interpolation}
 a_{K;k,m}(N)=\sum_{j=0}^{k+m}
       (\Delta_N^j a_{K;k,m})(2)\binom{N-2}{j}.
\end{equation}
\end{theorem}
\begin{proof}
The membership in $B$ and the specialization formula follow from
\cref{prop:rank-coefficient-jets} and
\eqref{eq:rank-Laurent-coefficient}.
Write the denominator of that equation as
\[
 \{k\}!=\eta^k d_k(\eta),\qquad
 d_k(0)=2^k k!\ne0.
\]
A Laurent monomial $A^a q^b$ in its numerator becomes
$(1+\eta)^{aN+b}$. Its coefficient of $\eta^j$ is
$\binom{aN+b}{j}$, a polynomial in $N$ of degree at most $j$.
Since $d_k(\eta)^{-1}\in\mathbb Q[[\eta]]$ is independent of $N$,
the quotient coefficient of order $m$ is a rational linear
combination of numerator coefficients of orders at most $k+m$.
This proves the degree bound. At $N=1$, the factor
$Aq^{-1}-A^{-1}q$ vanishes identically as a series in $\eta$.
The regular quotient therefore has every Taylor coefficient zero
at $N=1$. Finally, polynomial Newton interpolation in $N$ gives
\eqref{eq:quantum-rank-interpolation}.
\end{proof}

\begin{corollary}\label{cor:finite-rank-color-jets}
Let $L,s\ge1$. The exact Taylor data
\begin{equation}\label{eq:finite-rank-color-jets}
 J_r^{(n)}(K;1+\eta)\pmod{\eta^{s+2L}},
 \qquad 0\le r\le L,\qquad 2\le n\le L+s,
\end{equation}
determine $c_{K,n,k}(1+\eta)\pmod{\eta^s}$ for every
$n\ge2$ and $0\le k\le L$.
\end{corollary}
\begin{proof}
At each sampled rank the color equations form a triangular
system. For $r\ge k$, the kernel value
$B_k^{(n)}(x_r^{(n)})$ has $\eta$-order $2k$. In particular,
\begin{equation}\label{eq:rank-diagonal-leading}
 B_k^{(n)}(x_k^{(n)};1+\eta)
 =4^k k!\prod_{i=0}^{k-1}(n+k+i)\,\eta^{2k}
                 +O(\eta^{2k+1}).
\end{equation}
All factors in the leading integer are nonzero. Divide the
triangular equations in $\mathbb Q[[\eta]]$. Inductively,
$c_{K,n,k}$ is recovered modulo $\eta^{s+2(L-k)}$:
each earlier coefficient $c_{K,n,j}$ is known modulo
$\eta^{s+2(L-j)}$, and multiplication by its order-$2j$ kernel
makes its contribution known modulo $\eta^{s+2L}$.
Division by the diagonal then loses exactly $2k$ powers of
$\eta$. The recovered coefficients agree with the integral
coefficients in \eqref{eq:all-rank-cyclotomic}.

For $1\le k\le L$ and $m<s$, this determines $a_{K;k,m}(n)$
at the $L+s-1$ ranks $2,\ldots,L+s$. By
\cref{thm:quantum-rank-degree}, the polynomial has degree at most
$L+s-1$ and has the known value zero at $N=1$. Its values at the
$L+s$ distinct integers $1,\ldots,L+s$ therefore determine it.
Evaluation at any $n\ge2$ recovers the required coefficient.
The remaining coefficient $c_{K,n,0}$ is identically one.
\end{proof}

\section{Alexander denominators}\label{sec:alexander}

Let $\Delta_K(z)\in\mathbb Z[z^{\pm1}]$ be the symmetric Alexander
polynomial, normalized by $\Delta_K(1)=1$. There is a unique
$\delta_K\in\mathbb Z[u]$ with
\[
 \Delta_K(z)=\delta_K(z+z^{-1}-2),\qquad \delta_K(0)=1.
\]

We bound the denominators of the full coefficients $R_{K,n,m}$ of
\eqref{eq:full-rank-taylor} by the rational graph expansion, and then
pass to the coefficient generating functions $A_{K,n,m}$ of
\eqref{eq:rank-newton-jets}.

\subsection{Root charges and rational loop coefficients}
Give the roots of $\mathfrak{sl}_n$ squared length two. Let
$\omega_1$ be the first fundamental weight and $\rho$ half the sum
of the positive roots. Put $h=\log q$ and
$\mu=r\omega_1+\rho$, the shifted highest weight. We use Euler degree
$d=|E(\Gamma)|-|V(\Gamma)|=b_1(\Gamma)-1$ for a connected
trivalent graph $\Gamma$. Such a graph has $2d$ vertices and
$3d$ edges.

\begin{lemma}[Normalized rational graph expansion]
\label{lem:normalized-graph-input}
The reduced invariant has the formal expansion
\begin{equation}\label{eq:normalized-graph-expansion}
 J_r^{(n)}(K;q)
 =\prod_{\alpha>0}\Delta_K(q^{2(\mu,\alpha)})^{-1}
       \exp\!\left(\sum_{d\ge1}(2h)^d
                         \mathcal Q_{K,d}(\mu,q)\right).
\end{equation}
For each $d$, the term $\mathcal Q_{K,d}$ is a finite rational
linear combination of Lie weight evaluations of trivalent graphs
with $2d$ vertices. An edge has a rational label
$P(t_e)/\Delta_K(t_e)$, where $P$ is a Laurent polynomial.
Weight evaluation assigns a root space or a Cartan subspace to each
edge and contracts the Lie bracket at each vertex. The unknot
contribution is subtracted in $\mathcal Q_{K,d}$.
\end{lemma}
\begin{proof}
Kricker's rationality theorem \cite[Theorem~1.0.8]{Kricker00},
in the wheeled presentation of
\cite[equation~(4.12) and Remark~4.3]{RozanskyRat}, separates
the wheel term from the connected positive-Euler-degree diagrams.
The latter have the stated edge labels. Applying the Lie weight
system gives the graph terms in
\cite[equations~(5.22)--(5.27)]{RozanskyRat}.

The quantum parameter there is $q^2$ in our notation, so its
logarithm is $2h$. For the unreduced invariant $\mathcal J_\mu$,
the unknot value is the quantum dimension
\begin{equation}\label{eq:graph-quantum-dimension}
 d_{q^2,\mu}=\mathcal J_\mu(U;q^2)
   =\prod_{\alpha>0}
       \frac{q^{(\mu,\alpha)}-q^{-(\mu,\alpha)}}
            {q^{(\rho,\alpha)}-q^{-(\rho,\alpha)}}.
\end{equation}
Taking the quotient of the expressions for $K$ and $U$ cancels
this factor and the common Lie-algebra normalization. Zero framing
sets the framing factor to one. The wheel quotient is the product
of inverse Alexander polynomials in
\eqref{eq:normalized-graph-expansion}. The logarithmic graph terms
are replaced by their differences with the unknot terms; those
terms have Laurent edge labels because $\Delta_U=1$.
This proves the formula.
\end{proof}

For $\mathfrak{sl}_n$, take the symmetric-power direction $\omega_1$.
An edge labeled by $\alpha=e_i-e_j$ has charge
\[
 \varepsilon(\alpha)=(\omega_1,\alpha)
       =\delta_{i1}-\delta_{j1}\in\{-1,0,1\};
\]
a Cartan edge has charge zero. An edge is active if its charge is
nonzero.

\begin{lemma}\label{lem:active-edge-count}
A nonzero weight assignment to a trivalent graph of Euler degree $d$
has at most $2d$ active edges.
\end{lemma}
\begin{proof}
At a nonzero trivalent vertex the signed charges sum to zero, by
weight conservation for the Lie bracket. Three numbers from
$\{-1,1\}$ cannot sum to zero, so at most two incident half-edges
are active. There are $2d$ vertices. Counting both ends of each edge,
including twice for a loop, gives at most $2d$ active edges.
\end{proof}

\begin{lemma}\label{lem:diagonal-faithfulness}
If $G(z,\eta)\in\mathbb Q[[z-1,\eta]]$ satisfies
$G((1+\eta)^{n+2r},\eta)=0$ for infinitely many integers $r\ge0$,
then $G=0$. The substitution $u=z+z^{-1}-2$ is injective on
$\mathbb Q[[u,\eta]]$.
\end{lemma}
\begin{proof}
Write $w=z-1$. If $G\ne0$, let
\[
 H_d(w,\eta)=\sum_{i=0}^d a_iw^i\eta^{d-i}
\]
be its nonzero homogeneous part of least total degree. For a fixed
integer $t$, the substitution $w=(1+\eta)^t-1=t\eta+O(\eta^2)$
gives
\[
 [\eta^d]G((1+\eta)^t,\eta)=\sum_{i=0}^d a_it^i.
\]
The polynomial on the right is nonzero and cannot vanish at infinitely
many distinct values $t=n+2r$. This proves the first assertion.

For the second, $u=z+z^{-1}-2=w^2/(1+w)$. A nonzero power series
in $u$ of order $j$ becomes a power series in $w$ of order $2j$.
Applying this observation to each coefficient of $\eta$ proves
injectivity on $\mathbb Q[[u,\eta]]$.
\end{proof}

\begin{theorem}\label{thm:rank-loop-denominators}
For every knot $K$, $n\ge2$ and $m\ge0$,
\begin{equation}\label{eq:rank-loop-denominators}
 R_{K,n,m}(u)\in
 \delta_K(u)^{-(n-1+2m)}\mathbb Z[u],\qquad
 R_{K,n,0}(u)=\delta_K(u)^{-(n-1)}.
\end{equation}
\end{theorem}
\begin{proof}
Apply \cref{lem:normalized-graph-input} with $h=\log q$ and retain
$z=q^{n+2r}$ as an independent variable. For
$\mu=r\omega_1+\rho$, the positive-root arguments are
\begin{equation}\label{eq:rank-root-variables}
 q^{2(\mu,e_i-e_j)}=
 \begin{cases}
 zq^{2(j-1)-n},&i=1,\\
 q^{2(j-i)},&i>1.
 \end{cases}
\end{equation}
Negative roots give the inverse arguments. An active edge therefore
has an Alexander factor evaluated at $z^{\pm1}q^a$; an inactive
edge has one evaluated at $q^a$. At $h=0$, the active denominators
become $\Delta_K(z)$ by symmetry, and the inactive denominators
become $\Delta_K(1)=1$. The wheel factor has exactly $n-1$ active
positive roots, namely $e_1-e_j$ with $2\le j\le n$.

Consider a product of graph terms of total Euler degree $d$ in the
exponential. It carries $h^d$ and has at most $2d$ active edge
factors by \cref{lem:active-edge-count}. Each further derivative
in $h$ increases their total Alexander denominator exponent by at
most one: differentiating a factor
$P(z^{\pm1}e^{ah})/\Delta_K(z^{\pm1}e^{ah})$ has this property,
and the product rule preserves the bound. Derivatives of inactive
factors introduce no poles in $z$. Thus its contribution to the
coefficient of $h^m$, including the wheel factor, has exponent at most
\[
 (n-1)+2d+(m-d)=n-1+m+d\le n-1+2m.
\]
The term $d=0$ is the wheel factor alone. The subtracted unknot terms
have Laurent edge labels and satisfy the same bound. Only degrees
$d\le m$ occur, with finitely many graph products and weight
assignments in each degree. The coefficient of $h^m$ consequently
belongs to
$\Delta_K(z)^{-(n-1+2m)}\mathbb Q[z^{\pm1}]$.
Substitution $h=\log(1+\eta)$ combines only orders at most $m$
in the coefficient of $\eta^m$, so it preserves this bound.

Let $Q_m(z)$ be the rational coefficients so obtained. Since
$\Delta_K(1)=1$, their expansions at $z=1$ define a series in
$\mathbb Q[[z-1,\eta]]$. On every color diagonal
$z=(1+\eta)^{n+2r}$, this series equals the Taylor expansion of
$J_r^{(n)}(K;q)$. The Taylor image
\eqref{eq:full-rank-taylor}, with $u=z+z^{-1}-2$, has the same
values. By \cref{lem:diagonal-faithfulness},
\[
 Q_m(z)=R_{K,n,m}(z+z^{-1}-2)
\]
as formal series at $z=1$. The right side is inversion-invariant,
so the rational function $Q_m$ is inversion-invariant as well.
It follows that $\Delta_K^{n-1+2m}Q_m$ is a symmetric Laurent
polynomial. Such a polynomial belongs to
$\mathbb Q[z+z^{-1}]=\mathbb Q[u]$.

By \cref{thm:quantum-neighborhoods}, $R_{K,n,m}\in\mathbb Z[[u]]$.
Multiplication by $\delta_K^{n-1+2m}$ therefore gives a polynomial
in $\mathbb Q[u]$ whose Taylor coefficients are all integers;
it belongs to $\mathbb Z[u]$. At order zero, only the wheel
factor remains. Its $n-1$ active positive roots give
$R_{K,n,0}=\delta_K^{-(n-1)}$.
\end{proof}

\begin{corollary}\label{cor:universal-classical-knot}
Under \eqref{eq:universal-quantum-neighborhood} with $s=1$, the
rank-uniform knot element has image
\begin{equation}\label{eq:universal-classical-knot}
 \mathbf F_K\pmod{\eta}
       =\delta_K(u)^{-(N-1)}\in B[[u]].
\end{equation}
Its coefficient of $u^k$ is an integer-valued polynomial in $N$
of degree at most $k$.
\end{corollary}
\begin{proof}
The binomial expansion gives
\[
 \delta_K(u)^{-(N-1)}
 =\sum_{j\ge0}(-1)^j\binom{N+j-2}{j}
                         (\delta_K(u)-1)^j.
\]
Since $\delta_K(0)=1$, the coefficient of $u^k$ uses only
$j\le k$. It lies in $B$ and has degree at most $k$.
The coefficient of $u^k$ in the classical image of $\mathbf F_K$
also belongs to $B$. After $N=n$ the two coefficients agree
for every $n\ge2$ by \cref{thm:rank-loop-denominators}.
Equality at these infinitely many integers proves equality
as polynomials in $N$.
\end{proof}

\subsection{The triangular change of coefficients}
The full expansion includes the motion of the nodes as $q$ varies; we
express this motion by differential operators in $u$. Set
\[
 b_i^{(n)}(\eta)=(1+\eta)^{n+2i}+(1+\eta)^{-n-2i}-2.
\]
For $d\ge2$ and $1\le s\le\lfloor d/2\rfloor$, define
\begin{equation}\label{eq:rank-Pds}
 P_{d,s}^{(n)}(k)=[\eta^dy^s]
                  \prod_{i=0}^{k-1}(1+yb_i^{(n)}(\eta)).
\end{equation}

\begin{lemma}\label{lem:rank-Pds}
The function $P_{d,s}^{(n)}$ is a polynomial in $k$ over $\mathbb Q$
of degree at most $d+s$. It is integer-valued on the nonnegative
integers and vanishes at $k=0,\ldots,s-1$.
\end{lemma}
\begin{proof}
The coefficient
$\beta_d(i)=[\eta^d]b_i^{(n)}(\eta)$ is a polynomial in $i$ of
degree at most $d$, and is an integer at every nonnegative integer.
It vanishes for $d=0,1$. Appending the last factor to
\eqref{eq:rank-Pds} gives
\[
 P_{d,s}^{(n)}(k+1)-P_{d,s}^{(n)}(k)
   =\sum_{a=2}^d\beta_a(k)P_{d-a,s-1}^{(n)}(k).
\]
With $P_{0,0}=1$ and the other $P_{d,0}=0$, induction bounds the
right side by degree $d+s-1$. Discrete antidifferentiation gives
the stated degree bound. Integer values and the initial zeros follow
directly from the product definition.
\end{proof}

Put $\vartheta=u\,d/du$, $T_0^{(n)}=1$, $T_1^{(n)}=0$, and
\begin{equation}\label{eq:rank-Td}
 T_d^{(n)}=\sum_{s=1}^{\lfloor d/2\rfloor}
          (-1)^su^{-s}P_{d,s}^{(n)}(\vartheta)\qquad(d\ge2).
\end{equation}
The initial zeros in \cref{lem:rank-Pds} make these operators
well-defined on $\mathbb Z[[u]]$.

\begin{proposition}\label{prop:rank-triangular}
The two coefficient families satisfy
\begin{equation}\label{eq:rank-triangular}
 R_{K,n,m}=\sum_{d=0}^mT_d^{(n)}A_{K,n,m-d}.
\end{equation}
For every $b\ge1$ and $\delta\in\mathbb Z[u]$ with $\delta(0)=1$,
the bounds $R_m\in\delta^{-b-2m}\mathbb Z[u]$ for all $m$ are
equivalent to the same bounds for $A_m$. In either case,
\begin{equation}\label{eq:rank-highest-transport}
 \delta^{b+2m-1}(A_m-R_m)\in\mathbb Z[u]\qquad(m\ge1).
\end{equation}
\end{proposition}
\begin{proof}
Expanding the kernel by elementary symmetric functions gives
\[
 [\eta^d]\prod_{i<k}(u-b_i^{(n)}(\eta))=T_d^{(n)}u^k.
\]
Every coefficient of $u^j\eta^m$ involves only finitely many Newton
indices by \cref{thm:quantum-neighborhoods}. Multiplying by the
coefficient series and collecting the order $m$ terms proves
\eqref{eq:rank-triangular}.

For one summand of $T_d^{(n)}$, take $d\ge2$ and
$1\le s\le\lfloor d/2\rfloor$. An Euler derivative raises a
$\delta$-denominator exponent by at most one. Since
$\deg P_{d,s}^{(n)}\le d+s$, applying
$P_{d,s}^{(n)}(\vartheta)$ to
$f\in\delta^{-a}\mathbb Z[u]$ gives a rational function with
denominator dividing $\delta^{a+d+s}$ over $\mathbb Q$.
The coefficients of the formal series $f$ are integral because
$\delta(0)=1$. On its monomial $u^k$, the operator multiplies the
coefficient by the integer $P_{d,s}^{(n)}(k)$, which vanishes for
$k<s$. Hence $P_{d,s}^{(n)}(\vartheta)f\in u^s\mathbb Z[[u]]$.
Division by $u^s$ introduces no pole at zero. The resulting rational
function has denominator dividing $\delta^{a+d+s}$ and an integral
Taylor series. Multiplication by that denominator then gives a
polynomial in $\mathbb Q[u]\cap\mathbb Z[[u]]=\mathbb Z[u]$.

Suppose the stated bound holds for $A_{m-d}$. Its contribution with
$d\ge2$ and $s\le\lfloor d/2\rfloor$ has exponent at most
\[
 b+2(m-d)+d+s=b+2m-d+s\le b+2m-1.
\]
Since $T_0=1$ and $T_1=0$, \eqref{eq:rank-triangular} now gives
the bound for $R_m$ and the stronger bound for $R_m-A_m$.
Conversely, solve the same identity as
\[
 A_m=R_m-\sum_{d=2}^mT_d^{(n)}A_{m-d}.
\]
Starting with $A_0=R_0$ and $A_1=R_1$, induction gives the bounds
for $A_m$ and for the difference. This proves both directions and
\eqref{eq:rank-highest-transport}.
\end{proof}

For example, if
\[
 S_{2,n}(k)=n^2k+2nk(k-1)+\frac23k(k-1)(2k-1),
\]
then
\begin{equation}\label{eq:rank-low-operators}
 T_2^{(n)}=-u^{-1}S_{2,n}(\vartheta),\qquad
 T_3^{(n)}=u^{-1}S_{2,n}(\vartheta).
\end{equation}
Indeed, $[\eta^2]b_i^{(n)}=(n+2i)^2$ and
$[\eta^3]b_i^{(n)}=-(n+2i)^2$.

\begin{theorem}\label{thm:rank-Alexander}
For every zero-framed knot $K$, $n\ge2$ and $m\ge0$, there is a
polynomial $N_{K,n,m}\in\mathbb Z[u]$ such that
\begin{equation}\label{eq:rank-Alexander}
 A_{K,n,m}(u)=\frac{N_{K,n,m}(u)}
                   {\delta_K(u)^{n-1+2m}}.
\end{equation}
One has $N_{K,n,0}=1$, $A_{K,n,1}=R_{K,n,1}$, and
\begin{equation}\label{eq:rank-top-pole}
 \delta_K^{n+2m-2}(A_{K,n,m}-R_{K,n,m})\in\mathbb Z[u]
       \qquad(m\ge1).
\end{equation}
For each fixed pair $(n,m)$ the exponent $n-1+2m$ is optimal.
\end{theorem}
\begin{proof}
Apply \cref{prop:rank-triangular} with $b=n-1$ to
\cref{thm:rank-loop-denominators}. Optimality follows from the
nonzero highest pole for the trefoil in \cref{thm:all-rank-sharp}.
\end{proof}
For $m\ge1$ and $E=n-1+2m$, \eqref{eq:rank-top-pole} gives the
polynomial congruence
\[
 \delta_K^E A_{K,n,m}\equiv\delta_K^E R_{K,n,m}
              \pmod{\delta_K\mathbb Z[u]}.
\]
In particular, at every simple zero of $\delta_K$, the coefficients
of the highest permitted pole of $A_{K,n,m}$ and $R_{K,n,m}$ agree.

\subsection{Products, rank dependence, and coefficient recurrences}
For a family $A=\sum_mA_m\eta^m$, let
\[
 \mathfrak T_n(A)=\sum_m\sum_{d=0}^m T_d^{(n)}A_{m-d}\eta^m.
\]
It is an invertible $\eta$-adically triangular integral linear map.
The product corresponding to multiplication in $\C_n$ is
\begin{equation}\label{eq:rank-star-product}
 A\star_nA'=\mathfrak T_n^{-1}
       \bigl(\mathfrak T_n(A)\mathfrak T_n(A')\bigr).
\end{equation}
This follows from multiplicativity of the full Taylor map. If
$A_m\in\delta_1^{-b-2m}\mathbb Z[u]$ and
$A'_m\in\delta_2^{-b-2m}\mathbb Z[u]$, then
\[
 (A\star_nA')_m\in(\delta_1\delta_2)^{-b-2m}\mathbb Z[u].
\]
To prove it, transfer to the full coefficients, multiply the two
series and use a common denominator, then transfer back.
For knot elements this agrees with connected sum and
$\Delta_{K\#L}=\Delta_K\Delta_L$.

The classical rank dependence has the explicit form
\begin{equation}\label{eq:classical-rank-polynomial}
 c_{K,n,k}(1)=\sum_{j=0}^k(-1)^j\binom{n+j-2}{j}
                    [u^k](\delta_K(u)-1)^j.
\end{equation}
For fixed $k$ this is an integer-valued polynomial in $n$ of degree
at most $k$. The whole rank series is
\begin{equation}\label{eq:classical-rank-generating}
 \sum_{n\ge2}t^{n-2}A_{K,n,0}(u)=\frac1{\delta_K(u)-t}.
\end{equation}
Both identities follow by expanding $\delta_K^{-(n-1)}$.

\begin{corollary}\label{cor:jet-recurrence}
Fix $K,n,m$, and write
$\delta_K(u)^{n-1+2m}=\sum_{j=0}^Ld_ju^j$, with $d_0=1$.
The sequence $a_k=[\eta^m]c_{K,n,k}(1+\eta)$ satisfies
\begin{equation}\label{eq:jet-recurrence}
 a_k=-\sum_{j=1}^Ld_ja_{k-j}
\end{equation}
for all sufficiently large $k$; terms with negative indices are
understood as zero. Modulo any positive integer it is eventually
periodic. If $\Delta_K=1$, then for every $M$ there is $k_0$ such that
\begin{equation}\label{eq:Alexander-one-tail}
 c_{K,n,k}(q)\in(q-1)^{M+1}R\qquad(k\ge k_0).
\end{equation}
\end{corollary}
\begin{proof}
Multiply \eqref{eq:rank-Alexander} by its denominator and compare
coefficients beyond the numerator degree. Over a finite coefficient
ring the finite recurrence has finitely many states, hence is
eventually periodic. If $L=0$, the sequence is eventually zero.
When $\delta_K=1$, each $A_{K,n,m}$ is polynomial. Choose $k_0$
beyond the degrees for $0\le m\le M$. Vanishing of these Taylor
coefficients is equivalent to divisibility by $(q-1)^{M+1}$ in $R$.
\end{proof}

\section{Difference equations and Frobenius comparison}\label{sec:recurrences}

\subsection{The oriented localized solution}
Fix $n\ge2$, put $b=n-1$, and set
\begin{equation}\label{eq:rank-localization}
 A_\Delta=\mathbb Z[z^{\pm1},\Delta(z)^{-1}],\qquad
 \mathscr A_\Delta=A_\Delta[[\eta]],\qquad q=1+\eta,
\end{equation}
where $\Delta=\Delta_K$. The substitution
$\sigma(z)=q^2z$, $\sigma(q)=q$, defines an automorphism of
$\mathscr A_\Delta$: the series $\Delta(q^2z)$ has invertible
constant coefficient $\Delta(z)$, and the same applies to $q^{-2}$.
The Ore operator $E$ acts by $Ef=\sigma(f)$, with $Ez=q^2zE$.

By \cref{thm:rank-loop-denominators}, the knot element gives
\begin{equation}\label{eq:rank-localized-solution}
 \mathcal L_{K,n}(z,q)=\sum_{m\ge0}
 R_{K,n,m}(z+z^{-1}-2)\eta^m\in\mathscr A_\Delta.
\end{equation}
Its specialization at $z=q^{n+2r}$ is the Taylor expansion of
$J_r^{(n)}(K;q)$.

\begin{theorem}\label{thm:rank-annihilators}
For $P=\sum_{j=0}^da_j(z,q)E^j$, with
$a_j\in\mathbb Z[z^{\pm1},q^{\pm1}]$, one has
\begin{equation}\label{eq:rank-annihilators}
 P\mathcal L_{K,n}=0
 \quad\Longleftrightarrow\quad
 \sum_{j=0}^da_j(q^{n+2r},q)J_{r+j}^{(n)}(K;q)=0
 \quad\text{for every }r\ge0.
\end{equation}
On the right, it suffices to require the identity for all
sufficiently large $r$.
\end{theorem}
\begin{proof}
Evaluation at $z=q^{n+2r}$ turns $E^j$ into the color shift
$r\mapsto r+j$, giving one direction. In the other direction,
$P\mathcal L_{K,n}$ has an expansion in
$\mathbb Q[[z-1,\eta]]$ because $\Delta(1)=1$.
Its values on infinitely many color diagonals vanish, so it is zero
by \cref{lem:diagonal-faithfulness}. Since the localized ring embeds
in this formal ring, $P\mathcal L_{K,n}=0$.
\end{proof}
The fixed-rank $q$-holonomic relation can be written in this Ore
algebra by replacing its color variable $q^{2r}$ with $q^{-n}z$.

\subsection{The subprincipal equation and rational residues}
For an annihilator $P$, define
\[
 P_0(z,L)=\sum_ja_j(z,1)L^j,\quad
 A(z)=\partial_LP_0(z,1),\quad
 B(z)=\left.\sum_j\partial_qa_j(z,q)\right|_{q=1},\quad
 D=z\frac d{dz}.
\]

\begin{proposition}\label{prop:rank-subprincipal}
If $P\mathcal L_{K,n}=0$, then
\begin{equation}\label{eq:rank-subprincipal}
 P_0(z,1)=0,\qquad B(z)\Delta(z)=2bA(z)D\Delta(z).
\end{equation}
If $\Delta$ is squarefree over $\mathbb Q$ and $A\ne0$, then
$\Delta\mid A$ in $\mathbb Q[z^{\pm1}]$.
\end{proposition}
\begin{proof}
The constant coefficient is $P_0(z,1)\Delta^{-b}$.
Since $\sigma^jf=f+2j\eta Df+O(\eta^2)$, the first coefficient,
after the constant one vanishes, is
$2A D(\Delta^{-b})+B\Delta^{-b}$.
Multiplication by $\Delta^{b+1}$ gives the second identity.
A squarefree Laurent polynomial is relatively prime to its Euler
derivative, and $2b$ is nonzero in characteristic zero.
\end{proof}

Write $f_j(z)=R_{K,n,j}(z+z^{-1}-2)$. For $m\ge1$ set
\[
 S_m=[\eta^{m+1}]P\left(\sum_{j=0}^{m-1}f_j(z)\eta^j\right).
\]

\begin{theorem}\label{thm:rank-residue}
Assume $P_0(z,1)=0$, $A\ne0$ and $B\Delta=2bA D\Delta$.
If $f_0,\ldots,f_{m-1}\in\mathbb Q(z)$ solve the recurrence
through order $\eta^m$, then an extension through order
$\eta^{m+1}$ is equivalent to
\begin{equation}\label{eq:rank-residue-equation}
 (2A D+B)f_m=-S_m,\qquad
 d(\Delta^bf_m)=-\frac{\Delta^bS_m}{2zA}\,dz.
\end{equation}
A rational solution exists precisely when the differential on the
right has zero residue at every finite pole over $\overline{\mathbb Q}$.
Its solutions differ by a constant multiple of $\Delta^{-b}$.
For solutions regular at $z=1$, the normalization
$\mathcal L(q^n,q)=1$ determines the constant at every order.
\end{theorem}
\begin{proof}
At order $\eta^{m+1}$ the contribution of $f_{m+1}$ is
$P_0(z,1)f_{m+1}=0$, while that of $f_m$ is
$(2A D+B)f_m$. All earlier terms give $S_m$.
The subprincipal identity identifies the integrating factor as
$\Delta^b$, giving \eqref{eq:rank-residue-equation}.
A rational differential on the line has a rational primitive exactly
when its partial fraction decomposition has no simple-pole terms.
Those coefficients are its residues; the polynomial part and all
higher-order poles integrate rationally. To obtain a primitive over
$\mathbb Q$, choose a rational point away from its poles and set its
value there to zero. This normalized primitive is fixed by every
Galois automorphism. Division by $\Delta^b$ gives the stated solution
space. Finally, at order $m$ the normalization reads
\[
 f_m(1)+[\eta^m]\sum_{j<m}
       f_j((1+\eta)^n)\eta^j=0.
\]
Since $\Delta(1)=1$, this fixes the remaining constant.
\end{proof}

\subsection{Frobenius lifts and scalar comparisons}
For a prime $p$, write
\[
 \widehat A_{\Delta,p}=\varprojlim_v A_\Delta/p^vA_\Delta,
 \qquad \mathscr B=\widehat A_{\Delta,p}[[\eta]].
\]
Besides $\sigma$, consider
\[
 \varphi_p(z)=z^p,\quad\varphi_p(q)=q,
 \qquad \Psi_p(z)=z^p,\quad\Psi_p(q)=q^p.
\]
These maps are continuous for the $(p,\eta)$-adic topology.
Indeed, if
\[
 E_p(z)=\frac{\Delta(z^p)-\Delta(z)^p}{p},
\]
then
\begin{equation}\label{eq:rank-frob-inverse}
 \Delta(z^p)^{-a}
 =\sum_{j\ge0}(-p)^j\binom{a+j-1}{j}
       \frac{E_p(z)^j}{\Delta(z)^{p(a+j)}}
       \qquad(a\ge1)
\end{equation}
converges in $\widehat A_{\Delta,p}$. The first $v$ terms suffice
modulo $p^v$. Thus the Frobenius image of an order-$m$ loop
coefficient can be computed modulo $p^v$ with denominator at most
$\Delta^{p(b+2m+v-1)}$.

Checking on $q$ and $z$ gives the covariance identities
\begin{equation}\label{eq:rank-frob-covariance}
 \sigma\varphi_p=\varphi_p\sigma^p,\qquad
 \sigma\Psi_p=\Psi_p\sigma.
\end{equation}
On inversion-invariant functions, the simultaneous map restricts to
the Adams operation $\psi_p$.
For a step-$p$ recurrence, covariance gives
\[
 \sum_ja_j\sigma^{pj}f=0
 \quad\Longrightarrow\quad
 \sum_j\varphi_p(a_j)\sigma^j(\varphi_pf)=0.
\]
Applying $\Psi_p$ instead preserves an ordinary step-one recurrence.

\begin{proposition}\label{prop:rank-scalar-frobenius}
Let $f=\mathcal L_{K,n}=\sum_mf_m\eta^m$. Its inverse has
coefficients
\[
 [\eta^m]f^{-1}\in\Delta^{b-2m}\mathbb Z[z^{\pm1}]
       \qquad(m\ge0).
\]
Put $C_{K,n,p}=\varphi_p(f)/f$ and $a=\sigma(f)/f$. Then
\begin{align}
 \left.C_{K,n,p}\right|_{\eta=0}&=\frac{\Delta(z)^b}{\Delta(z^p)^b},
       \label{eq:rank-scalar-frobenius}\\
 \frac{\sigma C_{K,n,p}}{C_{K,n,p}}
 &=\frac{\varphi_p(\prod_{i=0}^{p-1}\sigma^ia)}{a},\notag\\
 [\eta^m]C_{K,n,p}
 &\in\frac{\Delta(z)^{b-2m}}{\Delta(z^p)^{b+2m}}
          \mathbb Z[z^{\pm1}].\notag
\end{align}
Here $\Delta(z^p)^{-1}$ is understood in $\mathscr B$ through
\eqref{eq:rank-frob-inverse}.
\end{proposition}
\begin{proof}
Write $f=\Delta^{-b}U$ with $U=1+\sum_{m\ge1}b_m\eta^m$
and $b_m\in\Delta^{-2m}\mathbb Z[z^{\pm1}]$.
An order-$m$ coefficient of $U^{-1}$ is a sum of products of $b_j$
whose indices add to $m$, so its denominator is at most
$\Delta^{2m}$. This proves the inverse estimate.
Covariance gives
$\sigma\varphi_p(f)=\varphi_p(\prod_{i=0}^{p-1}\sigma^ia)\varphi_p(f)$,
which yields the second formula. The coefficient bound follows by
summing $\varphi_p(f_j)[\eta^{m-j}]f^{-1}$ over $0\le j\le m$.
\end{proof}

\begin{proposition}\label{prop:rank-integral-log}
Let $p$ be odd, write $f=\mathcal L_{K,n}=\sum_{m\ge0}f_m\eta^m$,
and put $U=\Delta^bf$ and $b_1=\Delta^bf_1$. Then
\begin{equation}\label{eq:rank-integral-log}
 \Lambda_{K,n,p}=\frac1p\log\!\left(\frac{\Psi_p(U)}{U^p}\right)
 \in\eta\mathscr B,\qquad
 [\eta]\Lambda_{K,n,p}=b_1(z^p)-b_1(z).
\end{equation}
Moreover $\Psi_p(f)/f^p\in1+p\mathscr B$, with constant coefficient
$\Delta(z)^{pb}/\Delta(z^p)^b$.
\end{proposition}
\begin{proof}
The total Frobenius congruence gives $\Psi_p(U)\equiv U^p\pmod p$.
Both series have $\eta$-constant coefficient one, so their quotient
lies in $(1+p\mathscr B)\cap(1+\eta\mathscr B)$.
The series $p^{-1}\log(1+pY)$ has integral coefficients tending to
zero $p$-adically. It therefore converges, with zero constant
coefficient in $\eta$. Since $q^p-1=p\eta+O(\eta^2)$,
first-order comparison proves the stated coefficient. The same
Frobenius congruence applied directly to $f$ gives the last assertion.
\end{proof}

\begin{example}\label{ex:trefoil-frobenius}
In the chirality convention of \cref{prop:seed-coefficients},
$\Delta_{3_1}(z)=z+z^{-1}-1$. The first coefficient in
\eqref{eq:four-low-jets}, together with $R_{K,n,1}=A_{K,n,1}$,
gives
\[
 b_1(z)=\Delta_{3_1}(z)^{n-1}f_1(z)
       =n(n-1)\bigl(\Delta_{3_1}(z)^{-2}-1\bigr).
\]
Thus for every odd prime $p$ and every $n\ge2$,
\begin{equation}\label{eq:trefoil-frobenius-first}
 [\eta]\Lambda_{3_1,n,p}
 =n(n-1)\left(
       \Delta_{3_1}(z^p)^{-2}-\Delta_{3_1}(z)^{-2}\right).
\end{equation}
\end{example}

\section{Root-of-unity expansions and recurrence methods}\label{sec:transport}

We use $z$ for the oriented color coordinate, with $X=z+z^{-1}$,
$\xi$ for the exponential of a recurrence index, and $v$ for the
variable of block generating functions. In the one-sided color
completion, $Z=q^{-n}z$ and $Q=q^2$.

\subsection{Newton operators and inverse transforms}
We work with the oriented variable before passing to $X=z+z^{-1}$,
since the color shift is a substitution there but does not descend to
$X$. Put $D=\Z[Q^{\pm1}]$ and
\[
 P_k^+(Z)=\prod_{j=0}^{k-1}(Z-Q^j),\qquad
 \mathcal U^+=\varprojlim_kD[Z^{\pm1}]/(P_k^+).
\]
Over $\Z[q^{\pm1}]$, rescaling the oriented nodes
$q^nQ^r$ by $Z=q^{-n}z$ identifies their completion with this same
construction after the corresponding coefficient extension. 

\begin{proposition}\label{prop:band}
Multiplication $\mathsf Mf=Z f$ and the forward shift
$\mathsf Lf(Z)=f(QZ)$ are continuous operators on $\mathcal U^+$ and
satisfy $\mathsf L\mathsf M=Q\mathsf M\mathsf L$. In Newton coordinates,
\begin{equation}\label{eq:band}
 (\mathsf Mc)_k=Q^kc_k+c_{k-1},\qquad
 (\mathsf Lc)_k=Q^k\bigl(c_k+(Q^{k+1}-1)c_{k+1}\bigr),
 \qquad c_{-1}=0.
\end{equation}
\end{proposition}
\begin{proof}
The polynomial identities
\[
 Z P_k^+=P_{k+1}^++Q^kP_k^+,\qquad
 P_k^+(QZ)=Q^kP_k^+(Z)+Q^{k-1}(Q^k-1)P_{k-1}^+(Z)
\]
(the second for $k\ge1$) give the formulas. Multiplication preserves
the filtration, while the shift lowers its index by at most one, so
both extend continuously to the inverse limit. The relation
$\mathsf L\mathsf M=Q\mathsf M\mathsf L$ holds on polynomials and
hence by continuity.
\end{proof}

For $n\ge1$, let $f_r=\sum_{k\le r}h_kB_k^{(n)}(x_r^{(n)})$ and
write $\{j\}!=\prod_{i=1}^j\{i\}$, with $\{0\}!=1$.

The following inverse transform also yields the preservation of
$q$-holonomicity proved in \cite[Appendix~B]{FZ26}.

\begin{proposition}\label{prop:inverse}
Over $\Q(q)$ the inverse transform is
\begin{equation}\label{eq:inverse}
 h_k=\sum_{j=0}^k(-1)^{k-j}
 \frac{\{n+j-1\}!\{n+2j\}}
 {\{j\}!\{k-j\}!\{n+j+k\}!}\,f_j.
\end{equation}
For fixed $n$, the sequence $(h_k)$ is $q$-holonomic if and only if
$(f_r)$ is $q$-holonomic. In particular, $(c_{K,n,k})_{k\ge0}$ is
$q$-holonomic for every knot and every fixed rank $n\ge2$.
\end{proposition}
\begin{proof}
The coefficient of $f_j$ in a divided difference of order $k$ is
$\prod_{i\ne j,\,0\le i\le k}(x_j^{(n)}-x_i^{(n)})^{-1}$.
The node-difference factorization gives
\[
 \prod_{i\ne j}(x_j^{(n)}-x_i^{(n)})
 =(-1)^{k-j}\{j\}!\{k-j\}!
   \frac{\{n+j+k\}!}{\{n+j-1\}!\{n+2j\}},
\]
which proves \eqref{eq:inverse}. Both the forward kernel and this
inverse kernel are proper $q$-hypergeometric factorial quotients on their
finite triangular summation ranges. Closure of $q$-holonomic functions
under products and definite summation \cite{GL} proves both implications.
For a knot, take $f_r=J_r^{(n)}(K;q)$. Fixed-rank holonomicity and
integral Newton coefficients are supplied by
\cite[Appendix~B and Theorem~1.1]{FZ26}.
\end{proof}

\subsection{The block resolvent and its poles}
At a root of unity a coefficient recurrence becomes periodic, and
grouping one period at a time gives a transfer matrix that is
constant at quantum order zero. Its resolvent then determines the
higher orders recursively and bounds their poles.

Let $D$ be a characteristic-zero domain with fraction field $E$.
Suppose $\mathbf u_a(\varepsilon)\in D[[\varepsilon]]^r$ satisfies
\[
 \mathbf u_{a+1}=M_a(\varepsilon)\mathbf u_a,\qquad
 M_a(\varepsilon)=M_0+\sum_{j\ge1}M_j(a)\varepsilon^j,
\]
where $M_0\in\Mat_r(D)$ is independent of $a$ and
$M_j(a)\in\Mat_r(E[a])$ has degree at most $j$. Set
\[
 d(v)=\det(I-vM_0),\qquad
 G_m(v)=\sum_{a\ge0}[\varepsilon^m]\mathbf u_a(\varepsilon)v^a,
 \qquad \vartheta=v\frac{d}{dv}.
\]

\begin{theorem}\label{thm:resolvent}
Under these hypotheses,
\begin{equation}\label{eq:resolvent}
 G_m=(I-vM_0)^{-1}\left(\mathbf u_{0,m}
             +v\sum_{j=1}^mM_j(\vartheta)G_{m-j}\right),
 \qquad d(v)^{2m+1}G_m(v)\in D[v]^r.
\end{equation}
Over an algebraic closure of $E$, if $s_\lambda$ is the largest Jordan
block size of $M_0$ at a nonzero eigenvalue $\lambda$, the pole order at
$v=\lambda^{-1}$ is at most
\begin{equation}\label{eq:jordan-bound}
 (s_\lambda+1)m+s_\lambda.
\end{equation}
There are no other finite poles.
\end{theorem}
\begin{proof}
Extract $\varepsilon^m$ in the recurrence, multiply by $v^{a+1}$, and
sum over $a\ge0$. Multiplication by a polynomial in $a$ becomes the
same polynomial in $\vartheta$, giving the first identity. The inverse
matrix has denominator $d(v)$ and $d(0)=1$. If
$d^{2(m-j)+1}G_{m-j}$ is polynomial, application of an operator of degree
at most $j$ gives denominator at most $d^{2(m-j)+1+j}$. Multiplication
by the resolvent increases that exponent by at most one. Since
$2(m-j)+1+j+1\le2m+1$ for $j\ge1$, induction proves polynomiality over
$E$. The left side has power-series coefficients in $D$, so its finite
polynomial coefficients belong to $D$ as well.

Locally at $v=\lambda^{-1}$ the resolvent has pole order $s=s_\lambda$.
An Euler derivative increases pole order by at most one. By induction
a summand with index $j\ge1$ has order at most
\[
 s+\bigl((s+1)(m-j)+s\bigr)+j
 =(s+1)m+2s-sj\le(s+1)m+s.
\]
The initial-data term has order at most $s$. This proves
\eqref{eq:jordan-bound}; the recursion introduces no other poles.
\end{proof}

\begin{proposition}\label{prop:highest}
Under the hypotheses of \cref{thm:resolvent}, let $\lambda\ne0$ be
simple over a splitting field. Choose $M_0v_\lambda=\lambda v_\lambda$ and
$w_\lambda M_0=\lambda w_\lambda$, with $w_\lambda v_\lambda=1$.
Write $M_1(a)=aR+S$ and set
\[
 \alpha=w_\lambda\mathbf u_{0,0},\qquad
 \kappa=\frac{w_\lambda Rv_\lambda}{2\lambda}.
\]
For any constant row vector $e$, one has
\begin{equation}\label{eq:highest}
 \lim_{v\to\lambda^{-1}}(1-\lambda v)^{2m+1}eG_m(v)
 =\alpha(ev_\lambda)\frac{(2m)!}{m!}\kappa^m.
\end{equation}
In particular the pole is exactly $2m+1$ whenever the right side is
nonzero. 
\end{proposition}
\begin{proof}
Partial fractions in \cref{thm:resolvent} express the coefficients,
apart from a finite initial segment, as sums of polynomial multiples of
$\mu^a$, where $\mu$ ranges over the nonzero eigenvalues. The
$\lambda$-part has degree at most $2m$. Comparing its top degree in the
recurrence shows that its leading vector is $b_mv_\lambda$.
At degree $2m-1$, projection by $w_\lambda$ cancels the unknown next
vector coefficient. The term $aR$ acting on the order-$(m-1)$ leading
coefficient is the only remaining contribution at this degree: terms
with $j\ge2$ have degree at most $j+2(m-j)\le2m-2$ in the
$\lambda$-part. Thus
\[
 2m\lambda b_m=(w_\lambda Rv_\lambda)b_{m-1},\qquad b_0=\alpha.
\]
It follows that $b_m=\alpha\kappa^m/m!$. Distinct exponential sequences
are linearly independent over polynomial coefficients, so other
spectral parts do not affect this comparison. Finally
$\sum_{a\ge0}a^{2m}(\lambda v)^a$ has highest pole coefficient $(2m)!$.
\end{proof}

\begin{corollary}
\label{cor:laurent-block}
Suppose $\mathbf c_{k+1}(Q)=T(Q^k,Q)\mathbf c_k(Q)$, with
$T(x,Q)\in\Mat_r(\Z[x^{\pm1},Q^{\pm1}])$ and
$\mathbf c_k(Q)\in\Z[Q^{\pm1}]^r$. Let $\omega$ have order $\ell$.
For every residue $b$, the sequence
$\mathbf u_a=\mathbf c_{\ell a+b}(\omega(1+\varepsilon))$
satisfies \cref{thm:resolvent} over $D=\Z[\omega]$.
The polynomial $d(v)$ is independent of $b$.
\end{corollary}
\begin{proof}
The block transfer is the ordered product of $\ell$ matrices
$T(Q^{\ell a+b+j},Q)$. Every monomial in an entry of this product
becomes a root of unity times $(1+\varepsilon)^{ua+w}$, with $u,w$
integers. Its order-$j$ coefficient is a polynomial in $a$ of degree at
most $j$. The constant block matrices for different $b$ are cyclic
permutations of the same factors. The identity
$\det(I-vAB)=\det(I-vBA)$ shows that their polynomials $d(v)$ agree.
\end{proof}
To apply \cref{cor:laurent-block} at a given root, one needs a
transfer matrix regular there. A Laurent transfer in $q^k$ has
coefficient period dividing $\ord(\zeta)$ at $q=\zeta$, and a
transfer in $q^{2k}$ has period dividing $\ord(\zeta^2)$; the Newton
kernel always has period $\ord(\zeta^2)$.

\subsection{Complete Newton periods and full Taylor expansions}
Let $\zeta$ be a root of unity, set $\ell=\ord(\zeta^2)$ and
$\sigma=\zeta^{n\ell}\in\{1,-1\}$, and put
\begin{equation}\label{eq:rank-period-polynomial}
 W_{n,\zeta}(X)=\Theta_\ell(X)-2\sigma.
\end{equation}

\begin{proposition}\label{prop:rank-full-period}
For $a\ge0$ and $0\le b<\ell$,
\begin{equation}\label{eq:rank-full-period}
 B_{\ell a+b}^{(n)}(X;\zeta)
 =W_{n,\zeta}(X)^aB_b^{(n)}(X;\zeta).
\end{equation}
\end{proposition}
\begin{proof}
For an independent unit $v$ and a primitive $\ell$th root $\omega$,
\[
 \prod_{r=0}^{\ell-1}(X-v\omega^r-v^{-1}\omega^{-r})
       =\Theta_\ell(X)-v^\ell-v^{-\ell}.
\]
Over the rational function field in $v$, both sides are monic of
degree $\ell$ and have the same roots. Specialize to
$v=\zeta^n$ and $\omega=\zeta^2$. The two final terms are both
$\sigma$. Splitting the indices into complete periods gives the
result, including all root multiplicities.
\end{proof}

At quantum order $h$, at most
$h$ complete periods of a Newton product receive a positive-order
derivative. Every remaining period contributes a factor $W_{n,\zeta}$.
Consequently
\begin{equation}\label{eq:period-kernel-adic-bound}
 [\epsilon^h]B_{\ell a+b}^{(n)}(X;\zeta(1+\epsilon))
 \in W_{n,\zeta}^{\max(a-h,0)}\mathbb Z[\zeta][X].
\end{equation}
For any integral Newton element, each full Taylor coefficient at
$\zeta$ therefore converges in the $W_{n,\zeta}$-adic completion.

Write $W=W_{n,\zeta}$ and
$B_b^\zeta=B_b^{(n)}(X;\zeta)$. For each integer $a\ge0$, define
rational kernel corrections $K_{n,b,h}(a;X)\in\mathbb Q(\zeta)(X)$ by
\begin{equation}\label{eq:rank-kernel-correction}
 B_{\ell a+b}^{(n)}(X;\zeta(1+\epsilon))
   =W^aB_b^\zeta\sum_{h\ge0}K_{n,b,h}(a;X)\epsilon^h.
\end{equation}

\begin{lemma}\label{lem:rank-kernel-degree}
One has $K_{n,b,0}=1$. For $h\ge1$, $K_{n,b,h}$ is polynomial
in $a$ of degree at most $2h-1$, with coefficients having poles only
on $W=0$.
\end{lemma}
\begin{proof}
Write the logarithm of the quotient in
\eqref{eq:rank-kernel-correction} as
$\sum_{h\ge1}\Lambda_h(a;X)\epsilon^h$. Within a fixed residue
class of the index modulo $\ell$, the coefficient of order $h$
in the logarithm of one factor is polynomial in that index of
degree at most $h$. Summing over $a$ complete periods raises the
degree by at most one; the final prefix has degree at most $h$.
Thus $\deg_a\Lambda_h\le h+1$.

At first order, set
\[
 \alpha_r=\zeta^{n+2r},\qquad
 \bar x_r=\alpha_r+\alpha_r^{-1},\qquad
 \bar y_r=\alpha_r-\alpha_r^{-1}\quad(0\le r<\ell).
\]
Summing the derivatives of the factors gives
\begin{align*}
 \Lambda_1(a;X)
 &=-\sum_{r=0}^{\ell-1}
   \frac{\bigl(a(n+2r)+\ell a(a-1)\bigr)\bar y_r}{X-\bar x_r}\\*
 &\quad-\sum_{r=0}^{b-1}
   \frac{(n+2\ell a+2r)\bar y_r}{X-\bar x_r}.
\end{align*}
Its coefficient of $a^2$ is
\[
 -\ell\sum_{r=0}^{\ell-1}\frac{\bar y_r}{X-\bar x_r}=0.
\]
Indeed, the involution $r\mapsto-n-r$ sends $\alpha_r$ to
$\alpha_r^{-1}$, so it preserves $\bar x_r$ and negates
$\bar y_r$. At a fixed residue, $\bar y_r=0$. It follows that
$\deg_a\Lambda_1\le1$.

For $h\ge2$, the earlier estimate gives
$\deg_a\Lambda_h\le h+1\le2h-1$. In the exponential, a product
of $s$ positive-order logarithmic coefficients of total order $h$
has degree at most $2h-s\le2h-1$. This proves the degree bound for
$K_{n,b,h}$. Its possible denominators are products of the node
factors $X-\bar x_r$, all of which divide $W$. The constant term
of the quotient is one, giving $K_{n,b,0}=1$.
\end{proof}

\begin{lemma}\label{lem:monic-period-integrality}
Let $D$ be a domain with fraction field $E$, let $W\in D[X]$ be
monic of positive degree, and let $d(v)\in D[v]$ satisfy $d(0)=1$.
Suppose that, inside $\widehat{E[X]}^{\,(W)}[W^{-1}]$,
\[
 f\in E[X,W^{-1},d(W)^{-1}]
       \cap\widehat{D[X]}^{\,(W)}.
\]
If $d(W)^e f\in E[X,W^{-1}]$ for an integer $e\ge0$, then
$d(W)^e f\in D[X]$.
\end{lemma}
\begin{proof}
Division by the monic polynomial $W$ gives unique expansions
$\sum_a r_a(X)W^a$ with $\deg r_a<\deg W$.
In $E[X,W^{-1}]$ this sum is finite and may have negative indices.
An element that also belongs to the $W$-adic completion has no
negative coordinates, so it lies in $E[X]$.
Apply this to $d(W)^e f$. Its coordinates lie in $D[X]$, since
$f$ belongs to the $D$-coefficient completion. Only finitely many
are nonzero, and their sum belongs to $D[X]$.
\end{proof}

\begin{theorem}\label{thm:rank-full-series}
Let $F=\sum_{k\ge0}c_k(q)B_k^{(n)}\in\C_n$. Suppose
$d(v)\in\mathbb Z[\zeta][v]$, $d(0)=1$, and $\beta\ge1$ satisfy
\begin{equation}\label{eq:rank-block-hypothesis}
 d(v)^{\beta+2m}
 \sum_{a\ge0}[\epsilon^m]c_{\ell a+b}(\zeta(1+\epsilon))v^a
       \in\mathbb Z[\zeta][v]
\end{equation}
for every $b<\ell$ and $m\ge0$. Then
\begin{equation}\label{eq:rank-full-denominator}
 [\epsilon^m]F(X;\zeta(1+\epsilon))
  =\frac{U_{n,\zeta,m}(X)}
         {d(W_{n,\zeta}(X))^{\beta+2m}},
 \qquad U_{n,\zeta,m}\in\mathbb Z[\zeta][X].
\end{equation}
Every positive-order kernel correction has $d(W)$-denominator
exponent at most $\beta+2m-1$.
\end{theorem}
\begin{proof}
Write $H_{b,j}(v)$ for the block series in
\eqref{eq:rank-block-hypothesis}, and set $\vartheta_v=v\,d/dv$.
Multiplication of its $a$th coefficient by a polynomial in $a$
is represented by the same polynomial in $\vartheta_v$. Consequently
the contribution with kernel order $h$ to the full coefficient of
order $m$ is
\[
 S_{m,h}(X)=\sum_{b=0}^{\ell-1}B_b^\zeta(X)
 \left.K_{n,b,h}(\vartheta_v;X)H_{b,m-h}(v)\right|_{v=W(X)},
 \qquad 0\le h\le m.
\]
For $h\ge1$, \cref{lem:rank-kernel-degree} bounds the differential
order by $2h-1$. Each derivative increases the $d(v)$-denominator
exponent by at most one. Thus $S_{m,h}$ has
$d(W)$-denominator exponent at most
\[
 \beta+2(m-h)+2h-1=\beta+2m-1.
\]
For $h=0$, the exponent is at most $\beta+2m$. Any additional
possible poles come from $W=0$.

For each $S_{m,h}$, \eqref{eq:period-kernel-adic-bound} shows
that its defining sum converges in the $W$-adic completion of
$\mathbb Z[\zeta][X]$. After multiplication by the stated power of
$d(W)$, the rational expression lies in
$\mathbb Q(\zeta)[X,W^{-1}]$.
\Cref{lem:monic-period-integrality} first removes its possible
negative $W$-coordinates and then shows that its polynomial
coefficients are integral. Thus
\[
 d(W)^{\beta+2m}S_{m,0}\in\mathbb Z[\zeta][X],\qquad
 d(W)^{\beta+2m-1}S_{m,h}\in\mathbb Z[\zeta][X]
       \quad(1\le h\le m).
\]
Summing over $h$ proves \eqref{eq:rank-full-denominator} and the
stronger bound for every positive-order kernel correction.
\end{proof}

At $q=1$ the scalar series $\sum_kc_{K,n,k}(1)v^k$ is
$\delta_K(v)^{-(n-1)}$. Writing it as a matrix coefficient of
$(I-vM_0)^{-1}$ shows that $\delta_K(v)^{n-1}\mid\det(I-vM_0)$ in
$\mathbb Q[v]$ for any regular transfer matrix realizing these
coefficients.

\subsection{Cyclotomic recentering}\label{sec:recentering}
Fix $\Delta\in\mathbb Z[z^{\pm1}]$ with $\Delta(z)=\Delta(z^{-1})$
and $\Delta(1)=1$. For a root of unity $\zeta$, set
\[
 \ell_\zeta=\ord(\zeta^2),\qquad
 \sigma_\zeta=\zeta^{n\ell_\zeta},\qquad
 D_\zeta(z)=\Delta(\sigma_\zeta z^{\ell_\zeta}).
\]
By \eqref{eq:period-kernel-adic-bound}, every element of $\mathcal C_n$
has a full Taylor expansion at $q=\zeta$ in the specialized Newton
completion. Let $\mathcal C_{n,\Delta}^{\mathrm{rat}}$ be the subring
whose Taylor coefficients, after $X=z+z^{-1}$, have rational
representatives
\[
 f_\zeta(z,t)\in
 \mathbb Z[\zeta,z^{\pm1},D_\zeta(z)^{-1}][[t]],
 \qquad t=q-\zeta,
\]
for every $\zeta$. The representation is unique: the rational ring
embeds in the period completion, and $D_\zeta$ is a unit there
because it has value $\Delta(1)=1$ at every color node.

\begin{theorem}\label{thm:rank-recentering}
Let $F\in\mathcal C_{n,\Delta}^{\mathrm{rat}}$. Suppose that roots
$\zeta,\zeta'$ lie in a finite $p$-adic cyclotomic field with ring
of integers $\mathcal O$, and $v_p(\zeta'-\zeta)>0$.
In the common completed ring
\[
 \mathscr R=\widehat{\mathcal O[z^{\pm1},
                 D_\zeta(z)^{-1},D_{\zeta'}(z)^{-1}]}_p
\]
one has
\begin{equation}\label{eq:rank-recentering}
 f_\zeta(z,t+\zeta'-\zeta)=f_{\zeta'}(z,t).
\end{equation}
The map $F\mapsto(f_\zeta)_\zeta$ is injective and multiplicative.
\end{theorem}
\begin{proof}
Put $c=\zeta'-\zeta$ and write
$f_\zeta(z,t)=\sum_{j\ge0}f_{\zeta,j}(z)t^j$. The coefficient
of $t^h$ after recentering is
\[
 \sum_{j\ge h}\binom jh c^{j-h}f_{\zeta,j}(z).
\]
It converges in $\mathscr R$ because the coefficients are integral
there and $v_p(c)>0$.

Choose a uniformizer $\varpi$ of $\mathcal O$ and substitute
$z=(\zeta')^n(1+w)$. The two inverted denominators have unit
constant coefficient modulo $\varpi$. Indeed, $\zeta$ and
$\zeta'$ have the same residue, and
\[
 \sigma_\zeta((\zeta')^n)^{\ell_\zeta}
       \equiv\zeta^{2n\ell_\zeta}=1\pmod\varpi.
\]
The corresponding identity for $\zeta'$ is exact. Polynomial
substitution is injective modulo every power of $\varpi$;
localization by these units preserves injectivity. Taking the
separated limit gives
$\mathscr R[[t]]\hookrightarrow\mathcal O[[w,t]]$.

In this target the Newton series is
\begin{equation}\label{eq:recentering-common-series}
 \sum_{k\ge0}c_k(\zeta'+t)
 B_k^{(n)}\bigl((\zeta')^n(1+w)
             +(\zeta')^{-n}(1+w)^{-1};\zeta'+t\bigr).
\end{equation}
Let $d_0$ be the multiplicative order of the residue of
$(\zeta')^2$ and put $\mathfrak m=(\varpi,w,t)$. Every node
factor whose index is divisible by $d_0$ belongs to $\mathfrak m$.
Hence
\[
 B_k^{(n)}\bigl((\zeta')^n(1+w)
             +(\zeta')^{-n}(1+w)^{-1};\zeta'+t\bigr)
       \in\mathfrak m^{\lfloor k/d_0\rfloor}.
\]
The same estimate holds with the quantum center $\zeta$, since
the two centers have the same residue. It proves convergence of
both Newton series. The substitution $t\mapsto t+c$ preserves
$\mathfrak m$ and is continuous. The recentering identity for
polynomial partial sums therefore passes to the limit and gives
\eqref{eq:rank-recentering} in the target. Injectivity of the
substitution proves it in $\mathscr R[[t]]$.

Taylor expansion is multiplicative. If every component vanishes,
the component at $\zeta=1$ has zero joint Taylor series at
$(q,z)=(1,1)$. The injective substitution
$u=z+z^{-1}-2$, followed by \cref{thm:quantum-neighborhoods},
then gives $F=0$.
\end{proof}

For multiplicative coordinates put $\gamma=\zeta'/\zeta$ and
$f_\zeta=\sum_{j\ge0}F_{\zeta,j}(q/\zeta-1)^j$.
If $e\ge1$ satisfies $e\,v_p(\gamma-1)\ge1$, then
\begin{equation}\label{eq:rank-finite-recentering}
 F_{\zeta',h}\equiv\gamma^h
       \sum_{j=h}^{h+ev-1}\binom jh(\gamma-1)^{j-h}F_{\zeta,j}
       \pmod{p^v},\qquad h\ge0,\quad v\ge1.
\end{equation}
This follows from $q/\zeta-1=(\gamma-1)+\gamma(q/\zeta'-1)$, since
every omitted term has valuation at least $v$. If $\gamma$ is a primitive
$p^{a+1}$th root, one may take $e=p^a(p-1)$.

\subsection{Finite reconstruction of normal knot jets}\label{sec:finite-normal-jets}
Fix $n\ge2$, $1\le s\le\lfloor n/2\rfloor$, and a root of unity
$\zeta$. At $q=\zeta$, some color nodes coincide with the earlier node
$a_{n-2s}$, and we show that each normal coefficient there is a
finite linear combination of Newton coefficients, with a sharp
cutoff. Write $\ell=\ord(\zeta^2)$ and $u=X-a_{n-2s}(q)$.
By \eqref{eq:earlier-kernel}, the boundary kernel is
\begin{equation}\label{eq:finite-normal-kernel}
 \mathcal B_k(u,q):=B_k^{(n)}(a_{n-2s}(q)+u;q)
       =\prod_{i=0}^{k-1}\bigl(u-\{s+i\}\{n-s+i\}\bigr).
\end{equation}
Let the distinct collision residues, in increasing order, be
\begin{equation}\label{eq:normal-collision-residues}
 \{\rho_0<\cdots<\rho_{\nu-1}\}
  =\{(-s)\bmod\ell,(s-n)\bmod\ell\}\subset\{0,\ldots,\ell-1\}.
\end{equation}
Thus $\nu$ is one or two. Put
\begin{equation}\label{eq:normal-finite-cutoff}
 e_\partial=2/\nu,\qquad
 N_j=\ell\left\lfloor\frac j\nu\right\rfloor+\rho_{j\bmod\nu},
 \qquad N_{j,m}=N_{j+\lfloor m/e_\partial\rfloor}.
\end{equation}
The integer $e_\partial$ is two precisely when $\ell\mid n-2s$.

\begin{theorem}\label{thm:finite-normal-jets}
Let $F=\sum_{k\ge0}c_k(q)B_k^{(n)}\in\mathcal C_n$, and write
$t_j(F)=t_{n,s,j}(F)\in\mathcal H$ for its normal coefficients.
For every $j,m\ge0$,
\begin{align}
 t_j(F)(\zeta)
   &=\sum_{k=0}^{N_j}c_k(\zeta)[u^j]\mathcal B_k(u,\zeta),
       \label{eq:finite-normal-value}\\
 [\epsilon^m]t_j(F)(\zeta(1+\epsilon))
   &=[u^j\epsilon^m]\sum_{k=0}^{N_{j,m}}
       c_k(\zeta(1+\epsilon))
       \mathcal B_k(u,\zeta(1+\epsilon)).
       \label{eq:finite-mixed-normal}
\end{align}
Both expressions belong to $\mathbb Z[\zeta]$. The upper index
$N_j$ in \eqref{eq:finite-normal-value} is sharp for arbitrary
integral Newton elements.
\end{theorem}
\begin{proof}
Let
\begin{equation}\label{eq:collision-count}
 M_{n,s,\ell}(k)=\#\{0\le i<k:
                 i\equiv-s\text{ or }i\equiv s-n\pmod\ell\}.
\end{equation}
A constant factor $\{s+i\}_\zeta\{n-s+i\}_\zeta$ vanishes
exactly at these indices. Every such index contributes one factor
$u$ to \eqref{eq:finite-normal-kernel}; all other constant terms
are nonzero. Hence
\begin{equation}\label{eq:collision-order}
 \ord_u\mathcal B_k(u,\zeta)=M_{n,s,\ell}(k).
\end{equation}
The index $N_j$ is the location of the $(j+1)$st vanishing factor.
For $k>N_j$ the order in \eqref{eq:collision-order} is at least
$j+1$, proving \eqref{eq:finite-normal-value}. At $k=N_j$ that
order is exactly $j$ and the coefficient of $u^j$ is a nonzero
product. The Newton element $B_{N_j}^{(n)}$ proves sharpness.

At a collision index the scalar
$\{s+i\}\{n-s+i\}$ has exact $\epsilon$-adic order
$e_\partial$ after $q=\zeta(1+\epsilon)$. Indeed, each quantum
integer with a positive index has a simple zero at a root where
it vanishes. Exactly one vanishes when the two residue classes
are distinct, and both vanish when they coincide. In a coefficient
of $u^j$, at most $j$ of these scalar factors are omitted. Thus
\[
 [u^j]\mathcal B_k(u,\zeta(1+\epsilon))
       \in\epsilon^{e_\partial\max(M_{n,s,\ell}(k)-j,0)}
                   \mathbb Z[\zeta][[\epsilon]].
\]
Multiplication by the Laurent coefficient $c_k$ cannot lower this
order. For $k>N_{j,m}$ it exceeds $m$, proving
\eqref{eq:finite-mixed-normal}. The finite sums use integral
Laurent polynomials and their Hasse coefficients, which also
proves integrality.
\end{proof}

In rank two at its boundary $X=2$, the formulas become
\begin{equation}\label{eq:rank-two-normal-cutoff}
 N_j=(j+1)\ell-1,\qquad
 N_{j,m}=\bigl(j+\lfloor m/2\rfloor+1\bigr)\ell-1.
\end{equation}
For $n=3$, $s=1$, and $\ell=5$, the collision indices are
$3,4,8,9,\ldots$, so $N_0=3$, $N_1=4$, and $N_2=8$.
For the same rank at $q=i$ the residues are $0,1$, giving $N_j=j$.

Combined with the inverse Newton transform, the cutoff gives
reconstruction from finitely many colored invariants known to finite
quantum precision.

\begin{corollary}\label{cor:finite-normal-color-data}
Fix $j,m$, let $N=N_{j,m}$, and set
\begin{equation}\label{eq:normal-color-precision-loss}
 d_N=\max_{0\le r\le N}
    \sum_{\substack{0\le i\le N\\i\ne r}}
       \bigl(\mathbf 1_{\ell\mid r-i}
             +\mathbf 1_{\ell\mid n+r+i}\bigr)\le2N.
\end{equation}
The finite Taylor data
\begin{equation}\label{eq:normal-color-data}
 J_r^{(n)}(K;\zeta(1+\epsilon))
        \pmod{\epsilon^{m+d_N+1}},\qquad 0\le r\le N,
\end{equation}
determine
$[\epsilon^m]\tau_{K,j}^{(n,s)}(\zeta(1+\epsilon))$.
In particular, the Laurent polynomials
$J_0^{(n)}(K;q),\ldots,J_{N_j}^{(n)}(K;q)$ determine
$\tau_{K,j}^{(n,s)}(\zeta)$.
\end{corollary}
\begin{proof}
Before specialization, the divided-difference formula is
\[
 c_{K,n,k}(q)=\sum_{r=0}^k
       \frac{J_r^{(n)}(K;q)}
       {\prod_{\substack{0\le i\le k\\i\ne r}}
                         (x_r^{(n)}-x_i^{(n)})}.
\]
It is equivalent to \eqref{eq:inverse}. Each factor in the
denominator has $\epsilon$-order
$\mathbf 1_{\ell\mid r-i}+\mathbf 1_{\ell\mid n+r+i}$ by the
node-difference factorization. For $k\le N$, the total order is
at most $d_N$. Division in $\mathbb Q(\zeta)((\epsilon))$
therefore recovers each $c_{K,n,k}$ modulo $\epsilon^{m+1}$
from \eqref{eq:normal-color-data}. Formula
\eqref{eq:finite-mixed-normal} then gives the required coefficient.
For $m=0$ the same argument may be performed on the exact Laurent
polynomials before evaluating at $\zeta$.
\end{proof}

If the specialized knot element is rational, its normal sequence
satisfies a linear recurrence.

\begin{proposition}\label{prop:rational-normal-recurrence}
Suppose that at the chosen boundary
\[
 F_{K,n}(a_{n-2s}(\zeta)+u;\zeta)=\frac{P(u)}{D(u)},\qquad
 P,D\in\mathbb Z[\zeta][u],\quad D(0)=1.
\]
Write $D(u)=1+\sum_{a=1}^h d_au^a$ and
$\tau_j=\tau_{K,j}^{(n,s)}(\zeta)$. Then
\begin{equation}\label{eq:rational-normal-recurrence}
 \tau_j=[u^j]P(u)-\sum_{a=1}^{\min(j,h)}d_a\tau_{j-a}
       \qquad(j\ge0).
\end{equation}
If $P$ and $D$ are coprime over $\mathbb Q(\zeta)$, the minimal
order of an eventual homogeneous constant-coefficient recurrence
for $(\tau_j)$ over this field is $\deg D$.
\end{proposition}
\begin{proof}
Multiply the normal series by $D$ and compare coefficients. Beyond
$\deg P$, the resulting recurrence is homogeneous. Conversely, a
homogeneous recurrence of order $r$ implies that multiplying the
generating series by a polynomial of degree at most $r$ and constant
coefficient one gives a polynomial. Coprimality then forces $D$
to divide that recurrence polynomial, so $\deg D\le r$.
\end{proof}

For the oriented color coordinate put $z_0=\zeta^{n-2s}$ and
$z=z_0(1+t)$. The normal parameter becomes
\begin{equation}\label{eq:oriented-boundary-coordinate}
 u_s(t)=z+z^{-1}-z_0-z_0^{-1}
       =\frac{t(z_0-z_0^{-1}+z_0t)}{1+t}.
\end{equation}
Its order is $e_\partial$, equal to one for $z_0^2\ne1$ and two
for $z_0^2=1$. Substitution in the normal series gives
\begin{equation}\label{eq:oriented-normal-substitution}
 F_{K,n}(z_0(1+t)+z_0^{-1}(1+t)^{-1};\zeta)
       =\sum_{j\ge0}\tau_{K,j}^{(n,s)}(\zeta)u_s(t)^j.
\end{equation}
If $a$ is the leading coefficient of $u_s(t)$, the coefficient of
$t^{e_\partial j}$ is $a^j\tau_{K,j}^{(n,s)}(\zeta)$ plus a
combination of earlier normal coefficients. Thus the oriented series
recovers the normal coefficients triangularly over
$\mathbb Q(\zeta)$. At $z_0^2=1$, one has
$u_s(t)=z_0t^2/(1+t)$, and this recovery is integral.

\subsection{Laurent lattices and period matrices}\label{sec:laurent-lattices}
A finite-order recurrence with nonzero first and last coefficients
has a companion realization over a rational function field.
Write $\xi$ for the exponential of the index being shifted and
$Q$ for the shift multiplier, so that $\sigma(\xi)=Q\xi$.
A Laurent forward realization has matrix entries in
$\mathbb Z[Q^{\pm1},\xi^{\pm1}]$. A lattice preserved by both
forward and inverse transport has a matrix in
$\GL_r(\mathbb Z[Q^{\pm1},\xi^{\pm1}])$.
In \cref{sec:explicit-recurrences,sec:twist-spectral} the shifted
index is the Newton index.

A rational change of frame $H$ transforms a matrix $C$ into
$B=\sigma(H)CH^{-1}$. Its determinant gives a necessary condition
for the existence of an invertible Laurent lattice.

\begin{proposition}\label{prop:orbit}
Let $Q$ be transcendental over a field $k$, put
$R_0=k(Q)[\xi^{\pm1}]$, and let $C\in\GL_r(k(Q,\xi))$.
On each $\sigma$-orbit of closed points of $\Spec R_0$, the sum
of the orders of $\det C$ is invariant under rational change of
frame. If $C$ is rationally equivalent to a matrix in $\GL_r(R_0)$,
all these orbit sums are zero.
\end{proposition}
\begin{proof}
Taking determinants gives
\[
 \operatorname{div}(\det B)=\operatorname{div}(\det C)
       +\operatorname{div}(\sigma\det H)-\operatorname{div}(\det H).
\]
Each divisor has finite support, so summing along an orbit cancels
the last two terms. A Laurent unit has zero divisor on
$\Spec R_0$.
\end{proof}
The rank-two coefficient recurrences of $5_2$ and $6_1$ have a
nonzero determinant orbit sum. Their forward Laurent realizations
and their period matrices are given in
\cref{prop:twist-C-recurrences,prop:twist-periods}.

\begin{proposition}\label{prop:period}
Let $\omega$ have order $\ell$ in a field $k$, let
$B(\xi)\in\Mat_r(k[\xi^{\pm1}])$, and put
\[
 M_\ell(\xi)=B(\omega^{\ell-1}\xi)\cdots B(\xi).
\]
Then
\begin{equation}\label{eq:period-charpoly}
 \det(LI-M_\ell(\xi))\in k[\xi^{\pm\ell},L].
\end{equation}
If $B\in\GL_r(k[\xi^{\pm1}])$, one also has
\[
 M_\ell(\omega\xi)=B(\xi)M_\ell(\xi)B(\xi)^{-1}.
\]
\end{proposition}
\begin{proof}
The shifted product is a cyclic permutation of the factors of
$M_\ell$. The identity $\det(LI-AB)=\det(LI-BA)$ shows that its
characteristic polynomial is unchanged. Hence every coefficient
is invariant under $\xi\mapsto\omega\xi$ and has only powers
of $\xi^\ell$. When $B$ is invertible, the same cyclic permutation
is the displayed conjugacy.
\end{proof}
The polynomial in \eqref{eq:period-charpoly} depends on $\ell$: the
constant matrix
$B=\left(\begin{smallmatrix}2&1\\1&1\end{smallmatrix}\right)$
has characteristic polynomial $L^2-3L+1$, whereas $B^2$ has
$L^2-7L+1$. For the knots in \cref{sec:twist-spectral}, the period
polynomial at $q=i$ equals the classical $C$-polynomial with $\xi$
replaced by $\xi^2$, and the associated line modules decide whether
the period matrix is conjugate over the Laurent ring to the companion
matrix of this polynomial.

\subsection{Comparison of spectral modules}\label{sec:spectral-comparison}
Let $A_0=k[\xi^{\pm1}]$ and let $p(Y)\in A_0[Y]$ be monic of degree
$r$. Assume that $S=A_0[Y]/(p)$ is the coordinate ring of a smooth
connected affine curve, so that $S$ is a Dedekind domain. A matrix
$M\in\Mat_r(A_0)$ with characteristic polynomial $p$ makes $A_0^r$
an $S$-module $\mathscr P_M$ by letting $Y$ act as $M$.
Over $\Frac(A_0)$, irreducibility of $p$ makes this module
one-dimensional over $\Frac(A_0)[Y]/(p)$. Thus $\mathscr P_M$ is
torsion-free of rank one over $S$, and hence invertible.

Two such matrices are conjugate over $A_0$ exactly when their
$S$-modules are isomorphic; this is the invertible-module case of the
Latimer--MacDuffee correspondence \cite{LM33,KS24}.
Write $\Comp(p)$ for multiplication by $Y$ on $S$ in the basis
$1,Y,\ldots,Y^{r-1}$; its spectral module is free.

For a finite locally free map $\pi$ of degree $d$, the norm of a
line module is
\begin{equation}\label{eq:norm-definition}
 \Nm_\pi(\mathscr L)=\det(\pi_*\mathscr L)
                      \otimes\det(\pi_*\mathcal O)^{-1}.
\end{equation}
It is multiplicative and satisfies
$\Nm_\pi(\pi^*\mathscr N)=\mathscr N^{\otimes d}$
\cite[Lemmas~31.18.2 and~31.18.6]{Stacks}.

\begin{theorem}\label{thm:stabilization}
Let $M_1,M_2\in\Mat_r(A_0)$ have characteristic polynomial $p$,
and put
$\mathscr D=\mathscr P_{M_1}\otimes_S\mathscr P_{M_2}^{-1}$.
For every positive integer $m$,
\begin{equation}\label{eq:spectral-stabilization}
 M_1^{\oplus m}\sim_{\GL_{rm}(A_0)}M_2^{\oplus m}
       \quad\Longleftrightarrow\quad
 \mathscr D^{\otimes m}\simeq S.
\end{equation}
If $\mathscr D$ has exact order $d$ and $d$ is invertible in $k$,
a connected finite \'etale cover of $\Spec S$ of degree $d$
trivializes it. Every finite locally free trivializing cover has
degree divisible by $d$.
\end{theorem}
\begin{proof}
The conjugacy in \eqref{eq:spectral-stabilization} is equivalent
to
$\mathscr P_{M_1}^{\oplus m}\simeq\mathscr P_{M_2}^{\oplus m}$
as $S$-modules. The Steinitz decomposition over a Dedekind domain is
\[
 \mathscr P^{\oplus m}\simeq S^{m-1}\oplus\mathscr P^{\otimes m}.
\]
For invertible ideals $I,J$, choose integral
representatives in their classes with $I+J=S$, using approximation
to make the second avoid the primes dividing the first. The addition
map $I\oplus J\to S$ is surjective, with kernel
$I\cap J=IJ$, and splits. This gives
$I\oplus J\simeq S\oplus IJ$; iteration gives the decomposition.
Taking determinants and applying this decomposition show that the
two direct sums are isomorphic precisely when
$\mathscr P_{M_1}^{\otimes m}\simeq\mathscr P_{M_2}^{\otimes m}$.
This is the condition on $\mathscr D$.

Choose a trivialization of $\mathscr D^{\otimes d}$. It defines
an algebra
\[
 T=\bigoplus_{j=0}^{d-1}\mathscr D^{-j}.
\]
Locally $T=S[x]/(x^d-u)$ for a unit $u$. Both $x$ and
$dx^{d-1}$ are units, so $\Spec T\to\Spec S$ is finite \'etale
of degree $d$. Its tautological root trivializes $\mathscr D$.

If a finite locally free cover of degree $e$ trivializes
$\mathscr D$, taking the norm in \eqref{eq:norm-definition}
gives $\mathscr D^{\otimes e}\simeq S$, whence $d\mid e$.
Each nonempty connected component of $\Spec T$ is finite \'etale
and surjective over the connected base, and still trivializes
$\mathscr D$. Its positive degree is therefore divisible by $d$.
Since the total degree is $d$, $\Spec T$ has one component.
\end{proof}

When $M_2=\Comp(p)$, the class of $\mathscr D$ is
$[\mathscr P_{M_1}]$. For the period matrices of $5_2$ and $6_1$,
\cref{thm:elliptic-spectral} shows that this class has exact order
three.

\section{Explicit knot elements in every rank}\label{sec:examples}

\subsection{Four coefficient formulas}
Write
\[
 \qbin ak=\frac{\{a\}!}{\{k\}!\{a-k\}!},\qquad
 D_{n,k}(q)=\qbin{n+k-2}{k}
          =\prod_{j=1}^{n-2}\frac{\{k+j\}}{\{j\}}.
\]
The bracket is the symmetric Gaussian polynomial, and the product is
one when $n=2$. We use the chirality convention of
\cite[Section~9.2]{FZ26}, in which the displayed trefoil is left-handed.
The fixed-rank coefficients extracted there from the double-twist
formulas of \cite[Theorem~1.2 and Section~3]{CLZ21} are as follows.

\begin{proposition}[{\cite[Section~9.2]{FZ26}}]\label{prop:seed-coefficients}
For every $n\ge2$ and $k\ge0$,
\begin{align}
 c_{4_1,n,k}&=D_{n,k},\label{eq:four-seeds}\\
 c_{3_1,n,k}&=(-1)^kq^{k(2n+k-1)}D_{n,k},\notag\\
 c_{5_2,n,k}&=(-1)^kq^{4nk+3k(k-1)}D_{n,k}
       \sum_{j=0}^kq^{-2nj-3kj+j(j+2)}\qbin kj,\notag\\
 c_{6_1,n,k}&=q^{-2nk-2k(k-1)}D_{n,k}
       \sum_{j=0}^kq^{2nj+3kj-j(j+2)}\qbin kj.\notag
\end{align}
\end{proposition}
In the double-twist formula, the factor $D_{n,k}$ comes from the
kernel after $A=q^n$. Set
\begin{equation}\label{eq:four-Alexander-polynomials}
 (c_{3_1},c_{4_1},c_{5_2},c_{6_1})=(-1,1,-2,2),\qquad
 \delta_K(u)=1-c_Ku.
\end{equation}
The displayed Alexander polynomials follow by setting $q=1$ in
\eqref{eq:four-seeds} and applying
\cref{thm:rank-loop-denominators}.

Their classical fibers and first quantum coefficients are
\begin{equation}\label{eq:four-low-jets}
\begin{array}{c|c|c}
 K& A_{K,n,0}(u)&A_{K,n,1}(u)\\ \hline
 3_1&(1+u)^{-(n-1)}&
       -n(n-1)u(2+u)/(1+u)^{n+1}\\[1mm]
 4_1&(1-u)^{-(n-1)}&0\\[1mm]
 5_2&(1+2u)^{-(n-1)}&
       -n(n-1)u(6+5u)/(1+2u)^{n+1}\\[1mm]
 6_1&(1-2u)^{-(n-1)}&
       -n(n-1)u(2-u)/(1-2u)^{n+1}.
\end{array}
\end{equation}
To compute them, use $D_{n,k}(1)=\binom{n+k-2}{k}$ and
$D_{n,k}'(1)=0$, then differentiate the monomials in
\eqref{eq:four-seeds}. The inner Gaussian sums at $q=1$ are binomial
sums; their first moments reduce to sums of $j$ and $j^2$.
Summation over $k$ gives \eqref{eq:four-low-jets}.

\subsection{The figure-eight generating function and exact jet poles}
Put $b=n-1$ and let
\[
 \mathcal W_n=\{n-2,n-4,\ldots,2-n\}
\]
be the multiset of $b$ integer weights. Let $h_j(\mathcal W_n)$ be
its complete homogeneous symmetric polynomial of degree $j$.

\begin{theorem}\label{thm:figure-eight-product}
The Newton-coefficient generating function of the figure-eight knot is
\begin{equation}\label{eq:figure-eight-product}
 \sum_{k\ge0}c_{4_1,n,k}(q)v^k
       =\prod_{a\in\mathcal W_n}(1-q^av)^{-1}
       \quad\text{in }R[[v]].
\end{equation}
For $n\ge3$ and $j\ge0$, the coefficient $A_{4_1,n,2j}$ has a
pole of exact order $b+2j$ at $u=1$, and
\begin{equation}\label{eq:figure-eight-top-even}
 \left.(1-u)^{b+2j}A_{4_1,n,2j}(u)\right|_{u=1}
       =h_{2j}(\mathcal W_n)>0.
\end{equation}
One has $A_{4_1,n,1}=0$. For $j\ge1$, the coefficient
$A_{4_1,n,2j+1}$ has a pole of exact order $b+2j$, with
\begin{equation}\label{eq:figure-eight-top-odd}
 \left.(1-u)^{b+2j}A_{4_1,n,2j+1}(u)\right|_{u=1}
       =-j\,h_{2j}(\mathcal W_n).
\end{equation}
All numerators in these representations are integral polynomials.
\end{theorem}
\begin{proof}
The finite $q$-binomial identity
\[
 \prod_{i=0}^{b-1}(1-Q^iv)^{-1}
       =\sum_{k\ge0}\gaussbin{b+k-1}{k}{Q}v^k
\]
and $D_{n,k}=q^{-k(n-2)}
 \gaussbin{n+k-2}{k}{q^2}$ prove
\eqref{eq:figure-eight-product} after rescaling $v$.
Set $q=e^h$ and expand the right side with $v=u$.
It is even in $h$, because the weights occur in opposite pairs.
Its coefficient of $h^{2j}$ has denominator at most
$(1-u)^{b+2j}$: each derivative increases the pole order by at most
one. To obtain the highest polar coefficient, write $t=1-u$ and
retain $1-ue^{ah}=t-ah+\text{terms of higher total degree}$ in
each factor. The coefficient is
$[x^{2j}]\prod_{a\in\mathcal W_n}(1-ax)^{-1}
 =h_{2j}(\mathcal W_n)$.

The weights pair as $a,-a$, with a possible zero. Thus
\[
 \sum_{j\ge0}h_{2j}(\mathcal W_n)x^{2j}
       =\prod_{a\in\mathcal W_n,\ a>0}(1-a^2x^2)^{-1}.
\]
For $n\ge3$ there is a positive weight, so every displayed
coefficient is positive. Finally $h=\log(1+\eta)$ has
$[\eta^{2j}]h^{2j}=1$ and
$[\eta^{2j+1}]h^{2j}=-j$. All earlier even powers of $h$ have
strictly lower pole order. This proves both formulas and exactness.
The coefficients of \eqref{eq:figure-eight-product} are integral
Taylor series in $u,\eta$; multiplying by the indicated denominator
makes them rational polynomials with integral coefficients.
\end{proof}

In particular,
\begin{equation}\label{eq:figure-eight-low-allrank}
 A_{4_1,n,2}(u)=\binom n3\frac{u}{(1-u)^{n+1}},\qquad
 A_{4_1,n,3}=-A_{4_1,n,2}.
\end{equation}
For $n=3$, all the even highest polar coefficients in
\eqref{eq:figure-eight-top-even} equal one. For $n=4$ they are
$4^j$, and for $n=5$ they are $(9^{j+1}-1)/8$.
The full coefficients $R_{4_1,n,m}$ are obtained by the moving-node
operators of \cref{prop:rank-triangular}. At second order,
\eqref{eq:rank-low-operators} gives the all-rank formula
\begin{equation}\label{eq:figure-eight-full-second}
 R_{4_1,n,2}(u)
 =-\frac{n(n-1)\bigl(6n+(14-n)u+3(n-2)u^2\bigr)}
          {6(1-u)^{n+2}}.
\end{equation}
Indeed, $R_2=A_2+T_2^{(n)}A_0$, and on any series $f$ one has
\[
 T_2^{(n)}f=-n^2f'-2(n+1)uf''-\frac43u^2f'''.
\]
Applying this identity to $A_0=(1-u)^{-(n-1)}$ proves
\eqref{eq:figure-eight-full-second}. Its pole at $u=1$ has exact
order $n+2$ for every $n\ge2$. In rank two, $A_{4_1,2,2}=0$
whereas $R_{4_1,2,2}=-4(1+u)/(1-u)^4$.

\subsection{Sharp Alexander bounds for the other three knots}
For $K=3_1,5_2,6_1$ put $\kappa_K=1,\,7/4,\,-3/4$, respectively.

For a polynomial $P(a)$ of degree at most $D$ and $\lambda\ne0$,
expansion in the binomial basis gives
\begin{equation}\label{eq:polynomial-generating-top}
 \lim_{v\to\lambda^{-1}}(1-\lambda v)^{D+1}
       \sum_{a\ge0}P(a)(\lambda v)^a
       =D!\,[a^D]P(a).
\end{equation}
Indeed,
$\sum_{a\ge0}\binom aj(\lambda v)^a
=(\lambda v)^j/(1-\lambda v)^{j+1}$,
and the coefficient of $\binom aD$ in $P$ is $D![a^D]P$.
A nonzero right side therefore gives the exact pole order $D+1$.

\begin{theorem}\label{thm:all-rank-sharp}
For every $n\ge2$ and $m\ge0$, let
$N_{K,n,m}=(1-c_Ku)^{n-1+2m}A_{K,n,m}$. Then
\begin{equation}\label{eq:all-rank-highest-pole}
 N_{K,n,m}(1/c_K)
  =\frac{(n+2m-2)!}{(n-2)!\,m!}\,\kappa_K^m\ne0.
\end{equation}
Thus the denominator exponent in \cref{thm:rank-Alexander} is exact
for each of these knots, at every rank and every quantum order.
\end{theorem}
\begin{proof}
Factoring the linear terms of the hyperbolic sines gives
\[
 D_{n,k}(e^h)=\binom{n+k-2}{n-2}
 \prod_{j=1}^{n-2}
 \frac{\sinh((k+j)h)/((k+j)h)}{\sinh(jh)/(jh)}.
\]
The coefficient of $h^a$ is polynomial in $k$ of degree at most
$n-2+a$. Substituting $h=\log(1+\eta)$ preserves this bound.
At order zero the polynomial is $\binom{n+k-2}{n-2}$, with leading
coefficient $1/(n-2)!$.

For the trefoil, the remaining monomial is
$(-1)^kq^{k^2+(2n-1)k}$. After removing the sign, its order-$m$
Taylor coefficient is
$\binom{k^2+(2n-1)k}{m}$, of degree $2m$ and leading coefficient
$1/m!$. In the product with $D_{n,k}$, a term using its order-$a$
coefficient has degree at most
\[
 (n-2+a)+2(m-a)=n-2+2m-a.
\]
Only $a=0$ contributes to the highest degree. Thus
$(-1)^k[\eta^m]c_{3_1,n,k}(1+\eta)$ has degree $n-2+2m$ and
leading coefficient $1/((n-2)!m!)$.

For the two double sums, use the subset formula
\[
 \qbin kj=\sum_{\substack{S\subset\{1,\ldots,k\}\\|S|=j}}
                 q^{2\sum_{i\in S}i-j(k+1)}.
\]
Let $X_1,\ldots,X_k$ be independent Bernoulli variables with
parameter $1/2$, and put $L=\sum_iX_i$ and $W=\sum_i iX_i$.
Including the outer monomials in \eqref{eq:four-seeds}, their
exponents become
\begin{align*}
 E_5&=3k^2+(4n-3)k+L^2+(1-2n-4k)L+2W,\\
 E_6&=-2k^2+(2-2n)k-L^2+(2n+2k-3)L+2W.
\end{align*}
Consequently
\[
 c_{5_2,n,k}=(-2)^kD_{n,k}\,\mathbb E(q^{E_5}),\qquad
 c_{6_1,n,k}=2^kD_{n,k}\,\mathbb E(q^{E_6}).
\]

For fixed nonnegative integers $a,b$, the moment
$\mathbb E(L^aW^b)$ is a polynomial in $k$ of degree at most
$a+2b$. To see this, expand the product and partition its $a+b$ indices
according to equality. A term with $j$ distinct indices has
expectation $2^{-j}$. Its index weight has total degree $b$.
Inclusion--exclusion reduces the sum over distinct indices to
products of power sums, of degree at most $b+j\le a+2b$.
Assigning weights one to $k,L$ and two to $W$, every term of
$E_5$ and $E_6$ has weight at most two. It follows that
$\mathbb E\binom{E_i}{m}$ is polynomial in $k$ of degree at most
$2m$.

To compute its leading coefficient, note that
\begin{align*}
 \mathbb E L&=\frac{k}{2},&
 \operatorname{Var}(L)&=\frac{k}{4},\\
 \mathbb E W&=\frac{k(k+1)}4,&
 \operatorname{Var}(W)&=\frac{k(k+1)(2k+1)}{24}.
\end{align*}
Thus $L/k\to1/2$ and $W/k^2\to1/4$ in mean square as $k\to\infty$.
Both normalized variables are uniformly bounded, so all their fixed
joint moments converge. Substitution in the displayed exponents gives
\[
 E_5/k^2\longrightarrow 7/4,\qquad
 E_6/k^2\longrightarrow-3/4
\]
in every fixed moment. Hence
\[
 \lim_{k\to\infty}k^{-2m}\mathbb E\binom{E_i}{m}
          =\frac{\kappa_K^m}{m!}.
\]
This is the nonzero leading coefficient of the polynomial just
constructed. Multiplying by the Taylor series of $D_{n,k}$ and
using the degree estimate above shows that
$c_K^{-k}[\eta^m]c_{K,n,k}(1+\eta)$ has degree $n-2+2m$ and
leading coefficient $\kappa_K^m/((n-2)!m!)$.

Apply \eqref{eq:polynomial-generating-top} with
$D=n-2+2m$ and $\lambda=c_K$. The leading coefficient just
computed gives \eqref{eq:all-rank-highest-pole}.
\end{proof}

\subsection{Laurent coefficient recurrences}
\label{sec:explicit-recurrences}
The trefoil and figure-eight coefficient sequences have invertible
Laurent lattices in every rank. The rank-two sequences for $5_2$ and
$6_1$ have monic Laurent recurrences whose determinants have nonzero
orbit sums.

Let $\mathsf S$ shift the Newton index and let
$\mathsf M c_k=q^kc_k$, so $\mathsf S\mathsf M=q\mathsf M\mathsf S$.
Define
\[
 p_n(T)=\prod_{a\in\mathcal W_n}(T-q^a),\qquad
 \mathsf D_n=-q^{-2n}\mathsf M^{-2}\mathsf S.
\]

\begin{proposition}\label{prop:two-Laurent-lattices}
The figure-eight and trefoil coefficients satisfy respectively
\begin{equation}\label{eq:two-Laurent-recurrences}
 p_n(\mathsf S)c_{4_1,n}=0,\qquad
 p_n(\mathsf D_n)c_{3_1,n}=0.
\end{equation}
Both relations give companion matrices in
$\GL_{n-1}(\mathbb Z[q^{\pm1},M^{\pm1}])$.
For the figure-eight the matrix is independent of $M$.
\end{proposition}
\begin{proof}
The rational generating function \eqref{eq:figure-eight-product}
gives the constant-coefficient recurrence $p_n(\mathsf S)D_{n}=0$.
Its leading coefficient is one and its constant coefficient is
$(-1)^{n-1}$, since $\sum_{a\in\mathcal W_n}a=0$.
For the trefoil write
$c_{3_1,n,k}=g_kD_{n,k}$, where
$g_k=(-1)^kq^{k^2+(2n-1)k}$. Then
$g_k/g_{k+1}=-q^{-2n}q^{-2k}$, so conjugating the figure-eight
operator by the sequence $g_k$ replaces $\mathsf S$ with
$\mathsf D_n$. The leading coefficient of
$p_n(\mathsf D_n)$ is a Laurent monomial and its constant
coefficient is $(-1)^{n-1}$. After normalization, both end
coefficients are Laurent units. The companion matrices and their
inverses therefore have Laurent entries.
\end{proof}

The states of these lattices are the vectors
$(c_k,c_{k+1},\ldots,c_{k+n-2})^{\mathsf T}$, and transport advances
the Newton index $k$.

\begin{proposition}\label{prop:example-period-Jordan}
At $q=\zeta$, put $\ell=\ord(\zeta^2)$,
$\sigma=\zeta^{n\ell}$, and $a_0=\lfloor(n-2)/\ell\rfloor$.
The period-$\ell$ matrix of the figure-eight coefficient recurrence
has only the eigenvalue $\sigma$; its largest Jordan block has
size $a_0+1$. The trefoil period matrix has only the eigenvalue
$-\sigma$ and the same Jordan block sizes.
\end{proposition}
\begin{proof}
The figure-eight companion matrix is cyclic with characteristic
polynomial $p_n$. Its eigenvalues after specialization are
$\zeta^a$, $a\in\mathcal W_n$, with one Jordan block for each
distinct eigenvalue. The largest multiplicity is $a_0+1$.
Each eigenvalue has $\ell$th power $\sigma$; taking this power
preserves Jordan block sizes in characteristic zero.
For a state consisting of $n-1$ consecutive coefficients, let
$G_k=\diag(g_k,\ldots,g_{k+n-2})$. At the root of unity,
\[
 \frac{g_{k+\ell}}{g_k}
 =(-1)^\ell(\zeta^\ell)^{2k+\ell+2n-1}=-1.
\]
Here either $\zeta$ has odd order $\ell$ and $\zeta^\ell=1$, or it
has order $2\ell$ and $\zeta^\ell=-1$. The trefoil period matrix is
therefore $-G_kC^\ell G_k^{-1}$, where $C$ is the figure-eight
companion matrix. This proves the result.
\end{proof}

For the other two knots, the rank-two coefficients have second-order
Laurent recurrences. Set $Q=q^2$, $\xi=Q^k$, and
$c_{K,k}=c_{K,2,k}$. The following operators are the $p=2$ and
$p=-2$ cases of the $C$-polynomials of Garoufalidis--Sun
\cite[Theorem~1 and Appendix~A]{GS06}.

\begin{proposition}\label{prop:twist-C-recurrences}
For $K=5_2,6_1$, one has
\begin{equation}\label{eq:twist-C-recurrence}
 c_{K,k+2}+a_{1,K}(Q^k,Q)c_{K,k+1}
              +a_{0,K}(Q^k,Q)c_{K,k}=0,
\end{equation}
where
\begin{align}
 a_{0,5_2}(\xi,Q)&=Q^6\xi^2-Q^7\xi^3,&
 a_{1,5_2}(\xi,Q)&=(Q^3+Q^4)\xi-Q^5\xi^2+Q^7\xi^3,
       \label{eq:twist-C5}\\
 a_{0,6_1}(\xi,Q)&=-Q^{-2}\xi^{-1}+Q^{-1},&
 a_{1,6_1}(\xi,Q)&=-Q^{-4}\xi^{-2}+Q^{-2}\xi^{-1}-Q^{-1}-1.
       \label{eq:twist-C6}
\end{align}
The initial coefficients are
$c_{5_2,0}=c_{6_1,0}=1$,
$c_{5_2,1}=-Q^4-Q^2$, and $c_{6_1,1}=1+Q^{-2}$.
Their companion matrices are
\begin{equation}\label{eq:twist-transfer}
 B_K(\xi,Q)=
   \begin{pmatrix}0&1\\-a_{0,K}(\xi,Q)&-a_{1,K}(\xi,Q)\end{pmatrix}.
\end{equation}
Neither matrix is rationally equivalent to an element of
$\GL_2(\mathbb C(Q)[\xi^{\pm1}])$ for transcendental $Q$.
\end{proposition}
\begin{proof}
Define
\[
 H_k(Q)=\sum_{j=0}^k Q^{j(j+1)}\gaussbin{k}{j}{Q},\quad
 T_k(Q)=\sum_{j=0}^k Q^{j(j+2)}\gaussbin{k}{j}{Q},\quad
 V_k(Q)=\sum_{j=0}^k Q^{j(j+3)}\gaussbin{k}{j}{Q}.
\]
The two Gaussian Pascal identities give
\[
 H_{k+1}=H_k+Q^{k+2}T_k,\qquad
 T_{k+1}=T_k+Q^{k+3}V_k,\qquad
 H_{k+1}=T_k+Q^2V_k.
\]
Eliminating $T_k,V_k$ yields
\begin{equation}\label{eq:twist-H-recurrence}
 H_{k+2}=(Q^{2k+4}-Q^{k+2}+Q+1)H_{k+1}
                         +(Q^{k+2}-Q)H_k.
\end{equation}
In \eqref{eq:four-seeds} with $n=2$, replace the symmetric
Gaussian coefficient by $q^{-j(k-j)}\gaussbin{k}{j}{Q}$ and
then replace $j$ by $k-j$. This gives
\[
 c_{5_2,k}=(-1)^kQ^{k(k+3)/2}H_k(Q),\qquad
 c_{6_1,k}=H_k(Q^{-1}).
\]
For the second identity also use
$\gaussbin{k}{j}{Q^{-1}}=Q^{-j(k-j)}\gaussbin{k}{j}{Q}$.
The first substitution conjugates \eqref{eq:twist-H-recurrence}
by $(-1)^kQ^{k(k+3)/2}$ and gives \eqref{eq:twist-C5}.
Replacing $Q$ by $Q^{-1}$ gives \eqref{eq:twist-C6}.
The initial values follow from $H_0=1$ and $H_1=1+Q^2$.

The determinants of \eqref{eq:twist-transfer} are
\[
 Q^6\xi^2(1-Q\xi),\qquad
 Q^{-1}\xi^{-1}(\xi-Q^{-1}).
\]
Each has total divisor multiplicity one on the $Q$-orbit of
$\xi=1$. \Cref{prop:orbit} gives the final assertion.
\end{proof}

\subsection{Exact specializations at every root of unity}\label{sec:explicit-root-values}
Fix $\zeta$ and write
\[
 \ell=\ord(\zeta^2),\quad \epsilon_0=\zeta^\ell\in\{1,-1\},
 \quad n-2=a_0\ell+r_0,\quad 0\le r_0<\ell,
 \quad \sigma=\zeta^{n\ell}.
\]
Put $W=W_{n,\zeta}$, as in \eqref{eq:rank-period-polynomial}.

\begin{theorem}\label{thm:four-all-roots}
For the four knots in \eqref{eq:four-seeds}, every $a\ge0$ and
$0\le b<\ell$ satisfy
\begin{equation}\label{eq:four-root-blocks}
 c_{K,n,\ell a+b}(\zeta)
  =c_{K,n,b}(\zeta)(c_K\sigma)^a\binom{a_0+a}{a_0}.
\end{equation}
Consequently, in the specialized Newton completion,
\begin{equation}\label{eq:four-root-functions}
 F_{K,n}(X;\zeta)
   =\frac{\displaystyle\sum_{b=0}^{\ell-1}
                 c_{K,n,b}(\zeta)B_b^{(n)}(X;\zeta)}
          {(1-c_K\sigma W(X))^{a_0+1}}.
\end{equation}
After $X=z+z^{-1}$ its displayed denominator is
\begin{equation}\label{eq:four-root-Alexander}
 \Delta_K(\sigma z^\ell)^{1+\lfloor(n-2)/\ell\rfloor}.
\end{equation}
\end{theorem}
\begin{proof}
The fixed-length product for $D_{n,k}$ gives
\[
 D_{n,\ell a+b}(\zeta)
   =\epsilon_0^{(n-2)a}D_{n,b}(\zeta)\binom{a_0+a}{a_0}.
\]
If $b+r_0\ge\ell$, the numerator has one more vanishing factor
than the denominator, and both sides are zero. Otherwise the
vanishing numerator factors have indices $\ell(a+j)$,
$1\le j\le a_0$. Paired with the denominator factors of indices
$\ell j$, they contribute the ratios $(a+j)/j$ after their
root-of-unity phases are removed. The product of these ratios is
$\binom{a_0+a}{a_0}$; all factors together contribute the phase
$\epsilon_0^{(n-2)a}=\sigma^a$.
For the trefoil, the additional signed monomial has block ratio
$(-1)^a$, by the phase calculation in
\cref{prop:example-period-Jordan}. This proves
\eqref{eq:four-root-blocks} for the first two knots.

For the double sums, write $k=\ell a+b$ and $j=\ell d+c$,
where $0\le c<\ell$. Put $\omega=\zeta^2$. The ordinary
$q$-Lucas identity \cite[Theorem~2.2]{Sagan92} is
\begin{equation}\label{eq:ordinary-q-Lucas}
 \gaussbin{\ell a+b}{\ell d+c}{\omega}
       =\binom ad\gaussbin{b}{c}{\omega}.
\end{equation}
To pass to the symmetric Gaussian polynomial, use
$\qbin kj=q^{-j(k-j)}\gaussbin{k}{j}{q^2}$. The ratio of the
balancing phases for $(k,j)$ and $(b,c)$ is
\[
 \zeta^{-j(k-j)+c(b-c)}
 =\epsilon_0^{-\ell d(a-d)-db-ca+2dc}
 =\epsilon_0^{-\ell d(a-d)-db-ca}.
\]
Thus
\begin{equation}\label{eq:symmetric-q-Lucas}
 \genfrac{[}{]}{0pt}{}{\ell a+b}{\ell d+c}_{\zeta}
 =\binom ad\genfrac{[}{]}{0pt}{}{b}{c}_{\zeta}
       \epsilon_0^{-\ell d(a-d)-db-ca}.
\end{equation}
Terms with $c>b$ vanish. It remains to combine the balancing
phase with the knot monomials. Relative to the prefix term with
indices $b,c$, multiplication by the inner monomial in
\eqref{eq:four-seeds} changes the total exponent of $\epsilon_0$
to, respectively,
\begin{gather*}
 2\bigl(-2ac-2\ell ad-2bd+cd+\ell d^2-nd+d\bigr),\\
 2\bigl(ac+\ell ad+bd-cd+nd-d\bigr).
\end{gather*}
Both exponents are even, so the residual phases are one. Summing
$\binom ad$ over $0\le d\le a$ gives $2^a$ times the prefix sum.
The outer signed monomials have block ratios $(-1)^a$ for $5_2$
and $1$ for $6_1$. Together with the factor $\sigma^a$ from
$D_{n,k}$, these give $(-2\sigma)^a$ and $(2\sigma)^a$.
This proves \eqref{eq:four-root-blocks} for all four knots.

Now apply \cref{prop:rank-full-period} and sum the blocks using
\[
 \sum_{a\ge0}\binom{a_0+a}{a_0}v^a=(1-v)^{-a_0-1}.
\]
The resulting identity is \eqref{eq:four-root-functions} in the
$W$-adic completion. Since $\delta_K(u)=1-c_Ku$ for these knots
and $\sigma^2=1$,
\[
 1-c_K\sigma W(X)
 =1-c_K\bigl(\sigma\Theta_\ell(X)-2\bigr)
 =\Delta_K(\sigma z^\ell).
\]
This proves \eqref{eq:four-root-Alexander}.
\end{proof}
In rank two, this Alexander denominator also appears in Willetts's
unified ADO--colored Jones invariant \cite{Willetts22};
\cref{thm:four-all-roots} gives in addition the numerator and the
dependence of the exponent on the rank.

\subsection{All quantum orders for the trefoil and figure-eight}\label{sec:all-root-quantum}
For the trefoil and figure-eight, the fixed-length product for
$D_{n,k}$ bounds every quantum coefficient at a root of unity. The
quadratic monomial in the trefoil formula accounts for its larger
pole orders.
For a knot in \eqref{eq:four-seeds}, define
\[
 G_{K,n,b,m}(v)=\sum_{a\ge0}[\epsilon^m]
       c_{K,n,\ell a+b}(\zeta(1+\epsilon))v^a.
\]

\begin{theorem}\label{thm:trefoil-root-jets}
For the trefoil put $\lambda=-\sigma$. Then
\begin{equation}\label{eq:trefoil-root-jet-bound}
 (1-\lambda v)^{a_0+2m+1}G_{3_1,n,b,m}(v)
       \in\mathbb Z[\zeta][v].
\end{equation}
If $b+r_0<\ell$, the pole has exact order $a_0+2m+1$ and
\begin{equation}\label{eq:trefoil-root-highest}
 \lim_{v\to\lambda^{-1}}(1-\lambda v)^{a_0+2m+1}
       G_{3_1,n,b,m}(v)
 =c_{3_1,n,b}(\zeta)
       \frac{(a_0+2m)!\,\ell^{2m}}{a_0!\,m!}\ne0.
\end{equation}
For the figure-eight,
\begin{equation}\label{eq:figure-root-jet-bound}
 (1-\sigma v)^{a_0+m+1}G_{4_1,n,b,m}(v)
       \in\mathbb Z[\zeta][v].
\end{equation}
The full Taylor coefficients of the trefoil have denominator dividing
$(1+\sigma W)^{a_0+2m+1}$. For the figure-eight, at every $m\ge1$
the full coefficient has denominator dividing
$(1-\sigma W)^{a_0+2m}$.
\end{theorem}
\begin{proof}
In the fixed-length product
$D_{n,k}=\prod_{j=1}^{n-2}\{k+j\}/\{j\}$, the denominator
has exactly $a_0$ zero factors at $q=\zeta$. Remove their factor
$\epsilon^{a_0}$. The remaining denominator is a unit of
$\mathbb Q(\zeta)[[\epsilon]]$, independent of $a$. After its
root-of-unity phase is removed, the order-$t$ coefficient of the
numerator is polynomial in $a$ of degree at most $t$, because each
exponent is linear in $a$.
Thus the order-$j$ coefficient of $D_{n,\ell a+b}$, after the phase
$\sigma^a$ is removed, has degree at most $a_0+j$.

The trefoil monomial has exponent
$(\ell a+b)^2+(2n-1)(\ell a+b)$. Its order-$(m-j)$ coefficient
has degree at most $2(m-j)$, with leading coefficient
$\ell^{2(m-j)}/(m-j)!$. The total degree after removal of
$\lambda^a$ is at most $a_0+2m$. When $b+r_0<\ell$,
\eqref{eq:four-root-blocks} supplies the degree-$a_0$ constant jet,
and only $j=0$ contributes to the maximum degree. Its leading
coefficient is $c_{3_1,n,b}(\zeta)\ell^{2m}/(a_0!m!)$.
Formula \eqref{eq:polynomial-generating-top}, with
$D=a_0+2m$, gives both \eqref{eq:trefoil-root-jet-bound} and
\eqref{eq:trefoil-root-highest}. For the figure-eight there is no
quadratic monomial, so the degree is at most $a_0+m$.
The original block coefficients are integral, and multiplication by
the displayed denominator, which has constant coefficient one,
makes the rational numerator integral as well.

Apply \cref{thm:rank-full-series} with $\beta=a_0+1$ to obtain
the full trefoil estimate. For the figure-eight, an order-$h\ge1$
kernel correction to the order-$(m-h)$ coefficient has exponent at
most $a_0+(m-h)+1+(2h-1)\le a_0+2m$.
The $h=0$ term has exponent $a_0+m+1\le a_0+2m$ when $m\ge1$.
The same removal of apparent poles and integral numerator argument
as in \cref{thm:rank-full-series} proves the final assertion.
\end{proof}

In particular, the trefoil and figure-eight knot elements belong to
$\C_{n,\Delta_K}^{\mathrm{rat}}$ for every $n$, so their local
expansions satisfy \cref{thm:rank-recentering} and
\eqref{eq:rank-finite-recentering}.

\subsection{Normal coefficients at fourth and third roots of unity}
\label{sec:explicit-normal}
At an earlier boundary, the denominator in
\eqref{eq:four-root-functions} has constant coefficient one,
since $W_{n,\zeta}(a_{n-2s}(\zeta))=0$. Thus
\cref{prop:rational-normal-recurrence} computes every normal
coefficient from the finite numerator.

\begin{theorem}\label{thm:four-at-i}
For $K=3_1,4_1,5_2,6_1$, let $c=c_K$. At $q=i$ one has
\begin{align}
 F_{K,2d+1}(X;i)&=(1+cX^2)^{-d}\qquad(d\ge1),
       \label{eq:four-odd-at-i}\\
 F_{K,2d}(X;i)&=
       \frac{1+2c+(-1)^{d-1}cX}
            {(1-c(X^2-4))^d}\qquad(d\ge1).
       \label{eq:four-even-at-i}
\end{align}
For odd rank, every earlier boundary is $X=0$ and
\begin{equation}\label{eq:four-odd-normal-at-i}
 \tau_{K,2j}^{(2d+1,s)}(i)=(-c)^j\binom{d+j-1}{j},\qquad
 \tau_{K,2j+1}^{(2d+1,s)}(i)=0.
\end{equation}
For even rank, put $\epsilon_s=(-1)^{d-s}$ and
$\alpha_d=(-1)^{d-1}$. At $X=2\epsilon_s+u$,
\begin{equation}\label{eq:four-even-normal-at-i}
 \sum_{j\ge0}\tau_{K,j}^{(2d,s)}(i)u^j
 =\frac{1+2c+2c\alpha_d\epsilon_s+c\alpha_du}
        {(1-c(4\epsilon_su+u^2))^d}.
\end{equation}
At every earlier boundary the minimal order of an eventual
homogeneous constant-coefficient recurrence for the normal
sequence over $\mathbb Q(i)$ is $2\lfloor n/2\rfloor$.
\end{theorem}
\begin{proof}
Here $\ell=2$ and $a_0+1=\lfloor n/2\rfloor$.
For $n=2d+1$ one has $\sigma=-1$, $W=X^2$, and
$c_{K,n,1}(i)=0$. For $n=2d$ one has $\sigma=1$, $W=X^2-4$,
$c_{K,n,1}(i)=\alpha_dc$, and
$B_1^{(n)}(X;i)=X+2\alpha_d$.
These coefficient evaluations follow directly from
\eqref{eq:four-seeds}. Substitution in
\eqref{eq:four-root-functions} proves
\eqref{eq:four-odd-at-i} and \eqref{eq:four-even-at-i}.
The odd-rank boundaries specialize to zero; the even-rank boundary
$a_{2d-2s}(i)$ is $2\epsilon_s$. The binomial expansion and this
translation give \eqref{eq:four-odd-normal-at-i} and
\eqref{eq:four-even-normal-at-i}.

For odd rank the numerator is one, so the reduced denominator
has degree $2d$. For even rank the numerator has one root,
$X=-(1+2c)/(\alpha_dc)$. At this point
\[
 1-c(X^2-4)=-\frac{3c+1}{c}\ne0
                \qquad(c=-1,1,-2,2).
\]
The numerator is therefore coprime to the denominator, of degree
$2d$. Translation to a boundary preserves coprimality. Apply
\cref{prop:rational-normal-recurrence} to obtain the stated
minimal orders.
\end{proof}

In rank two, abbreviate $\tau_j=\tau_{K,j}^{(2,1)}(i)$. Then
\begin{equation}\label{eq:rank-two-normal-at-i}
 \sum_{j\ge0}\tau_ju^j=\frac{1+4c+cu}{1-4cu-cu^2},\qquad
 \tau_0=1+4c,\qquad \tau_1=5c+16c^2,
\end{equation}
and
\begin{equation}\label{eq:rank-two-normal-recurrence}
 \tau_j=4c\tau_{j-1}+c\tau_{j-2}\qquad(j\ge2).
\end{equation}
The initial terms are
\begin{center}
\begin{tabular}{crrrrr}
\toprule
$K$&$\tau_0$&$\tau_1$&$\tau_2$&$\tau_3$&$\tau_4$\\
\midrule
$3_1$&$-3$&$11$&$-41$&$153$&$-571$\\
$4_1$&$5$&$21$&$89$&$377$&$1597$\\
$5_2$&$-7$&$54$&$-418$&$3236$&$-25052$\\
$6_1$&$9$&$74$&$610$&$5028$&$41444$\\
\bottomrule
\end{tabular}
\end{center}
For each column, \eqref{eq:finite-normal-value} gives the same
coefficient using only the indices $k\le2j+1$.

In rank three the boundary is $X=0$ and
$F_{K,3}(u;i)=(1+cu^2)^{-1}$. Consequently
\begin{center}
\begin{tabular}{ccrrrr}
\toprule
$K$&$F_{K,3}(u;i)$&$\tau_0$&$\tau_1$&$\tau_2$&$\tau_4$\\
\midrule
$3_1$&$(1-u^2)^{-1}$&$1$&$0$&$1$&$1$\\
$4_1$&$(1+u^2)^{-1}$&$1$&$0$&$-1$&$1$\\
$5_2$&$(1-2u^2)^{-1}$&$1$&$0$&$2$&$4$\\
$6_1$&$(1+2u^2)^{-1}$&$1$&$0$&$-2$&$4$\\
\bottomrule
\end{tabular}
\end{center}
The boundary values and first normal coefficients agree, whereas
normal order two distinguishes the four specialized knot elements.
Here $N_j=j$, so \eqref{eq:finite-normal-value} computes each
column from the first $j+1$ Newton coefficients.

For a primitive cube root $\omega$, with
$\omega^2+\omega+1=0$, take $n=2$ and $X=2+u$. The complete
period is
\[
 W_{2,\omega}(2+u)=9u+6u^2+u^3.
\]
\Cref{thm:four-all-roots} gives
\begin{equation}\label{eq:four-normal-at-omega}
 F_{K,2}(2+u;\omega)=
       \frac{P_K(u)}{1-c_K(9u+6u^2+u^3)},
\end{equation}
where
\begin{center}
\begin{tabular}{cl}
\toprule
$K$&$P_K(u)$\\
\midrule
$3_1$&$1+6\omega+5\omega u+\omega u^2$\\
$4_1$&$13+7u+u^2$\\
$5_2$&$-5+18\omega+(-5+12\omega)u+(-1+2\omega)u^2$\\
$6_1$&$19-12\omega+(12-7\omega)u+(2-\omega)u^2$\\
\bottomrule
\end{tabular}
\end{center}
Indeed, the prefix kernels are $1$, $u+3$, and $(u+3)^2$.
Direct substitution in \eqref{eq:four-seeds} gives respectively
\[
 (c_{K,2,1}(\omega),c_{K,2,2}(\omega))
 =(-\omega,\omega),\ (1,1),\
   (1,-1+2\omega),\ (-\omega,2-\omega),
\]
which proves the table by expanding the three prefix terms.
Writing $\tau_j=\tau_{K,j}^{(2,1)}(\omega)$, one obtains
\begin{equation}\label{eq:four-normal-omega-recurrence}
 \tau_j=c_K(9\tau_{j-1}+6\tau_{j-2}+\tau_{j-3})
                       \qquad(j\ge3).
\end{equation}
The initial coefficients are
\begin{center}
\begin{tabular}{clll}
\toprule
$K$&$\tau_0$&$\tau_1$&$\tau_2$\\
\midrule
$3_1$&$1+6\omega$&$-9-49\omega$&$75+406\omega$\\
$4_1$&$13$&$124$&$1195$\\
$5_2$&$-5+18\omega$&$85-312\omega$&$-1471+5402\omega$\\
$6_1$&$19-12\omega$&$354-223\omega$&$6602-4159\omega$\\
\bottomrule
\end{tabular}
\end{center}
For example, the next figure-eight coefficient is $11512$.
The finite formula uses $k\le3j+2$ for the $j$th normal
coefficient at this root.

The oriented expansions follow from
\eqref{eq:oriented-normal-substitution}. In rank two,
$z_0=1$ and $u=t^2/(1+t)$; for $5_2$ at $q=i$,
\[
 F_{5_2,2}((1+t)+(1+t)^{-1};i)
       =-7+54t^2-54t^3-364t^4+O(t^5).
\]
In rank three, $z_0=i$ and $u=it(2+t)/(1+t)$, giving
\[
 F_{K,3}(i(1+t)-i(1+t)^{-1};i)
       =1+4c_Kt^2-4c_Kt^3+(5c_K+16c_K^2)t^4+O(t^5).
\]

\subsection{Spectral line modules for \texorpdfstring{$5_2$ and $6_1$}{5\_2 and 6\_1}}
\label{sec:twist-spectral}
Keep the coefficient-index variable $\xi$ independent in the
monic recurrences \eqref{eq:twist-transfer}. At $q=i$ one has
$Q=-1$, and the one-step matrices are
\begin{equation}\label{eq:twist-at-i}
 B_{5_2}(\xi)=
   \begin{pmatrix}0&1\\-\xi^2(1+\xi)&\xi^2(\xi-1)\end{pmatrix},
 \qquad
 B_{6_1}(\xi)=
   \begin{pmatrix}0&1\\1+\xi^{-1}&(1-\xi)\xi^{-2}\end{pmatrix}.
\end{equation}
The shift sends $\xi$ to $-\xi$, and
$M_K(\xi)=B_K(-\xi)B_K(\xi)$ is the two-step matrix.

\begin{proposition}\label{prop:twist-periods}
The period matrices are
\begin{align}
 M_{5_2}(\xi)&=
 \begin{pmatrix}
 -\xi^2(\xi+1)&\xi^2(\xi-1)\\
 \xi^4(\xi+1)^2&\xi^2(-\xi^4+\xi^2+\xi-1)
 \end{pmatrix},\label{eq:twist-M5}\\
 M_{6_1}(\xi)&=
 \begin{pmatrix}
 (\xi+1)/\xi&(1-\xi)/\xi^2\\
 (\xi+1)^2/\xi^3&(\xi^4-\xi^3-\xi^2+1)/\xi^4
 \end{pmatrix}.\label{eq:twist-M6}
\end{align}
Their characteristic polynomials are
\begin{align}
 p_{5_2}(L,\xi)&=L^2+\xi^2(\xi^4-\xi^2+2)L
                           +\xi^4(1-\xi^2),\label{eq:twist-p5}\\
 p_{6_1}(L,\xi)&=L^2-(2-\xi^{-2}+\xi^{-4})L
                           +1-\xi^{-2}.\label{eq:twist-p6}
\end{align}
If
$C_K(L,\xi;1)=L^2+a_{1,K}(\xi,1)L+a_{0,K}(\xi,1)$,
then
\begin{equation}\label{eq:twist-period-classical}
 p_K(L,\xi)=C_K(L,\xi^2;1),\qquad K=5_2,6_1.
\end{equation}
\end{proposition}
\begin{proof}
Multiply the two matrices in \eqref{eq:twist-at-i}. Their traces
and determinants give \eqref{eq:twist-p5}--\eqref{eq:twist-p6}.
Setting $Q=1$ in \eqref{eq:twist-C5}--\eqref{eq:twist-C6}
and substituting $\xi^2$ gives the same polynomials.
\end{proof}

Put $A_0=\mathbb C[\xi^{\pm1}]$ and $S_K=A_0[L]/(p_K)$.
We use the consecutive-coefficient lattice of the monic recurrence
\eqref{eq:twist-C-recurrence}: its period module is
$\mathscr P_K=A_0^2$, with $L$ acting by $M_K$.

\begin{theorem}\label{thm:elliptic-spectral}
Let $E$ be the smooth projective elliptic curve
\begin{equation}\label{eq:twist-elliptic-curve}
 y^2=x^3+x^2-x
\end{equation}
with origin $O$ at infinity, and put $E^\circ=E\setminus E[2]$.
Both $\Spec S_{5_2}$ and $\Spec S_{6_1}$ are isomorphic to
$E^\circ$. Under the coordinate maps below,
\begin{equation}\label{eq:twist-spectral-points}
 \mathscr P_{5_2}\simeq\mathcal O_{E^\circ}(P),\quad P=(-1,-1),
 \qquad
 \mathscr P_{6_1}\simeq\mathcal O_{E^\circ}(R),\quad R=(1,1).
\end{equation}
Their classes have exact order three in $\Pic(E^\circ)$ and
satisfy $[\mathscr P_{6_1}]=-[\mathscr P_{5_2}]$ under these
identifications.
\end{theorem}
\begin{proof}
For $S_{5_2}$ define
\begin{equation}\label{eq:twist-elliptic-map5}
 x=L\xi^{-4}+\xi^2-1+\xi^{-2},\qquad y=\xi x.
\end{equation}
The relation $p_{5_2}=0$ gives
\[
 x^{-1}=L\xi^{-4}+\xi^{-2},\qquad
 \xi^2=x+1-x^{-1}.
\]
Thus $x,y$ are units and satisfy \eqref{eq:twist-elliptic-curve}.
The inverse map is
\[
 \xi=y/x,\qquad L=\xi^2(x-1)/x^2.
\]
For $S_{6_1}$ define instead
\begin{equation}\label{eq:twist-elliptic-map6}
 x=\xi^2(L-1),\qquad y=\xi(L-1).
\end{equation}
The relation $p_{6_1}=0$ gives
$x^{-1}=x+1-\xi^{-2}$, so $x,y$ are again units and
$\xi^{-2}=x+1-x^{-1}$. The inverse map is
\[
 \xi=x/y,\qquad L=x(x+1).
\]
In both cases these formulas are inverse ring maps to
$\mathbb C[x,y,(xy)^{-1}]/(y^2-x^3-x^2+x)$.
The omitted points are $O$ and the three affine points with
$y=0$, exactly $E[2]$. The cubic has discriminant $5$, so
$E$ is smooth and each $S_K$ is a Dedekind domain.

The characteristic polynomials are irreducible over
$\mathbb C(\xi)$, since these coordinate rings are domains.
Hence each $\mathscr P_K$ is an invertible $S_K$-module as in
\cref{sec:spectral-comparison}. Consider the section
$e_2=(0,1)^{\mathsf T}$ and the map
\[
 S_K\longrightarrow\mathscr P_K,\qquad
 f\longmapsto f(M_K)e_2.
\]
Its matrix in the $A_0$-basis $1,L$ is $[e_2,M_Ke_2]$,
with determinant $-(M_K)_{12}$.
In \eqref{eq:twist-M5} and \eqref{eq:twist-M6}, this determinant
is a Laurent unit times $\xi-1$. The cokernel therefore has
length one over $A_0$, supported at $\xi=1$. Modulo the image of
$e_2$, multiplication by $L$ is multiplication by $(M_K)_{11}$.
At $\xi=1$ this is $-2$ for $5_2$ and $2$ for $6_1$.
Thus the section has a single simple zero at
$(\xi,L)=(1,-2)$ or $(1,2)$, respectively.
The maps \eqref{eq:twist-elliptic-map5} and
\eqref{eq:twist-elliptic-map6} send these points to $P$ and $R$.
The divisor of the section gives \eqref{eq:twist-spectral-points}.

The chord-and-tangent law on \eqref{eq:twist-elliptic-curve}
\cite[Chapter~III, Sections~2--3]{Silverman09} gives
\begin{equation}\label{eq:twist-elliptic-group-law}
 2P=R,\qquad 3P=(0,0),\qquad 3R=O.
\end{equation}
The tangent slope at $P$ is zero, so $2P=(1,1)$.
The line through $P$ and $R$ has slope one and its third
intersection is $(0,0)$, a point of order two. It follows that
$P$ has order six and $R$ has order three.

The divisor localization sequence for $E^\circ\subset E$
\cite[Chapter~II, Section~6]{Hartshorne77} gives
\[
 \Pic(E^\circ)=\Pic(E)/\langle[T]:T\in E[2]\rangle
                \simeq E(\mathbb C)/E[2].
\]
Here $P\mapsto\mathcal O_E(P-O)$ identifies $E(\mathbb C)$
with $\Pic^0(E)$ \cite[Chapter~III, Proposition~3.4]{Silverman09}.
The point $O$ splits the degree map, and the subgroup of removed
classes is
$\mathbb Z[O]\oplus E[2]$.
The images of $P$ and $R$ have exact order three in this quotient.
Finally $R=2P\equiv-P\pmod{E[2]}$, proving the relation between
the two spectral classes.
\end{proof}

The third powers have explicit principalization functions. Put
\begin{equation}\label{eq:twist-principalization-functions}
 f_P=x+2-y/x,\qquad f_R=y-2x+1.
\end{equation}
Using \eqref{eq:twist-elliptic-curve}, one obtains
\[
 f_P(x,y)f_P(x,-y)=\frac{(x+1)^3}{x},\qquad
 f_R(x,y)f_R(x,-y)=-(x-1)^3.
\]
At $P$ and $R$, respectively, the conjugate factor is nonzero,
so these zeros have order three. The poles of $f_P$ have orders
one at $(0,0)$ and two at $O$; $f_R$ has a pole of order three
at $O$. Thus
\begin{equation}\label{eq:twist-principalization-divisors}
 \operatorname{div}(f_P)=3P-(0,0)-2O,\qquad
 \operatorname{div}(f_R)=3R-3O.
\end{equation}
If $\mathfrak p_P,\mathfrak p_R$ are the corresponding point
ideals on $E^\circ$, then
$\mathfrak p_P^3=(f_P)$ and $\mathfrak p_R^3=(f_R)$.

\begin{corollary}\label{cor:threefold-conjugacy}
For $K=5_2,6_1$ and every positive integer $m$,
\begin{equation}\label{eq:twist-threefold-conjugacy}
 M_K^{\oplus m}\sim_{\GL_{2m}(A_0)}\Comp(p_K)^{\oplus m}
                         \quad\Longleftrightarrow\quad 3\mid m.
\end{equation}
A connected finite \'etale cover of $\Spec S_K$ of degree three
trivializes $\mathscr P_K$, and the degree of every finite locally
free trivializing cover is divisible by three.
\end{corollary}
\begin{proof}
The companion module is $S_K$. Apply
\cref{thm:stabilization} to the exact order computed in
\cref{thm:elliptic-spectral}.
\end{proof}

Over $A_0[(\xi-1)^{-1}]$ the matrix $[e_2,M_Ke_2]$ is invertible,
since its determinant is a Laurent unit times $\xi-1$, and it
conjugates $M_K$ to $\Comp(p_K)$. Thus the obstruction is supported
at $\xi=1$.

\section*{Acknowledgements}
We would like to thank Zhou, Chen and some other anonymous experts for their valuable
support and helpful contributions to this project. The author was
supported by the National Natural Science Foundation of China (NSFC) under
Grant No.~12341105.

\section*{Declaration of AI Use}
Claude Opus 5.5 was used solely for language editing and polishing of this manuscript. The author independently directed and carried out all aspects of the research, reviewed the final manuscript, and takes full responsibility for its content.

\end{document}